\documentclass[11pt]{amsart}
\usepackage{geometry}               
\usepackage{graphicx}
\usepackage{tikz}
\usepackage{amssymb}
\usepackage{pdfsync}
\usepackage{hyperref}
\usepackage{verbatim} 
\usepackage{epstopdf}
\def\R{\mathbb{R}}

\def\N{\mathbb{N}}
\def\Z{\mathbb{Z}}
\def\Co{\mathbb{C}}

\newtheorem{theorem}{Theorem}[section]

\newtheorem{remark}[theorem]{Remark}
\newtheorem{proposition}[theorem]{Proposition}
\newtheorem{lemma}[theorem]{Lemma}

\numberwithin{equation}{section}

\title{Sharp extension inequalities on finite fields}

\author[C. González-Riquelme]{Cristian González-Riquelme}
\email{cristian.g.riquelme@tecnico.ulisboa.pt}

\author[D. Oliveira e Silva]{Diogo Oliveira e Silva}
\email{diogo.oliveira.e.silva@tecnico.ulisboa.pt}\address{ 
Center for Mathematical Analysis, Geometry and Dynamical Systems \&
Departamento de Matem\'{a}tica\\ 
Instituto Superior T\'{e}cnico\\
Av. Rovisco Pais\\ 
1049-001 Lisboa, Portugal}

\begin{document}

\begin{abstract}
Sharp restriction theory and the finite field extension problem  have both received a great deal of attention in the last two decades, but so far they have not intersected. 
In this paper, we initiate the study of sharp restriction theory on finite fields. We prove that constant functions maximize the Fourier extension inequality from the parabola $\mathbb{P}^1\subset \mathbb{F}^{2\ast}_q$ and the paraboloid $\mathbb{P}^2\subset \mathbb{F}_q^{3\ast}$ at the euclidean Stein--Tomas endpoint; here, $\mathbb{F}_q^{d\ast}$ denotes the (dual) $d$-dimensional vector space over the finite field $\mathbb F_q$ with $q=p^n$ elements, where $p$ is a prime number greater than $3$ or $2$, respectively. We fully characterize the maximizers for the $L^2\to L^4$ extension inequality from $\mathbb{P}^2$ whenever $q\equiv 1(\textup{mod}\, 4)$.
Our methods lead to analogous results on the hyperbolic paraboloid, whose corresponding euclidean problem remains open.
We further establish that constants maximize the $L^2\to L^4$  extension inequality from the cone $\Gamma^3:=\{(\boldsymbol{\xi},\tau, \sigma)\in \mathbb{F}^{4\ast}_q: \tau\sigma=\boldsymbol{\xi}^2\}\setminus \{{\bf 0}\}$ whenever $q\equiv 3(\textup{mod}\, 4)$. By contrast, we prove that constant functions fail to be critical points for the corresponding inequality on $\Gamma^3\cup \{{\bf 0}\}$ over $\mathbb{F}_p^4$. 
While some inspiration is drawn from the euclidean setting,  entirely new phenomena emerge which are related to the underlying arithmetic and discrete structures.
\end{abstract}

\maketitle

\section{Introduction}
Stein's restriction problem in euclidean space \cite{St93} concerns  the possibility of restricting the Fourier transform of certain sufficiently nice functions  to {\it curved} null subsets of $\R^{d+1}$ like the paraboloid $\{(\boldsymbol{\xi},\tau)\in\R^d\times\R: \tau=|\boldsymbol{\xi}|^2\}$ and the cone $\{(\boldsymbol{\xi},\tau)\in\R^d\times\R: \tau^2=|\boldsymbol{\xi}|^2\}$.
The dual problem is formulated in terms of the adjoint restriction, or {\it extension}, operator. In turn, extension estimates from the paraboloid and the cone
are equivalent to the well-known Strichartz inequalities \cite{Str77} for the Schrödinger equation $iu_t=\Delta u$ with initial datum $u(0,\cdot)=f$,
\begin{equation}\label{eq_StrSchr}
     \|u\|_{L^{2+4/d}(\R^{d+1})} \leq {\bf S}_d \|f\|_{L^2(\R^d)},
\end{equation}
and for the  wave equation $u_{tt}=\Delta u$ with initial data $(u(0,\cdot),u_t(0,\cdot))=(f,g)$,
\begin{equation}\label{eq_StrWave}
\|u\|_{L^{2+4/(d-1)}(\R^{d+1})}\leq {\bf W}_d \|(f,g)\|_{\dot H^{1/2}(\R^d)\times\dot H^{-1/2}(\R^d)}.
\end{equation}
The restriction/extension problem has attracted widespread attention due to its deep links to problems in harmonic analysis (Bochner--Riesz), partial differential equations (local smoothing), geometric measure theory (Kakeya) and number theory (decoupling).
Conversely, progress in restriction theory has emerged via powerful tools from different fields, e.g., the multilinear approach which pivots on  higher notions of curvature and transversality \cite{BCT06, BG11} and the algebro-geometric approach known as the {\it polynomial method} \cite{Gu16,Gu18}.\\

In 2002, Mockenhaupt--Tao \cite{MT04} inaugurated the  study of the extension\footnote{The Kakeya problem over finite fields had been previously introduced by Wolff \cite{Wo99} in 1999, and was famously solved by Dvir \cite{Dv09} in 2008 using the polynomial method; see also \cite{EOT10}.} phenomenon for finite fields.
Given exponents $1\leq p,q\leq\infty$, let ${\bf R}^\ast_S(p\to q)$ denote the best constant for which the extension inequality
\begin{equation}\label{eq_FFExt}
    \|(f\sigma)^\vee\|_{L^q(\mathbb F^d,\textup d \boldsymbol{x})} \leq {\bf R}^\ast_S(p\to q) \|f\|_{L^p(S,\textup d \sigma)}
\end{equation}
holds for all functions $f:S\to\Co$ on a given nonempty set of frequencies $S\subset\mathbb F^{d\ast}$; here, $\mathbb F^{d\ast}$ is dual to the $d$-dimensional vector space $\mathbb F^d$ over the finite field $\mathbb F$, and $S$ is then called a ``surface''. In \eqref{eq_FFExt}, $(f\sigma)^\vee$ denotes the extension operator acting on $f$, $\textup d \boldsymbol{x}$ is the usual counting measure on $\mathbb F^d$ and $\textup d\sigma$ is the normalized surface measure on $S$; these concepts are defined in \S \ref{FAFF_sec} below.
The restriction problem for $S$ asks for the set of exponents $p,q$ such that ${\bf R}^\ast_S(p\to q)\leq C_{p,q}$, where $C_{p,q}$ does {\it not} depend on the underlying field $\mathbb F$.
In this case, we say that the restriction property ${\bf R}^\ast_S(p\to q)$ holds.
The authors of \cite{MT04} considered two particular surfaces: the {\it paraboloid} (we abbreviate $\boldsymbol{\xi}^2:=\boldsymbol{\xi}\cdot\boldsymbol{\xi}$)
\begin{equation}\label{eq_paraboloid}
    \mathbb P^{d}:=\{(\boldsymbol{\xi},\tau)\in \mathbb F^{d\ast}\times\mathbb F^\ast: \tau=\boldsymbol{\xi}^2\}  
\end{equation} 
when $d\in\{1,2\}$, and the {\it cone}
\begin{equation}\label{eq_coneF3}
    \Gamma^d:=\{(\boldsymbol{\xi},\tau,\sigma)\in \mathbb F^{(d-1)\ast}\times\mathbb F^\ast\times\mathbb F^\ast: \tau\sigma=\boldsymbol{\xi}^2\}\setminus \{{\bf 0}\},
\end{equation}   
when $d=2,$
both equipped with their normalized surface measures. They showed that the restriction property ${\bf R}^\ast_S(2\to 4)$ holds when 
 $S\in\{\mathbb P^1,\mathbb P^2,\Gamma^2\}$. 
Many authors followed, exploiting methods from additive combinatorics and incidence geometry, among others.   We  remark that there is no particularly natural notion of curvature in finite fields; to some extent, it may be replaced by the maximal size of affine subspaces, or the Witt index for a quadratic surface; see \cite{Le19} for details and a careful account of the state of the art concerning the finite field extension problem.
Other discrete models for the extension and the Kakeya phenomena have been proposed in \cite{HW18}; very recently, the Kakeya conjecture has been  established over rings of integers modulo $N$ \cite{Dh24, DD21}, the $p$-adics \cite{Ar24} and local fields of positive characteristic \cite{Sa23}.
\\

Sharp restriction theory aims at discovering the best constants and maximizers for extension-type inequalities like \eqref{eq_StrSchr} and \eqref{eq_StrWave}.
In 2006, Foschi \cite{Fo07} proved that gaussians maximize the cases $d\in\{1,2\}$ of \eqref{eq_StrSchr}, which respectively correspond to $L^2\to L^6$ and $L^2\to L^4$ extension inequalities.
In fact, all maximizers are given by initial data corresponding to the orbit of the Schrödinger propagator of the standard gaussian $\exp(-|\cdot|^2)$ under the galilean group of symmetries. In the same paper, Foschi proved that \eqref{eq_StrWave} is saturated by the pair $((1+|\cdot|^2)^{-1},0)$ if $d=3$, which corresponds to an $L^2\to L^4$ extension inequality. All maximizers are then obtained by letting the Poincaré group act on the wave propagator of $((1+|\cdot|^2)^{-1},0)$. In particular, the best constants ${\bf S}_1, {\bf S}_2, {\bf W}_3$ are  known.
Alternative proofs of some of these facts rely on 
representation formulae \cite{Ca09, HZ06}, 
heat flow monotonicity \cite{BBCH09} and
orthogonal polynomials \cite{Go19}, but they all ultimately hinge on the Lebesgue exponents in question being even integers. In this case, one can invoke Plancherel's identity in order to reduce the problem to a (simpler) multilinear convolution estimate.
Sharp restriction theory flourished, with many interesting works on submanifolds of codimension $1$ or $d-1$; see the recent survey \cite{NOST23} for the state of the art.\\

In the finite setting of Mockenhaupt--Tao \cite{MT04}, no extension inequality is yet known in sharp form.\footnote{ Interestingly, the proof of \cite[Theorem 3]{Le19} is based on a careful analysis of the near-maximizers for the Stein--Tomas inequality, but the goal there is to analyze concentration effects; see also \cite[\S17.5--6]{Le19}.}
In the present paper, we inaugurate the study of sharp restriction theory on finite fields, and proceed to describe our main results on the paraboloids $\mathbb P^1, \mathbb P^2$ and the cone $\Gamma^3$.
We are also able to handle the hyperbolic paraboloid $\mathbb H^2$ defined on \eqref{eq_defHP} below, and present some further results on the cone $\Upsilon^3$ defined on \eqref{eq_defCone2} below.\\

Let $p$ be an odd prime number and $q=p^n$ for some $n\in\mathbb N$. Let $\mathbb F_q$ denote the finite field with $q$ elements.

\subsection{Sharp parabolic extension}
Our first result establishes the sharp $L^2\to L^4$ extension inequality from the paraboloid $\mathbb P^2\subset \mathbb F_q^{3\ast}$ equipped with the normalized surface measure $\sigma= \sigma_{\mathbb P^2}$.

\begin{theorem}\label{thm1}
 It holds that
${\bf R}_{\mathbb P^2}^\ast(2\to 4)=(1+q^{-1}-q^{-2})^{\frac 14}$.     
 In other words, the inequality 
\begin{equation}\label{eq_sharpineq1}
      \|(f\sigma)^\vee\|_{L^4(\mathbb F_q^3,\textup d \boldsymbol{x})}^4 \leq \left(1+\frac1q-\frac1{q^2}\right)\|f\|_{L^2(\mathbb P^2,\textup d \sigma)}^4
\end{equation}  
 is sharp, and 
 equality holds if  $f:\mathbb P^2\to\Co$ is a constant function. 
 Moreover, any  maximizer of \eqref{eq_sharpineq1} has constant modulus.
\end{theorem}
Our second result fully characterizes the maximizers of inequality \eqref{eq_sharpineq1} whenever $q\equiv 1(\textup{mod}\, 4)$. The latter condition ensures that $-1$ is a square in $\mathbb F_q$.
Given $\boldsymbol{\xi}\in \mathbb F_q^{2\ast}$, we then write $\boldsymbol{\xi}=\eta(1,w)+\zeta(1,-w)$, where $w^2=-1$. 
We could use the canonical basis $\{(1,0),(0,1)\}$ of $\mathbb F_q^{2\ast}$ instead, but the lines spanned by $(1,\pm w)$ turn out to capture the geometry of the parabolic extension problem in a more transparent way.
We slightly abuse notation and identify a function  $f:\mathbb P^2\to\Co$  with its projection $f:\mathbb F_q^{2\ast}\to\Co$, $f(\boldsymbol{\xi})=f(\boldsymbol{\xi},\boldsymbol{\xi}^2)$.
 Finally, the trace map,
 $\textup{Tr}_n:\mathbb F_q\to \mathbb{F}_{p}$, is defined in \eqref{eq_defTrace} below.
\begin{theorem}\label{thm2}
     Let $q=p^n$ and  $w\in\mathbb F_q$ be such that $q\equiv 1(\textup{mod}\, 4)$ and $w^2=-1$.
     Then $f:\mathbb P^2\to\Co$ is a maximizer of \eqref{eq_sharpineq1} if and only if  
\begin{equation}\label{eq_MaxShape}
    {f}(\eta(1,w)+\zeta(1,-w))=\lambda \exp \frac{2\pi i \textup{Tr}_n(a\eta+b\zeta+c\eta\zeta)}{p},
\end{equation}
 for some $\lambda\in\Co\setminus \{0\}$ and  $a,b,c\in \mathbb{F}_q$.
\end{theorem}
Our third result establishes the sharp $L^2\to L^6$ extension inequality from the parabola  $\mathbb P^1\subset \mathbb F_q^{2\ast}$ equipped with normalized surface measure $\sigma=\sigma_{\mathbb P^1}$. This requires $\textup{char}(\mathbb F_q)=p>3$.
\begin{theorem}\label{thm3}
Let $p>3$.
 It holds that
${\bf R}_{\mathbb P^1}^\ast(2\to 6)=(1+q^{-1}-q^{-2})^{\frac 16}$.     
 In other words, the inequality 
\begin{equation}\label{eq_sharpineq2}
      \|(f\sigma)^\vee\|_{L^6(\mathbb F_q^2,\textup d \boldsymbol{x})}^6 \leq \left(1+\frac1q-\frac1{q^2}\right)\|f\|_{L^2(\mathbb P^1,\textup d \sigma)}^6
\end{equation}  
 is sharp, and 
 equality holds if  $f:\mathbb P^1\to\Co$ is a constant function. 
 Moreover, any  maximizer of \eqref{eq_sharpineq2}  has constant modulus. 
\end{theorem}
Theorems \ref{thm1}, \ref{thm2} and \ref{thm3} are finite field analogues of Foschi's results \cite[Theorems 1.1 and 1.4]{Fo07} for the euclidean paraboloid. The crucial observation there was that the relevant convolutions of the projection measure on the paraboloid are constant {\it in the interior of their support}, whereas the boundary values can  be safely disregarded since they are attained in a null set. In the case of finite fields, substantial complications arise which are entirely new. These are due to the finite nature of $\mathbb{F}_q$, which makes boundary terms non-negligible, and to various arithmetic matters which we proceed to describe. 

The proofs of Theorems \ref{thm1} and \ref{thm3} share a common high-level structure which can be summarized as follows. After rewriting the problem in convolution form,  we compute the relevant $k$-fold convolution measures, $k\in\{2,3\}$. 
In both cases, the convolution attains two distinct values, one along the so-called {\it critical set}, denoted $\mathcal{C}$, and another one on the remaining  {\it generic set}.  A direct application of the Cauchy--Schwarz inequality suffices to handle the generic set. Our main effort is geared towards a careful estimate of the terms associated to the critical set. The need arises for a distinction between the cases $q\equiv 1\,\text{ or }\, 3\,(\textup{mod}\, 4)$ on $\mathbb{P}^2$ and $q\equiv 1\, \text{ or }\, 2\,(\textup{mod}\, 3)$ on $\mathbb{P}^1$. 
In fact, the set of $k$-tuples of points on the paraboloid which sum to a given $(\boldsymbol{\xi},\tau)$,
\begin{equation}\label{eq_defSigmaSet}
    \Sigma_{\mathbb P^d}^k(\boldsymbol{\xi},\tau):=\left\{(\boldsymbol{\xi}_i)_{i=1}^k\in (\mathbb F_q^{d\ast})^k: \sum_{i=1}^k (\boldsymbol{\xi}_i,\boldsymbol{\xi}_i^2)=(\boldsymbol{\xi},\tau)\right\},
\end{equation}
is a singleton whenever $(\boldsymbol{\xi},\tau)\in\mathcal C$, provided
 $q\equiv  3(\textup{mod}\, 4)$ and $(d,k)=(2,2)$, or   $q\equiv  2 (\textup{mod}\, 3)$ and $(d,k)=(1,3)$. This simplifies the analysis considerably. Estimating the terms associated to $\mathcal C$ in the remaining cases is much more delicate, since adequate bounds for the relevant quantity,
\begin{align}\label{criticalcurveintro}
\sum_{(\boldsymbol{\xi},\tau)\in \mathcal{C}}\left(\left|\underset{{(\boldsymbol{\xi}_i)_{i=1}^k\in \Sigma_{\mathbb P^d}^k(\boldsymbol{\xi},\tau)}}{\sum}\prod_{i=1}^{k}f(\boldsymbol{\xi}_i,\boldsymbol{\xi}_i^2)\right|^2
-(q-1)\underset{{(\boldsymbol{\xi}_i)_{i=1}^k\in \Sigma_{\mathbb P^d}^k(\boldsymbol{\xi},\tau)}}{\sum}\left|\prod_{i=1}^kf(\boldsymbol{\xi}_i,\boldsymbol{\xi}_i^2)\right|^2\right),
\end{align}
do not follow from any straightforward application of the Cauchy--Schwarz inequality; see Remarks \ref{remarkexamplelines} and \ref{remarkexamplediracdelta} below.
At this point, the analysis for $\mathbb{P}^2$ and $\mathbb{P}^1$ splits, as a detailed understanding of the geometry of the sets  $\Sigma_{\mathbb P^d}^k(\boldsymbol{\xi},\tau)$ for $(\boldsymbol{\xi},\tau)\in \mathcal{C}$ is required. 

As far as the proof of Theorem \ref{thm2} is concerned, we note that the euclidean methods from \cite[\S 7]{Fo07} do not seem adequate to handle finite fields.
Instead, we study the cases of equality for all intermediate inequalities required to estimate \eqref{criticalcurveintro}, and then bring in the constraint coming from the initial Cauchy--Schwarz step for the generic set. The structure of the sets $\Sigma_{\mathbb P^d}^k(\boldsymbol{\xi},\tau)$ for $(\boldsymbol{\xi},\tau)\in\mathcal C$ again plays a key role. 

Our methods are quite robust. Surprisingly,  the proof of Theorem \ref{thm1} when $q\equiv \, 1(\textup{mod}\, 4)$ can be modified to yield the  sharp $L^2\to L^4$ extension inequality from the {\it hyperbolic paraboloid},
\begin{equation}\label{eq_defHP}
    \mathbb H^2:=\{(\xi_1,\xi_2,\tau)\in \mathbb{F}^{3\ast}_q:\, \tau=\xi_1^2-\xi_2^2\},
\end{equation}
equipped with the normalized surface measure $\sigma=\sigma_{\mathbb H^2}$, together with the complete characterization of maximizers. This is the content of our fourth result.
We highlight that the corresponding euclidean problem remains open \cite{COS22, DMPS18}. 

\begin{theorem}\label{thm4}
     It holds that
${\bf R}_{\mathbb H^2}^\ast(2\to 4)=(1+q^{-1}-q^{-2})^{\frac 14}$.     
 In other words, the inequality 
\begin{equation}\label{eq_sharpineq4}
      \|(f\sigma)^\vee\|_{L^4(\mathbb F_q^3,\textup d \boldsymbol{x})}^4 \leq \left(1+\frac1q-\frac1{q^2}\right)\|f\|_{L^2(\mathbb H^2,\textup d \sigma)}^4
\end{equation}  
 is sharp, and 
 equality holds if  $f:\mathbb H^2\to\Co$ is a constant function. 
 Moreover, letting $q=p^n$, then
 $f:\mathbb H^2\to\Co$ is a maximizer of \eqref{eq_sharpineq4} if and only if  
\begin{equation}\label{eq_MaxShapeHyp}
    {f}(\eta(1,1)+\zeta(1,-1))=\lambda \exp \frac{2\pi i \textup{Tr}_n(a\eta+b\zeta+c\eta\zeta)}{p},
\end{equation}
 for some $\lambda\in\Co\setminus \{0\}$ and  $a,b,c\in \mathbb{F}_q$.
\end{theorem}

\subsection{Sharp conic extension}
 The restriction property  ${\bf R}^\ast_{\Gamma^3}(2\to 4)$  is  known to hold \cite{KLP21}.
Our fifth result establishes the sharp $L^2\to L^4$ extension inequality from the cone $\Gamma^3\subset \mathbb F_q^{4\ast}$ equipped with the normalized surface measure $\nu= \nu_{\Gamma}$ whenever $q\equiv\, 3(\textup{mod}\, 4)$.
\begin{theorem}\label{thm5}
Let $q\equiv\, 3(\textup{mod}\, 4)$.
 It holds that
\[{\bf R}_{\Gamma^3}^\ast(2\to 4)=(1-2q^{-1}+2q^{-2}-3q^{-4}+3q^{-5})^{\frac 14}(1-q^{-1})^{-\frac34}(1+q^{-2})^{-\frac34}.\]     
 In other words, the inequality 
\begin{equation}\label{eq_sharpineq5}
      \|(f\nu)^\vee\|_{L^4(\mathbb F_q^4,\textup d \boldsymbol{x})}^4 \leq \frac{q^4(q^5-2q^4+2q^3-3q+3)}{(q-1)^3(q^2+1)^3} \|f\|_{L^2(\Gamma^3,\textup d \nu)}^4
\end{equation}  
 is sharp, and 
 equality holds if  $f: \Gamma^3\to\Co$ is a constant function. 
 Moreover, any  maximizer of \eqref{eq_sharpineq5}  has constant modulus.
\end{theorem}

Theorem \ref{thm5} is the finite field analogue of Foschi's result \cite[Theorem 1.5]{Fo07} for the euclidean cone.
The two-fold convolution of the Lorentz invariant measure on the euclidean cone is constant in the interior of its support, but again this does not suffice to handle the case of finite fields. 
In fact, the convolution of the surface measure on the cone $\Gamma^3$ attains {\it three} different values: one  at the origin $\bf 0$, a different value on $\Gamma^3$, and yet a different one on the remaining generic set.
 It turns out that $\Gamma^3$ can be decomposed as a union of disjoint lines, $\{\mathcal{L}_s: s\in S\}$, and that these lines satisfy $\{\boldsymbol{\eta}_1\in \Gamma^3: \boldsymbol{\eta}-\boldsymbol{\eta}_1\in \Gamma^3\}\subset \mathcal{L}_s$, for each  $\boldsymbol{\eta}\in \mathcal{L}_s$.   This structural property plays a key role in the analysis.  \\

Mockenhaupt--Tao remark on \cite[p.~53]{MT04} that  the origin has been removed in \eqref{eq_coneF3} for technical convenience, but that it can be restored with no significant change to the results.
Surprisingly, this is {\it not} the case for the sharp results which we seek to prove. Indeed, if we define the cone
\begin{equation}\label{eq_defCone2}
    {\Upsilon}^3:=\{(\boldsymbol{\xi},\tau,\sigma)\in \mathbb F_q^{2\ast}\times\mathbb F_q^{\ast}\times\mathbb F_q^{\ast}: \tau^2+\sigma^2=\boldsymbol{\xi}^2\}\setminus\{{\bf 0}\},
\end{equation}
 we are able to show that constant functions fail to be  critical points on the full cones ${\Gamma}_0^3:=\Gamma^3\cup\{{\bf 0}\}$ and ${\Upsilon}_0^3:=\Upsilon^3\cup\{{\bf 0}\}.$ This is the content of our sixth and final result, which is  valid on fields of prime cardinality, and stands in sharp contrast with the euclidean case \cite{NOSST23}.

\begin{theorem}\label{thm6}
     Constants are not critical points for the $L^2(S,\textup d \nu)\to L^4(\mathbb F_p^4,\textup d \boldsymbol{x})$ extension inequality from $S\in\{{\Gamma}_0^3,{\Upsilon}_0^3\}$.
\end{theorem}
In particular, constant functions are not local maximizers, and therefore not global maximizers, for ${\bf R}_{\Gamma_0^3}^\ast(2\to 4)$ or ${\bf R}_{\Upsilon_0^3}^\ast(2\to 4)$. We refer the reader to \cite{NOSST23, NOST23} for the state of the art on sharp conic restriction in the euclidean setting.\\

{\bf Overview.} The paper is organized as follows:
In \S\ref{sec_NP}, we  discuss the relevant finite field  preliminaries from  Fourier analysis  and  number theory.
In \S\ref{sec_Conv}, we compute the relevant convolution measures exactly, both via Fourier inversion and counting methods.
In the remaining six sections, \S\ref{sec_thm1}--\S\ref{sec_thm6}, we prove Theorems \ref{thm1}--\ref{thm6}, respectively.

\section{Notation and preliminaries}\label{sec_NP}
\subsection{Notation}
Given  a set $A$, we denote its indicator function by ${\bf 1}_A$.
We occasionally extend this notation to statements $S$, defining ${\bf 1}(S)=1$ if $S$ is true and ${\bf 1}(S)=0$ if $S$ is false. 
If $A$ is  a finite set, then  $|A|$ denotes its cardinality.
Real and imaginary parts of a given complex number $z\in\Co$ are  denoted by $\Re(z)$ and $\Im(z)$, and 
the principal value of the argument is $\textup{Arg}(z)\in(-\pi,\pi]$.
If $\mathcal{F}$ is a finite set of variables, then $C(\mathcal{F})$ denotes a quantity which only depends on elements of $\mathcal{F}$.  
\\

We reserve the letter $p$ to denote an odd prime number, and let $q=p^n$ for some $n\in\mathbb N:=\{1,2,\ldots\}$. As in the introduction, $\mathbb F_q$ denotes the finite field with $q$ elements.

\subsection{Fourier analysis on finite fields}\label{FAFF_sec} 
A useful reference for this section is \cite{Ca06}.
We are interested in the additive characters of $\mathbb F_q$. These can be listed via the {\it trace map}, $\textup {Tr}_n:\mathbb F_q\to \mathbb{F}_{p}$, given by
\begin{equation}\label{eq_defTrace}
    \textup{Tr}_n(x):=x+x^p+\ldots+ x^{p^{n-1}}.
\end{equation}
If $n=1$, then $\textup{Tr}_1$ is just the identity.
For each $a\in\mathbb F_q$, the map $e_a:\mathbb F_q\to\mathbb S^1, e(x):=\exp(2\pi i \textup{Tr}_n(ax)/p)$, is a character of $\mathbb F_q$; if $a\neq 0$, then we say that $e_a$ is a nonprincipal character, in which case $\{e_a(b\cdot): b\in \mathbb F_q\}$ is a listing of all the characters of $\mathbb F_q$. If $e$ is a nonprincipal character of $\mathbb F_q$, which we fix from now onwards, then
\begin{equation}\label{eq_SumCharct}
    \sum_{x\in\mathbb F_q} e(x)=0.
\end{equation}
Let $\mathbb F_q^d$ denote the vector field over $\mathbb F_q$ of dimension $d<\infty$. 
Then $\mathbb F_q^d$ is a finite abelian group, and we can describe its Fourier analysis in terms of the nonprincipal character $e$, since the characters of $\mathbb F_q^d$ are indexed by elements $\boldsymbol{\xi}\in\mathbb F_q^{d\ast}$ of the dual group, via
\[e_{\boldsymbol{\xi}}(\boldsymbol{x}):= e(\boldsymbol{x}\cdot\boldsymbol{\xi})=e\left(\sum_{i=1}^n x_i\xi_i\right)=\prod_{i=1}^n e_{\xi_i}(x_i).\]
From this and \eqref{eq_SumCharct}, it follows that
\begin{equation}\label{eq_SumCharctV}
    \sum_{\boldsymbol{x}\in\mathbb F_q^d} e(\boldsymbol{x}\cdot\boldsymbol{\xi}):=
\left\{ \begin{array}{ll}
q^d, & \text{ if } \boldsymbol{\xi}=\boldsymbol{0},\\
0, & \text{ if } \boldsymbol{\xi}\neq \boldsymbol{0}.
\end{array}\right.
\end{equation}
Choosing a specific nonprincipal character $e:\mathbb F_q\to \mathbb S^1$ amounts to fixing a group isomorphism between $\mathbb F_q^d$ and $\mathbb F_q^{d\ast}$, but the corresponding natural measures are different. We endow $\mathbb F_q^d$ with the counting measure $\textup d \boldsymbol{x}$, and $\mathbb F_q^{d\ast}$ with the normalized counting measure $\textup d \boldsymbol{\xi}$ so that $\mathbb F_q^{d\ast}$ has total mass 1:
\[\int_{\mathbb F_q^d} f(\boldsymbol{x})\,\textup d \boldsymbol{x} := \sum_{\boldsymbol{x}\in\mathbb F_q^d} f(\boldsymbol{x}), \text{ and } \int_{\mathbb F_q^{d\ast}} g(\boldsymbol{\xi})\,\textup d \boldsymbol{\xi} := \frac1{q^d} \sum_{\boldsymbol{\xi}\in\mathbb F_q^{d\ast}} g(\boldsymbol{\xi}).\]
Given a function $f:\mathbb F_q^d\to\Co$,  its Fourier transform $\widehat f:\mathbb F_q^{d\ast}\to\Co$ is defined via
\[\widehat f(\boldsymbol{\xi}):=\int_{\mathbb F_q^d} f(\boldsymbol{x})e(-\boldsymbol{x}\cdot\boldsymbol{\xi})\,\textup d \boldsymbol{x}=\sum_{\boldsymbol{x}\in\mathbb F_q^d} f(\boldsymbol{x}) e(-\boldsymbol{x}\cdot\boldsymbol{\xi}).\]
Fourier inversion states that
    \begin{equation}\label{eq_FI}
        f(\boldsymbol{x}) = \frac1{q^d} \sum_{\boldsymbol{\xi}\in \mathbb F_q^{d\ast}} \widehat f(\boldsymbol{\xi}) e(\boldsymbol{x}\cdot \boldsymbol{\xi})
        =\int_{\mathbb F_q^{d\ast}} \widehat f(\boldsymbol{\xi}) e(\boldsymbol{x}\cdot\boldsymbol{\xi})\,\textup d \boldsymbol{\xi}=: (\widehat f)^\vee(\boldsymbol{x}).
    \end{equation}
Convolution on $\mathbb F_q^d$ is defined in the usual way with respect to counting measure $\textup d \boldsymbol{x}$, whereas on $\mathbb F_q^{d\ast}$ convolution is defined with respect to normalized counting measure $\textup d \boldsymbol{\xi}$. The Fourier transform intertwines convolution and multiplication:
\begin{equation}\label{eq_intertwining}
    \widehat f\, \widehat g = \widehat{f \ast g}, \text{ and } \widehat{fg}=\widehat f\ast \widehat g.
\end{equation}
If $\sigma$ is a measure on $\mathbb F_q^{d\ast}$ defined via its action on a function $f:\mathbb F_q^{d\ast}\to\Co$ by
\[\langle f,\sigma\rangle=\int_{\mathbb F_q^{d\ast}} f(\boldsymbol{\xi})\,\textup d\sigma(\boldsymbol{\xi})
=\frac1{q^d}\sum_{\boldsymbol{\xi}\in\mathbb F_q^{d\ast}} f(\boldsymbol{\xi})w(\boldsymbol{\xi}),\]
then we identify $\sigma$ with the function $w.$
If $1\leq s<d$ and $\pi:\mathbb F_q^s\to \mathbb F_q^{d\ast}$ parametrizes an $s$-dimensional surface  $S\subset\mathbb F_q^{d\ast}$, then the normalized surface measure $\sigma_S$ (with total mass 1) is given by
\begin{equation}\label{eq_fSigmaS}
    \langle f,\sigma_S\rangle
=\frac1{q^s}\sum_{\boldsymbol{y}\in\mathbb F_q^s} f(\pi(\boldsymbol{y}))
=\frac1{q^d}\sum_{\boldsymbol{\xi}\in\mathbb F_q^{d\ast}} f(\boldsymbol{\xi}) q^{d-s} |\pi^{-1}(\boldsymbol{\xi})|.
\end{equation}
In particular, the measure $\sigma_S$ is associated with the function $w(\boldsymbol{\xi})=q^{d-s} |\pi^{-1}(\boldsymbol{\xi})|$. 
The inverse Fourier transform of $\sigma_S$ is given by
\[\sigma_S^\vee(\boldsymbol{x})=\langle e(\cdot\, \boldsymbol{x}),\sigma_S\rangle=\frac1{q^s}\sum_{\boldsymbol{y}\in\mathbb F_q^s} e(\boldsymbol{x}\cdot \pi(\boldsymbol{y})),\]
and, more generally, if $f$ is a complex-valued function defined on the image of $\pi$, then
\begin{equation}\label{eq_fSigmaSHat}
    (f\sigma_S)^\vee(\boldsymbol{x})=\frac1{q^s} \sum_{\boldsymbol{y}\in\mathbb F_q^s} f(\pi(\boldsymbol{y}))e(\boldsymbol{x}\cdot \pi(\boldsymbol{y})).
\end{equation}
In this paper, the surface $S$ will be either
a paraboloid $\mathbb P^d$,
cone $\Gamma^d$,
hyperbolic paraboloid $\mathbb H^2$,
or cone $\Upsilon^3$, respectively defined in
\eqref{eq_paraboloid},
\eqref{eq_coneF3},
\eqref{eq_defHP},
\eqref{eq_defCone2}. In this case, we abbreviate $\sigma=\sigma_S$, and note that \eqref{eq_fSigmaS} specializes to
\begin{equation}\label{eq_convmeas}
   \langle f, \sigma\rangle= \frac1{|S|} \sum_{\boldsymbol{\xi}\in S} f(\boldsymbol{\xi}),
\end{equation}
and that the {\it extension operator} \eqref{eq_fSigmaSHat} acting on a function $f:S\to\Co$ is given by
\begin{equation}\label{eq_FEOdef}
    (f\sigma)^\vee(\boldsymbol{x})=\frac1{|S|} \sum_{\boldsymbol{\xi}\in S} f(\boldsymbol{\xi}) e(\boldsymbol{x}\cdot\boldsymbol{\xi}).
\end{equation}

For the sake of notation, we will henceforth drop the star from $\mathbb F_q^{d\star}$ altogether.\\

Recall the definition of the best constant ${\bf R}_S^\ast$ given in \eqref{eq_FFExt}. 
It turns out that extension inequalities with even exponents admit combinatorial reformulations which are better suited towards sharp refinements.
The next result makes this effective. 
 Further connections between extension inequalities and counting problems are discussed in \cite[p.~49]{Gr13} and \cite[p.~43]{MT04}.

\begin{proposition}\label{prop_counting}
The extension inequality
       \begin{equation}\label{eq_equiv1}
           \|(f\sigma)^\vee\|_{L^{2k}(\mathbb F_q^d,\textup d \boldsymbol{x})} \leq {\bf R}^\ast_S(2\to 2k) \|f\|_{L^2(S,\textup d \sigma)}
       \end{equation} 
is equivalent to the combinatorial inequality
\begin{equation}\label{eq_equiv2}
    \sum_{\boldsymbol{\xi}\in\mathbb F_q^d} \left| \sum_{\substack{\boldsymbol{\xi}_1+\ldots+\boldsymbol{\xi}_k=\boldsymbol{\xi}\\
    \boldsymbol{\xi}_i\in S}} \prod_{i=1}^k f(\boldsymbol{\xi}_i)\right|^2 \leq {\bf C}_S^*(2\to 2k) \left(\sum_{\boldsymbol{\xi}\in S} |f(\boldsymbol{\xi})|^2\right)^k,
\end{equation}
in the sense that the set of maximizers of \eqref{eq_equiv1} coincides with that of \eqref{eq_equiv2}, and the corresponding best constants are related via
\[{\bf C}_S^*(2\to 2k) = q^{-d}|S|^k  {\bf R}^\ast_S(2\to 2k)^{2k}.\]
\end{proposition}
\begin{proof}
Start by noting that
\[ \|f\|_{L^2(S,\textup d \sigma)}^{2k}=|S|^{-k}\left(\sum_{\boldsymbol{\xi}\in S}|f(\boldsymbol{\xi})|^2\right)^k.\]
    Raising the left-hand side of \eqref{eq_equiv1} to the power $2k$, using \eqref{eq_FEOdef}, and reversing the order of summation,
    \begin{align*}
        \sum_{\boldsymbol{x}\in\mathbb F_q^d} |(f\sigma)^\vee(\boldsymbol{x})|^{2k}
        =\frac1{|S|^{2k}}\sum_{\boldsymbol{x}\in\mathbb F_q^d}\left|\sum_{\boldsymbol{\xi}\in S} f(\boldsymbol{\xi})e(\boldsymbol{x}\cdot \boldsymbol{\xi})\right|^{2k}
        = \frac{q^d}{|S|^{2k}} {\sum^\ast} \prod_{i=1}^k f(\boldsymbol{\xi}_i)\overline{f(\boldsymbol{\eta}_i)},
    \end{align*}
    where the  sum $\overset{\ast}{\sum}$ runs over $k$-tuples $(\boldsymbol{\xi}_i)_{i=1}^k,(\boldsymbol{\eta}_i)_{i=1}^k\in S^k$ such that $\sum_{i=1}^k \boldsymbol{\xi}_i=\sum_{i=1}^k \boldsymbol{\eta}_i$. In fact, the orthogonality relation \eqref{eq_SumCharctV} implies that
\begin{equation}\label{eq_notalwayszero}
    \sum_{\boldsymbol{x}\in\mathbb F_q^d} e\left(\boldsymbol{x}\cdot \left(\sum_{i=1}^k\boldsymbol{\xi}_i-\sum_{i=1}^k\boldsymbol{\eta}_i\right)\right)=0
\end{equation}
    unless $\sum_{i=1}^k \boldsymbol{\xi}_i=\sum_{i=1}^k \boldsymbol{\eta}_i$, in which case the left-hand side of \eqref{eq_notalwayszero} equals $q^d$. The final observation is that 
\[ \sum^\ast \prod_{i=1}^k f(\boldsymbol{\xi}_i)\overline{f(\boldsymbol{\eta}_i)}=\sum_{\boldsymbol{\xi}\in\mathbb F_q^d}\left|\sum_{\substack{\boldsymbol{\xi}_1+\ldots+\boldsymbol{\xi}_k=\boldsymbol{\xi}\\
    \boldsymbol{\xi}_i\in S}} \prod_{i=1}^k f(\boldsymbol{\xi}_i)\right|^2. \qedhere\]
\end{proof}

\subsection{Number theory on finite fields} 
A useful reference for this section is \cite[Chapters 5 and 6]{LN97}.
Recall that $p$ denotes an odd prime number, and $q=p^n$ for some $n\in\N$.
Let $\mathbb F_q^\times:=\mathbb F_q\setminus\{0\}$.
For the reader's convenience, we include elementary proofs of some well-known results, including the   size of certain hyperbolas and ellipses in $\mathbb F_q^2$ via the following lemma; see also \cite[Lemma 6.24]{LN97}.

\begin{lemma}\label{lem_pointsofellipse}
Let $r\in \mathbb F_q^\times$.
If $c\neq 0$ is a square in $\mathbb{F}_q$, then 
\begin{equation}\label{eq_ellipse}
\left|\{(x,y)\in  \mathbb{F}_q^2:\, x^2-cy^2=r\}\right|=q-1.    
\end{equation}
 If $c$ is not a square in $\mathbb{F}_q$, then 
\begin{equation}\label{eq_hyperbola}
\left |\{(x,y)\in  \mathbb{F}_q^2:\, x^2-cy^2=r\}\right|=q+1.
\end{equation}
\end{lemma}
\begin{remark}\label{rem_degenerate}
    In the degenerate case $r=0$, one easily checks that 
\begin{equation}\label{eq_lines}
    |\{(x,y)\in\mathbb F_q^2:\, x^2= cy^2\}|:=
\left\{ \begin{array}{ll}
2q-1, & \text{ if } c\neq 0 \text{ is a square in }\mathbb{F}_q,\\
1, & \text{ if } c\neq 0 \text{ is not a square in }\mathbb{F}_q,\\
q, & \text{ if } c=0.  \end{array}\right.
\end{equation}
\end{remark}
\begin{proof}[Proof of Lemma \ref{lem_pointsofellipse}]
If $c\neq 0$ is a square in $\mathbb F_q$, then  let  $w\in \mathbb F_q^\times$ be such that $w^2=c$.
The change of variables $(x,y)\mapsto (u,v):=(x-wy,x+wy)$ then yields 
\begin{align*}
  \left|\{(x,y)\in  \mathbb{F}_q^2:\, x^2-cy^2=r\}\right|&=\left|\{(u,v)\in  \mathbb{F}_q^2:\, uv=r\}\right|
  =\left|\{(u,ru^{-1}):\, u\in \mathbb{F}_q^\times\}\right|
  =q-1,
\end{align*}
since $r\neq 0$.
Now suppose that $c$ is not a square in $\mathbb F_q$. 
Consider the quadratic field extension $\mathbb{F}_q(w)/\mathbb{F}_q$, where $w$ is an element of the algebraic closure of $\mathbb F_q$ such that $w^2=c$. 
Since $\mathbb{F}_q(w)/\mathbb{F}_q$ is a Galois extension of degree two, it has just one non-trivial automorphism.
Therefore the 
automorphism $(x+wy)\mapsto (x+wy)^q$  coincides with the map $x+wy\mapsto x-wy$. Hence
\begin{align*}
\left|\{(x,y)\in  \mathbb{F}_q^2:\, x^2-cy^2=r\}\right|
&=\left|\{(x,y)\in  \mathbb{F}_q^2:\, (x+wy)(x-wy)=r\}\right|\\
&=\left|\{(x,y)\in  \mathbb{F}_q^2:\, (x+wy)^{q+1}=r\}\right|\\
&=\left|\{a\in  \mathbb{F}_q(w):\, a^{q+1}=r\}\right|.
\end{align*}
Since the fields $\mathbb{F}_q(w)$ and $\mathbb F_{q^2}$ are isomorphic, the group $\mathbb{F}_q(w)^\times$ is cyclic; let  $g$ be a generator, and write $g^{\alpha}=r\neq 0$.
Since $r^{q}=r$, it follows  that $g^{\alpha(q-1)}=1$ and therefore $q^2-1$ divides $\alpha(q-1)$. This implies $\alpha=\alpha_0(q+1)$, for some $\alpha_0\in \mathbb{Z}_{q^2-1}$, and so 
\begin{align*}
\left|\{a\in  \mathbb{F}_q(w):\,a^{q+1}=r\}\right|
&=\left|\{\beta\in \mathbb{Z}_{q^2-1}:\,\beta(q+1)\equiv\alpha (\textup{mod} (q^2-1))\}\right|\\
&=\left|\{\beta\in \mathbb{Z}_{q^2-1}:\,\beta\equiv\alpha_0(\textup{mod}(q-1))\}\right|=q+1.
\end{align*}
This concludes the proof of the lemma.
\end{proof}

The Legendre symbol,
\[\left(\frac ap\right):=
\left\{ \begin{array}{ll}
1, & \text{ if } a\neq 0 \text{ is a square in } \mathbb F_p,\\
-1, & \text{ if } a \text{ is not a square in } \mathbb F_p,\\
0, & \text{ if }  a=0,  \end{array}\right.\]
 is a completely multiplicative function of its top argument, i.e., for every $a,b\in\mathbb F_p$, 
\begin{equation}\label{eq_LegMult}
    \left(\frac {ab}p\right)=\left(\frac ap\right)\left(\frac bp\right).
\end{equation}
The Jacobi symbol, $\left(\frac a m\right)$, is a generalization of the Legendre symbol that allows for a composite bottom argument $m$, which is assumed to be odd and positive.

Legendre symbols can be used to evaluate quadratic Gauss sums. 
    If $a\neq 0$, then
    \begin{equation}\label{eq_Sa1}
        S(a):=\sum_{x\in\mathbb F_p} e(ax^2)=\left(\frac ap\right) S(1),
    \end{equation}
where
\begin{equation}\label{eq_S1}
    S(1)=\varepsilon_p\sqrt p,
\text{ and } 
    \varepsilon_p:=\left\{ \begin{array}{ll}
1, & \text{ if } p\equiv 1(\textup{mod } 4),\\
i, & \text{ if } p\equiv 3(\textup{mod } 4);
\end{array} \right.
\end{equation}
 see \cite[Theorems 5.12(i) and 5.15]{LN97}.

\begin{lemma}\label{lem_weightedGS}
Let $\varepsilon_p$ be as in \eqref{eq_S1}.
If $a\in\mathbb F_p$,  then
\begin{equation}\label{eq_weightedGS}
    \sum_{x\in\mathbb F_p^\times} \left(\frac xp\right) e(ax)=\left(\frac ap\right)\varepsilon_p\sqrt{p}.
\end{equation}
\end{lemma}
\begin{proof}
No generality is lost in  assuming $a\neq 0$.
Since $(\frac 0p)=0$, the sum on the left-hand side of \eqref{eq_weightedGS} can run over the whole $\mathbb F_p$ instead of $\mathbb F_p^\times$.
    Since $(\frac ap)^2=1$, identity \eqref{eq_LegMult} and the change of variables $ax=y$ together yield
\begin{equation}\label{eq_intermediate}
    \sum_{x\in\mathbb F_p} \left(\frac xp\right) e(ax)=\left(\frac ap\right)\sum_{x\in\mathbb F_p} \left(\frac {ax}p\right) e(ax)=\left(\frac ap\right)\sum_{y\in\mathbb F_p} \left(\frac {y}p\right) e(y).
\end{equation}
By  \eqref{eq_SumCharct}, $\sum_{y\in\mathbb F_p}  e(y)=0$. On the other hand, $1+(\frac yp)$ equals the number of solutions to the equation $z^2= y$ in $\mathbb F_p$.
It follows that the right-hand side of \eqref{eq_intermediate} equals
\[\left(\frac ap\right)\sum_{z\in\mathbb F_p} e(z^2)=\left(\frac ap\right)S(1),\]
which via
 \eqref{eq_S1} leads to the  desired result.
\end{proof}

\begin{lemma}\label{lem_GenGaussSum}
   Let $\varepsilon_p$ be as in \eqref{eq_S1}.
   If $a\in\mathbb F_p^\times$ and $b\in\mathbb F_p$, then
    \[\sum_{x\in\mathbb F_p} e(a x^2+b x)=e(-\tfrac{b^2}{4a}) \left( \dfrac{a}{p}\right) \varepsilon_p\sqrt{p}.\]
\end{lemma}
\begin{proof}
    Complete the square and change variables. In more detail:
    \begin{equation*}
        \sum_{x\in\mathbb F_p} e(a x^2+b x)
    = e(-\tfrac{b^2}{4a}) \sum_{x\in\mathbb F_p} e(a(x+\tfrac b{2a})^2)
    = e(-\tfrac{b^2}{4a}) \sum_{y\in\mathbb F_p} e(ay^2)
    = e(-\tfrac{b^2}{4a}) S(a).
    \end{equation*}
  The desired result follows from \eqref{eq_Sa1} and \eqref{eq_S1}.
\end{proof}

We recall the law of quadratic reciprocity,
\begin{equation}\label{eq_LQR}
    \left(\frac pr\right)\left(\frac rp\right)=(-1)^{\frac{(p-1)(r-1)}4},
\end{equation}
which  holds for Jacobi symbols of arbitrary odd positive coprime integers $p,r$. 
This follows from the usual version (e.g., \cite[Theorem 5.17]{LN97}) by induction.
We will  need \eqref{eq_LQR} only for odd primes $p$.
 The first supplement then reads
\begin{equation}\label{eq_1stSupp}
    \left(\frac {-1}p\right) = (-1)^{\frac{p-1}2},
\end{equation}
and the second supplement then reads
\begin{equation}\label{eq_2ndSupp}
    \left(\frac 2p\right) = (-1)^{\frac{p^2-1}8}.
\end{equation}
We conclude this section by presenting the following well-known characterization of when $-1$ and $-3$ are squares over $\mathbb F_q$, with $q=p^n$ and $p$ an odd prime.
\begin{lemma}\label{lem_squaresmodQ}
    $-1$ is a square in $\mathbb F_q$ if and only if $q\equiv 1(\textup{mod } 4)$.
If $p>3$, then
    $-3$ is a square in $\mathbb F_q$ if and only if $q\equiv 1(\textup{mod } 3)$.
\end{lemma}
\begin{proof}
    The first statement follows from two elementary observations:  the group $\mathbb F_q^\times$ is cyclic of order $q-1=:m$, and $-1$ is the unique element of order 2 in $\mathbb F_q^\times$ (this uses $p>2$). It follows that there exists $x\in\mathbb F_q^\times$ for which $x^2=-1$ if and only if $4|m$, or equivalently $q\equiv 1(\textup{mod }4)$.
    The second statement is also straightforward to verify. The case $n=1$ follows from complete multiplicativity \eqref{eq_LegMult} of the Legendre symbol, the first supplement \eqref{eq_1stSupp} and quadratic reciprocity \eqref{eq_LQR}:
    \[\left(\frac {-3}p\right)=\left(\frac {-1}p\right)\left(\frac {3}p\right)=(-1)^{\frac{p-1}2}(-1)^{\frac{(p-1)(3-1)}4}\left(\frac p3\right)=\left(\frac p3\right),\]
    which equals 1 if and only if $p\equiv 1(\textup{mod }3)$.
    If $q=p^n$ for some $n>1$, we consider two cases: if $p\equiv 1(\textup{mod }3)$, then $-3$ is  a square already in $\mathbb F_p$. If $p\equiv 2(\textup{mod }3)$, then the polynomial $x^2+3$ is irreducible over $\mathbb F_p$, and so $K:=\mathbb F_p[x]/\langle x^2+3\rangle$ is a field which contains both zeros of $x^2+3$. Since $[K:\mathbb F_p]=2$, we have that $-3$ is a square in $\mathbb F_{p^2}$. But $\mathbb F_{p^2}\subset
\mathbb F_{p^n}$ if and only if $2|n$. To conclude, we observe that $p\equiv 1(\textup{mod }3)$ or $2|n$ is equivalent to $q\equiv 1(\textup{mod }3)$.
\end{proof}

\section{Convolutions}\label{sec_Conv}

The paraboloid $\mathbb P^d$ and the cone $\Gamma^d$ were defined in \eqref{eq_paraboloid} and \eqref{eq_coneF3}, respectively.
In this section, we compute the $k$-fold convolution of normalized surface measure $\sigma=\sigma_S$ on $S\in\{\mathbb P^d, \Gamma^d\}$ for different values of $(d,k,n)$, where $n$ is such that $q=p^n$.
We take two complementary approaches.

In \S\ref{sec_FIparab}, we use Fourier analysis to handle the case of $\mathbb P^d$ for general $d,k\geq 2$, but only when $n=1$. 
In \S\ref{sec_CountParab}, we use  elementary counting methods to handle the case of $\mathbb P^d$ for general $n\geq 1$, but only when directly relevant to Theorems \ref{thm1}--\ref{thm3}, i.e., when $(d,k)\in\{(2,2),(1,3)\}$. 
In \S\ref{sec_ConesFI}, we use Fourier analysis to compute the two-fold convolution on the full cone $\Gamma_0^3:=\Gamma^3\cup\{{\bf 0}\}$, but only when $n=1$, and note that this is related to the case of the full cone $\Upsilon_0
^3:=\Upsilon^3\cup\{{\bf 0}\}$ defined in \eqref{eq_defCone2} when $p\equiv 1(\textup{mod } 4)$.
The case of $\Gamma_0^3$ when $q\equiv 3(\textup{mod }4)$ and $n\geq 1$ is arbitrary, which is directly relevant to Theorem \ref{thm5}, is then handled using counting methods in \S\ref{sec_ConesCount}.

\subsection{Paraboloids in vector spaces over $\mathbb F_p$ via Fourier analysis}\label{sec_FIparab}

If $k$ is even, then we write $k=2^{\nu_2(k)}\ell$, with   $\nu_2(k)=\max\{\ell\in\mathbb N: 2^\ell | k\}$ and $\ell\geq 1$ odd.
Let $(\frac p1):=1$.

\begin{proposition}\label{prop_parabConv}
    Given  $d,k\geq 2$ and an odd prime  $p>k$, let
     $\sigma=\sigma_{\mathbb P^d}$ denote the normalized surface measure on the paraboloid $\mathbb P^{d}\subset\mathbb F_p^{d+1}$. Then the $k$-fold convolution measure of $\sigma$ is given  by
    \begin{equation}
        \sigma^{\ast k}(\boldsymbol{\xi},\tau)=1+\varepsilon_p^{d(k+1)} p^{\frac{d(1-k)}2}\varphi(\boldsymbol{\xi},\tau),\,\,\,\,\,\,(\boldsymbol{\xi},\tau)\in\mathbb F_p^{d+1},
    \end{equation}
    where $\varepsilon_p$ was defined in \eqref{eq_S1} and the function $\varphi=\varphi_{d,k,p}$ is given by
\begin{equation}\label{eq_varphi}
    \varphi(\boldsymbol{\xi},\tau):=\left\{ \begin{array}{ll}
p{\bf 1}_{\{\tau=\boldsymbol{\xi}^2/k\}}-1, & \text{ if } d \text{ is even},\\
(-1)^{\frac{(p-1)(k+1)}4}\left(\frac pk\right)(p{\bf 1}_{\{\tau=\boldsymbol{\xi}^2/k\}}-1), & \text{ if } d, k \text{ are odd},\\
\varepsilon_p \sqrt p (-1)^{\frac{(p-1)(\ell+1)}4+\frac{p^2-1}8 \nu_2(k)}\left(\frac p{\ell}\right)\left(\frac{\boldsymbol{\xi}^2/k-\tau}p\right), & \text{ if } d \text{ is odd and }  k  \text{ is even}.
\end{array} \right.
\end{equation}
\end{proposition}

\begin{proof}
It was verified in \cite[Eq.\@ (17)]{MT04} that $\sigma^\vee=\delta_0+K$, where the Dirac delta $\delta_0$ is defined as usual via $\delta_0({\bf 0},0)=1$ and $\delta_0(\boldsymbol{x},t)=0$ if $(\boldsymbol{x},t)\neq ({\bf 0},0)$, and $K$ is the Bochner--Riesz kernel (see \cite[p.~52]{Gr13}), defined as $K(\boldsymbol{x},0)=0$ for all $\boldsymbol{x}\in\mathbb F_p^d$,  and 
\[K(\boldsymbol{x},t)=\frac1{|\mathbb P^d|}\sum_{(\boldsymbol{\xi},\boldsymbol{\xi}^2)\in \mathbb P^d} e(\boldsymbol{x}\cdot\boldsymbol{\xi}+t\boldsymbol{\xi}^2)=p^{-d} S(t)^{d} e\left(-\tfrac{\boldsymbol{x}^2}{4t}\right),\text{ if } t\neq 0.\]
Here, $S(t)$ denotes the quadratic Gauss sum as in \eqref{eq_Sa1}.
 Using  Fourier inversion \eqref{eq_FI} and the intertwining property \eqref{eq_intertwining}, we have for $k\geq 2$:
\begin{align*}
    \sigma^{\ast k}(\boldsymbol{\xi},\tau)
    &=[(\sigma^\vee)^k]^\wedge(\boldsymbol{\xi},\tau)
    =1+\sum_{(\boldsymbol{x},t)\in \mathbb F_p^d\times\mathbb F_p^\times} (\sigma^\vee)^k(\boldsymbol{x},t)e(-\boldsymbol{x}\cdot\boldsymbol{\xi}-t\tau)\\
    &=1+p^{-dk}\sum_{(\boldsymbol{x},t)\in \mathbb F_p^d\times\mathbb F_p^\times} S(t)^{dk}e\left(-\tfrac{k{\boldsymbol{x}}^2}{4t}\right)e(-{\boldsymbol{x}}\cdot{\boldsymbol{\xi}})e(-t\tau).
\end{align*}
Completing the square,
\[e\left(-\tfrac{k{\boldsymbol{x}}^2}{4t}\right)e(-{\boldsymbol{x}}\cdot{\boldsymbol{\xi}})
=e\left(-\tfrac{k}{4t}(\boldsymbol{{x}}+\tfrac{2t}{k}{\boldsymbol{\xi}})^2\right)e\left(\tfrac{t{\boldsymbol{\xi}}^2}k\right),\]
we obtain via a shift in the $\boldsymbol{x}$-variable:
\begin{align}
    \sigma^{\ast k}(\boldsymbol{\xi},\tau)
    &=1+p^{-dk} \sum_{t\in \mathbb F_p^\times} S(t)^{dk} e(t(\tfrac{{\boldsymbol{\xi}}^2}k-\tau))\left(\sum_{\boldsymbol{{x}}\in \mathbb F_p^d} e(-\tfrac{k\boldsymbol{{x}}^2}{4t} )\right)\notag\\
    &=1+p^{-dk} \sum_{t\in \mathbb F_p^\times} S(t)^{dk} e(t(\tfrac{\boldsymbol{{\xi}}^2}k-\tau))  S(-\tfrac{k}{4t})^{d}   \notag
    \\
    &=1+p^{-dk} S(1)^{d(k+1)}\sum_{t\in\mathbb F_p^\times}\left(\frac t p\right)^{dk}\left(\frac{-k/(4t)}p\right)^{d} e(t(\tfrac{\boldsymbol{\xi}^2}k-\tau)),\label{eq_sigmak}
\end{align} 
where the passage to the last line uses \eqref{eq_Sa1}. We split the analysis into two cases, depending on the parity of $d$. If $d$ is odd, then we will further split into two subcases, depending on the parity of $k$.\\

\noindent{\bf Case 1.}
If $d$ is even, then  all the powers of Legendre symbols appearing in \eqref{eq_sigmak} are equal to $1$. 
Appealing to \eqref{eq_S1}, we then have that
\[ \sigma^{\ast k}(\boldsymbol{\xi},\tau)=1+\varepsilon_p^{d(k+1)}p^{\frac{d(1-k)}2}\sum_{t\in\mathbb F_p^\times} e(t(\tfrac{\boldsymbol{\xi}^2}k-\tau)).\]
By \eqref{eq_SumCharct}, the latter sum evaluates to
\[\sum_{t\in\mathbb F_p^\times} e(t(\tfrac{\boldsymbol{\xi}^2}k-\tau))=\left\{ \begin{array}{ll}
p-1, & \text{ if } \tau=\frac{\boldsymbol{\xi}^2}{k},\\
-1, & \text{ otherwise. }
\end{array} \right.\]
For even $d$, it then holds that
\[\sigma^{\ast k}(\boldsymbol{\xi},\tau)=1+\varepsilon_p^{d(k+1)}p^{\frac{d(1-k)}2}(p{\bf 1}_{\{\tau=\boldsymbol{\xi}^2/k\}}-1).\]

\noindent{\bf Case 2.}
If $d$ is odd, then matters are simpler if $k$ is odd. In this case, $k$ is coprime to $p$ since   $k<p$.
Complete multiplicativity \eqref{eq_LegMult} of the Legendre symbol, the first supplement \eqref{eq_1stSupp}, and quadratic reciprocity \eqref{eq_LQR} together yield
\begin{multline*}
    \left(\frac t p\right)\left(\frac{-k/(4t)}p\right)=\left(\frac {-k/4} p\right)=\left(\frac {-k} p\right)=\left(\frac {-1} p\right)\left(\frac {k} p\right)\\=(-1)^{\frac{p-1}2}(-1)^{\frac{(p-1)(k-1)}4}\left(\frac pk\right)=(-1)^{\frac{(p-1)(k+1)}4}\left(\frac pk\right).
\end{multline*}
From \eqref{eq_sigmak} we then have, for odd $d$ and $k$,
\[\sigma^{\ast k}(\boldsymbol{\xi},\tau)=1+\varepsilon_p^{d(k+1)}p^{\frac{d(1-k)}2}(-1)^{\frac{(p-1)(k+1)}4}\left(\frac pk\right)(p{\bf 1}_{\{\tau=\boldsymbol{\xi}^2/k\}}-1).\]

We finally come to the case when $d$ is odd and $k$ is even.
If $k$ is not a power of 2, then $k=2^{\nu_2(k)}\ell$, for some odd integer $\ell>1$. We compute:
\begin{align}
\begin{split}\label{eq_LegendreMagic}
 \left( \dfrac{t}{p}\right)\left( \dfrac{-k/(4t)}{p}\right)
&=\left( \dfrac{-k}{p}\right)
=(-1)^{\frac{p-1}2}\left( \dfrac{k}{p}\right)\\
&=(-1)^{\frac{p-1}2}\left( \dfrac{2}{p}\right)^{\nu_2(k)}\left( \dfrac{\ell}{p}\right)
=(-1)^{\frac{(p-1)(\ell+1)}4+\frac{p^2-1}8\nu_2(k)}\left( \dfrac{p}{\ell}\right).
\end{split}
\end{align}
The last identity uses quadratic reciprocity \eqref{eq_LQR} ($\ell$ is odd  and coprime to $p$  since $k<p$) and the second supplement \eqref{eq_2ndSupp}. 
From \eqref{eq_sigmak}, we then have, for odd $d$ and even $k$,
\[\sigma^{\ast k}(\boldsymbol{\xi},\tau)=1+\varepsilon_p^{d(k+1)}p^{\frac{d(1-k)}2}(-1)^{\frac{(p-1)(\ell+1)}4+\frac{p^2-1}8\nu_2(k)}\left(\frac p{\ell}\right)\sum_{t\in\mathbb F_p^\times} \left(\frac tp\right)e(t(\tfrac{\boldsymbol{\xi}^2}k-\tau)).\]
The latter sum can be evaluated with Lemma \ref{lem_weightedGS},
\[\sum_{t\in\mathbb F_p^\times} \left(\frac tp\right)e(t(\tfrac{\boldsymbol{\xi}^2}k-\tau))=\left(\frac {{\boldsymbol{\xi}^2}/k-\tau}p\right)\varepsilon_p\sqrt{p},\]
which finally yields
\begin{equation}\label{eq_finalNon2Powers}
    \sigma^{\ast k}(\boldsymbol{\xi},\tau)=1+\varepsilon_p^{d(k+1)+1}p^{\frac{d(1-k)+1}2}(-1)^{\frac{(p-1)(\ell+1)}4+\frac{p^2-1}8\nu_2(k)}\left(\frac p\ell\right)\left(\frac {{\boldsymbol{\xi}^2}/k-\tau}p\right).
\end{equation}
If $k=2^{\nu_2(k)}$ is a power of 2, then matters are simpler. In this case, \eqref{eq_LegendreMagic} simplifies to 
\[ \left( \dfrac{t}{p}\right)\left( \dfrac{-k/(4t)}{p}\right)
=(-1)^{\frac{p-1}2}\left( \dfrac{2}{p}\right)^{\nu_2(k)},\]
and \eqref{eq_sigmak} then boils down to
\begin{equation}\label{eq_sigmakeveneven}
\sigma^{\ast k}(\boldsymbol{\xi},\tau)=1+\varepsilon_p^{d(k+1)+1}p^{\frac{d(1-k)+1}2}(-1)^{\frac{p-1}2+\frac{p^2-1}8\nu_2(k)}\left(\frac {{\boldsymbol{\xi}^2}/k-\tau}p\right).
\end{equation}
In other words, formula \eqref{eq_finalNon2Powers} is still valid for $\ell=1$ with the convention that $(\frac p1)=1$.
This completes the proof of Proposition \ref{prop_parabConv}.
\end{proof}

\begin{remark}
From \eqref{eq_sigmakeveneven}, it follows that the two-fold convolution of normalized surface measure on $\mathbb P^1\subset\mathbb F_p^{2}$, corresponding to $(d,k)=(1,2)$, is given by
\begin{equation}\label{eq_2foldconvF2}
    (\sigma\ast\sigma)({\xi},\tau)=1+(-1)^{\frac{p-1}2+\frac{p^2-1}8}\left(\frac {{{\xi}^2}/2-\tau}p\right),\,\,\,(\xi,\tau)\in\mathbb F_p^2.
\end{equation}
This expression takes on three distinct values,  depending on whether ${{\xi}^2}/2-\tau$ is or is not a square in $\mathbb F_p$, or whether ${{\xi}^2}/2-\tau$ is divisible by $p$, and each of these values occurs multiple times.

\end{remark}

\subsection{Low dimensional paraboloids in vector spaces over $\mathbb F_q$ via counting}\label{sec_CountParab}
Letting $\sigma=\sigma_{\mathbb P^d}$  denote the normalized surface measure on the paraboloid $\mathbb P^d\subset\mathbb F_q^{d+1}$, from Proposition \ref{prop_counting} we have
\begin{equation}\label{eq_countingSigma}
    \sigma^{\ast k}(\boldsymbol{\xi},\tau)= \frac{q^{d+1}}{|\mathbb{P}^{d}|^k} | \Sigma_{\mathbb P^d}^k(\boldsymbol{\xi},\tau)|,
\end{equation}
where the set $\Sigma_{\mathbb P^d}^k(\boldsymbol{\xi},\tau)$ was defined in \eqref{eq_defSigmaSet}.
Since $\mathbb P^d$ is the graph of the function $\varphi: \mathbb F_q^d\to\mathbb F_q$, $\varphi(\boldsymbol{\xi})=\boldsymbol{\xi}^2$, we have that $|\mathbb P^d|=q^d$. 
In the proof of the following result, we compute the cardinality of $\Sigma_{\mathbb P^d}^k(\boldsymbol{\xi},\tau)$, thereby generalizing to vector spaces over $\mathbb F_q$ (and not just over $\mathbb F_p$) the cases $(d,k)\in\{(1,3),(2,2)\}$ of Proposition \ref{prop_parabConv}.
\begin{proposition}\label{propositionparaboloidfinitefields}
    The two-fold convolution on $\mathbb P^2\subset\mathbb F_q^3$ is given by
\begin{equation}\label{eq_2foldconv}
    (\sigma\ast\sigma)(\boldsymbol{\xi},\tau)=\frac1q\times\left\{ \begin{array}{ll}
q\pm q\mp 1, & \text{ if } \tau=\frac{\boldsymbol{\xi}^2}{2},\\
q\mp 1, & \text{ otherwise, }
\end{array} \right.
\end{equation}
where the first choice of signs corresponds to $q\equiv 1(\textup{mod } 4)$ and the second one to $q\equiv 3(\textup{mod } 4)$.
If $p>3$, then the three-fold convolution on $\mathbb P^1\subset\mathbb F_q^2$ is given by
\begin{equation}\label{eq_3foldconv}
    (\sigma\ast\sigma\ast\sigma)({\xi},\tau)=\frac1q\times\left\{ \begin{array}{ll}
q\pm q\mp 1, & \text{ if } \tau=\frac{{\xi}^2}{3},\\
q\mp 1, & \text{ otherwise, }
\end{array} \right.
\end{equation}
where the first choice of signs corresponds to $q\equiv 1(\textup{mod } 3)$ and the second one to $q\equiv 2(\textup{mod } 3)$.

\end{proposition}

Given $(\boldsymbol{\gamma},s)\in\mathbb F_q^d\times\mathbb F_q$, we define the {\it sphere}
of center $\boldsymbol{\gamma}$ and squared radius $s$ as in \cite{IK10}:
\begin{equation}\label{eq_BallDef}
    \mathcal S(\boldsymbol{\gamma},s):=\{\boldsymbol{\eta}\in\mathbb F_q^d:(\boldsymbol{\gamma}-\boldsymbol{\eta})^2=s\}.
\end{equation}

\begin{proof}[Proof of Proposition \ref{propositionparaboloidfinitefields}]
Let us start with the two-fold convolution. 
From $\boldsymbol{\xi}=\boldsymbol{\xi}_1+\boldsymbol{\xi}_2$ and $\tau=\boldsymbol{\xi}_1^2+\boldsymbol{\xi}_2^2$ it follows that
$\boldsymbol{\xi}_1^2+(\boldsymbol{\xi}-\boldsymbol{\xi}_1)^2=\tau$, or equivalently $\boldsymbol{\xi}_1\in\mathcal{S}(\tfrac{\boldsymbol{\xi}}2,\tfrac{2\tau-\boldsymbol{\xi}^2}4)$.
In particular, 
\[|\Sigma_{\mathbb P^2}^2(\boldsymbol{\xi},\tau)|=\left |\mathcal{S}(\tfrac{\boldsymbol{\xi}}{2},\tfrac{2\tau-\boldsymbol{\xi}^2}{4})\right|.\]
If $\tau={\boldsymbol{\xi}^2}/2$,  then a translation, the first statement in Lemma \ref{lem_squaresmodQ} and Remark \ref{rem_degenerate}  imply
\[|\Sigma_{\mathbb P^2}^2(\boldsymbol{\xi},\tau)|=|\mathcal{S}(\tfrac{\boldsymbol{\xi}}2,0)|=\left\{ \begin{array}{ll}
2q-1, & \text{ if } q\equiv 1(\textup{mod }4),\\
1, & \text{ if } q\equiv 3(\textup{mod }4).
\end{array} \right.\]
If $\tau\neq{\boldsymbol{\xi}^2}/2$, then  Lemma \ref{lem_pointsofellipse} implies
\[|\Sigma_{\mathbb P^2}^2(\boldsymbol{\xi},\tau)|=\left\{ \begin{array}{ll}
q-1, & \text{ if } q\equiv 1(\textup{mod }4),\\
q+1, & \text{ if } q\equiv 3(\textup{mod }4).
\end{array} \right.\]
Identity \eqref{eq_2foldconv} follows from this via \eqref{eq_countingSigma}.  
To handle the three-fold convolution,  observe that $\xi=\xi_1+\xi_2+\xi_3$ and $\tau=\xi_1^2+\xi_2^2+\xi_3^2$ together imply
\begin{equation}\label{eq_ToCount}
    3\tau-\xi^2=(\xi_1-\xi_2)^2+(\xi_2-\xi_3)^2+(\xi_3-\xi_1)^2.
\end{equation}
We thus need to count the number of solutions to \eqref{eq_ToCount}. Changing variables $(\xi_1,\xi_2,\xi_3)\mapsto (\xi_1,\alpha_1,\alpha_2):=(\xi_1,\xi_2-\xi_1,\xi_3-\xi_2)$, the latter equals the number of solutions to $\alpha_1^2+\alpha_2^2+(\alpha_1+\alpha_2)^2=3\tau-\xi^2$
or, by further renaming $(\beta,\gamma)=:(\alpha_1+\alpha_2/2,\alpha_2/2)$, the number of solutions to 
$\beta^2+3\gamma^2=(3\tau-\xi^2)/{2}$.
If $\tau={\xi^2}/{3}$,  then  the second statement in Lemma \ref{lem_squaresmodQ} and Remark \ref{rem_degenerate} together imply
\begin{equation}\label{eq_SigmaSetCount0}
|\Sigma_{\mathbb P^1}^3(\xi,\tau)|=\left |\left\{(\beta,\gamma)\in \mathbb{F}_q^2:\,\beta^2+3\gamma^2=0\right\}\right|=\left\{ \begin{array}{ll}
2q-1, & \text{ if } q\equiv 1(\textup{mod }3),\\
1, & \text{ if } q\equiv 2(\textup{mod }3).
\end{array} \right.
\end{equation}
If $\tau\neq{\xi^2}/3$, then $r:=(3\tau-\xi^2)/2$ is nonzero, and Lemma \ref{lem_pointsofellipse} implies
\begin{equation}\label{eq_SigmaSetCount}
    |\Sigma_{\mathbb P^1}^3(\xi,\tau)|=\left |\left\{(\beta,\gamma)\in \mathbb{F}_q^2:\,\beta^2+3\gamma^2=r\right\}\right|=\left\{ \begin{array}{ll}
q-1, & \text{ if } q\equiv 1(\textup{mod }3),\\
q+1, & \text{ if } q\equiv 2(\textup{mod }3).
\end{array} \right.
\end{equation}
Identity \eqref{eq_3foldconv} follows from this via \eqref{eq_countingSigma}. This concludes the proof of the proposition.
\end{proof}

\subsection{Cones in vector spaces over $\mathbb F_p$ via Fourier analysis}\label{sec_ConesFI}
Recall the definition of the cones $\Gamma^3$ and $\Upsilon^3$ given in \eqref{eq_coneF3} and \eqref{eq_defCone2}, respectively, and  that ${\Gamma}_0^3:=\Gamma^3\cup\{{\bf 0}\}$ and ${\Upsilon}_0^3:=\Upsilon^3\cup\{{\bf 0}\}$.
\begin{proposition}\label{prop_ConicConv}
    The two-fold convolution of normalized surface measure on ${\Upsilon}_0^3\subset\mathbb F_p^4$, denoted  $\nu_\Upsilon$,  is given by
\begin{equation}\label{eq_convConeUps}  (\nu_{{\Upsilon}}\ast\nu_{{\Upsilon}})(\boldsymbol{\xi},\tau,\sigma)=\frac{p^3}{(p^2+ p-1)^2}\times\left\{
\begin{array}{ll}
p^2+ p- 1, & \text{ if } (\boldsymbol{\xi},\tau,\sigma)={\bf 0},\\
2 p - 1, & \text{ if } (\boldsymbol{\xi},\tau,\sigma)\in \Upsilon^3,\\
p+ 1, & \text{ if } (\boldsymbol{\xi},\tau,\sigma)\notin {\Upsilon}_0^3.
\end{array} \right.
    \end{equation}  
The two-fold convolution of normalized surface measure on  ${\Gamma}_0^3\subset\mathbb F_p^4$, denoted  $\nu_\Gamma$,  is given by
\begin{equation}\label{eq_convConeGam}(\nu_{{\Gamma}}\ast\nu_{{\Gamma}})(\boldsymbol{\xi},\tau,\sigma)=\frac{p^3}{(p^2\pm p\mp 1)^2}\times\left\{
\begin{array}{ll}
p^2\pm p\mp 1, & \text{ if } (\boldsymbol{\xi},\tau,\sigma)={\bf 0},\\
p\pm p \mp 1, & \text{ if } (\boldsymbol{\xi},\tau,\sigma)\in \Gamma^3,\\
p\pm 1, & \text{ if } (\boldsymbol{\xi},\tau,\sigma)\notin {\Gamma}_0^3,
\end{array} \right.
    \end{equation}
where the first choice of signs corresponds to $p\equiv 1(\textup{mod } 4)$ and the second one to $p\equiv 3(\textup{mod } 4)$.

\end{proposition}

\begin{proof}[Proof of Proposition \ref{prop_ConicConv}]
    We start with the proof of \eqref{eq_convConeUps}, which is  similar to  that of \eqref{eq_convConeGam} when $p\equiv 1(\textup{mod }4).$ The case $p\equiv 3(\textup{mod }4)$  of \eqref{eq_convConeGam} is  simpler and will be presented afterwards.\\

\noindent{\bf The case of } $\Upsilon_0 ^3$. Write $(\boldsymbol{x},t,s)\in \mathbb F_p^4$, where $\boldsymbol{x}=(x_1,x_2)\in\mathbb F_p^2$, for the  variables dual to $(\boldsymbol{\xi},\tau,\sigma)\in\mathbb F_p^{4}$.
Let $\zeta(\boldsymbol{x},t,s):=\boldsymbol{x}^2-t^2-s^2$.
    Start by noting that   
    \begin{equation}\label{eq_nuhatUps}
        \nu_{\Upsilon}^\vee (\boldsymbol{x},t,s)= \left\{ \begin{array}{ll}
1, & \text{ if } (\boldsymbol{x},t,s)={\bf 0},\\
\frac{p-1}{p^2+p-1}, & \text{ if } \zeta(\boldsymbol{x},t,s)=0 \text{ but }(\boldsymbol{x},t,s)\neq {\bf 0},\\
\frac{-1}{p^2+p-1}, & \text{ if } \zeta(\boldsymbol{x},t,s)\neq 0;
\end{array} \right.
    \end{equation}
Indeed, since $(\boldsymbol{\xi},\tau,\sigma)\in{\Upsilon}_0^3$ if and only if $\zeta(\boldsymbol{\xi},\tau,\sigma)=0$, it follows that
\begin{align*}
   \frac{|\Upsilon_0^3|}{p^4} \nu_{{\Upsilon}}^\vee(\boldsymbol{x},t,s)
&=\frac1{p^4}\sum_{(\boldsymbol{\xi},\tau,\sigma)\in{\Upsilon}_0^3} e((\boldsymbol{x},t,s)\cdot (\boldsymbol{\xi},\tau,\sigma))\\
&=\frac1{p^4}\sum_{(\boldsymbol{\xi},\tau,\sigma)\in\mathbb F_p^{4}} \left(\frac1p\sum_{\lambda\in\mathbb F_p}e(\lambda\zeta(\boldsymbol{\xi},\tau,\sigma))\right) e((\boldsymbol{x},t,s)\cdot (\boldsymbol{\xi},\tau,\sigma))\\
&=\frac{\delta_0(\boldsymbol{x},t,s)}{p}+\frac1{p^5}\sum_{\lambda\neq 0}\sum_{(\boldsymbol{\xi},\tau,\sigma)\in\mathbb F_p^4} e(\lambda\zeta(\boldsymbol{\xi},\tau,\sigma)+(\boldsymbol{x},t,s)\cdot (\boldsymbol{\xi},\tau,\sigma))\\
&=\frac{\delta_0(\boldsymbol{x},t,s)}{p}+\frac1{p^3}\sum_{\lambda\neq 0}  e\left(\tfrac{\zeta
(\boldsymbol{x},t,s)}{-4\lambda}\right),
\end{align*}
where the last identity follows from four consecutive applications of Lemma \ref{lem_GenGaussSum}. Therefore
\begin{align*}
    \frac{|\Upsilon_0^3|}{p^4} \nu_{{\Upsilon}}^\vee(\boldsymbol{x},t,s) 
= \frac{\delta_0(\boldsymbol{x},t,s)}{p}+\frac1{p^3}\times\left\{
\begin{array}{ll}
p-1, & \text{ if } \zeta(\boldsymbol{x},t,s)=0,\\
-1, & \text{ if } \zeta(\boldsymbol{x},t,s)\neq 0.
\end{array} \right.
\end{align*}     
From $\nu_{{\Upsilon}}^\vee({\bf 0})=1$, it then follows that $|{\Upsilon}_0^3|=p(p^2+p-1)$, which implies \eqref{eq_nuhatUps}.
We use this to compute the convolution measure $\nu_{{\Upsilon}}\ast\nu_{{\Upsilon}}$ via Fourier inversion \eqref{eq_FI}  and the intertwining property \eqref{eq_intertwining}:
\begin{align}\label{eq_FirstFInv}
\begin{split}
    &(\nu_{{\Upsilon}}\ast\nu_{{\Upsilon}})(\boldsymbol{\xi},\tau,\sigma) = [(\nu_{{\Upsilon}}^\vee)^2]^\wedge(\boldsymbol{\xi},\tau,\sigma)\\
&=1+\left(\tfrac{-1}{p^2+p-1}\right)^2\sum_{\zeta(\boldsymbol{x},t,s)\neq 0} e(-(\boldsymbol{x},t,s)\cdot (\boldsymbol{\xi},\tau,\sigma))+\left(\tfrac{p-1}{p^2+p-1}\right)^2\sum_{\zeta(\boldsymbol{x},t,s)=0 \atop (\boldsymbol{x},t,s)\neq {\bf 0}} e(-(\boldsymbol{x},t,s)\cdot (\boldsymbol{\xi},\tau,\sigma))\\
&=:1+\left(\tfrac{-1}{p^2+p-1}\right)^2\sum_{(\boldsymbol{x},t,s)\neq {\bf 0}} e(-(\boldsymbol{x},t,s)\cdot (\boldsymbol{\xi},\tau,\sigma))+\tfrac{p^2-2p}{(p^2+p-1)^2}\left(-1+\frak{S}(\boldsymbol{\xi},\tau,\sigma)\right).
\end{split}
\end{align}
Since the first sum on the previous line equals $p^4\delta_0(\boldsymbol{\xi},\tau,\sigma)-1$, our main task will be to compute 
\[\frak{S}(\boldsymbol{\xi},\tau,\sigma):=\sum_{\zeta(\boldsymbol{x},t,s)=0} e(-(\boldsymbol{x},t,s)\cdot (\boldsymbol{\xi},\tau,\sigma)).\]
Note that $\frak{S}({\bf 0})=|{\Upsilon}_0^3|$ and that $\frak{S}(\pm \xi_1,\pm \xi_2,\pm \tau,\pm \sigma)$ is independent of the choice of signs.
Changing variables $(a,b):=(s-x_2,s+x_2)$, which implies $(x_2,s)=\frac12(b-a,a+b)$, yields
\[(\boldsymbol{x},t,s)\cdot(\boldsymbol{\xi},\tau,\sigma)
=(x_1,\tfrac{b-a}2,t,\tfrac{a+b}2)\cdot(\boldsymbol{\xi},\tau,\sigma)
= x_1\xi_1+a\tfrac{\sigma-\xi_2}2+t\tau+b\tfrac{\xi_2+\sigma}2.\]
After this change of variables, $\zeta(\boldsymbol{x},t,s)=0$ if and only $x_1^2-t^2=a b$, and so
\begin{equation}\label{eq_S1S2}
    \frak{S}(\boldsymbol{\xi},\tau,\sigma)=\left(\sum_{x_1^2-t^2=a b \atop b=0}+\sum_{x_1^2-t^2=a b \atop b\neq 0}\right)e(-(\boldsymbol{x},t,s)\cdot(\boldsymbol{\xi},\tau,\sigma))=:\frak S_1(\boldsymbol{\xi},\tau,\sigma)+\frak S_2(\boldsymbol{\xi},\tau,\sigma).
\end{equation}
The first sum in \eqref{eq_S1S2}, corresponding to $b=0$, evaluates to
\begin{align}
    \frak S_1(\boldsymbol{\xi},\tau,\sigma) 
    &= \sum_{t=\pm x_1} \sum_{a\in\mathbb F_p} e(-x_1\xi_1 -t\tau) e\left(a\frac{\xi_2-\sigma}2\right)\notag\\
    &= \left(\sum_{x_1\in\mathbb F_p}  e(-x_1 (\xi_1+\tau))  + \sum_{x_1\in\mathbb F_p}  e(-x_1 (\xi_1-\tau)) -1\right)\sum_{a\in\mathbb F_p} e\left(a\frac{\xi_2-\sigma}2\right),\label{eq_2linesminus1point}
\end{align}
which is nonzero only if $\xi_2=\sigma$. More precisely, we have that
\begin{equation}\label{eq_frakS1}
    \frak S_1(\boldsymbol{\xi},\tau,\sigma) =
\left\{\begin{array}{ll}
(2p-1)p, & \text{ if } \xi_1=\tau=0 \text{ and }\xi_2=\sigma,\\
(p-1)p, & \text{ if } \xi_1=\pm \tau\neq 0 \text{ and }\xi_2=\sigma,\\
-p, & \text{ if } \xi_1\neq\pm \tau \text{ and }\xi_2=\sigma,\\
0, & \text{ otherwise}.
\end{array} \right.
\end{equation}
We proceed to compute the second sum in \eqref{eq_S1S2}, corresponding to $b\neq 0$. 
Changing variables $(A,B):=(\sigma-\xi_2,\sigma+\xi_2)$, we have that
\begin{align}
    \frak S_2(\boldsymbol{\xi},\tau,\sigma)&= \sum_{x_1,t \in\mathbb F_p} \sum_{b\neq 0} e\left(-x_1\xi_1 -\frac{x_1^2-t^2}{b}\frac{\sigma-\xi_2}{2}-t\tau-b\frac{\xi_2+\sigma}2\right)\notag\\
    &=\sum_{b\neq 0} \left( \sum_{x_1\in\mathbb F_p} e(-x_1\xi_1 -\tfrac A {2b} x_1^2)\right)
    \left( \sum_{t\in\mathbb F_p} e(-t\tau+\tfrac A {2b} t^2)\right) e(-\tfrac{bB}2).\label{eq_A}
\end{align}

If $A\neq 0$, i.e.,  $\xi_2\neq \sigma$, then Lemma \ref{lem_GenGaussSum} implies
\begin{align}
    \frak{S}_2(\boldsymbol{\xi},\tau,\sigma)
&=\varepsilon_p^2 p\sum_{b\neq 0} \left( \dfrac{-A/2b}{p}\right) \left( \dfrac{A/2b}{p}\right)e\left(\tfrac{\xi_1^2 b}{2A}\right) e\left(\tfrac{\tau^2 b}{-2A}\right)e\left(-\tfrac{bB}2\right)\notag\\
&=p\varepsilon_p^2 \left( \dfrac{-1}{p}\right)\sum_{b\neq 0} e\left(\tfrac{(\xi_1^2-\tau^2)b}{2A}\right)e\left(-\tfrac{bB}2\right)\notag\\
&=p\sum_{b\neq 0} e\left(\tfrac{\zeta(\boldsymbol{\xi},\tau,\sigma)b}{2A}\right)
=p\times\left\{\begin{array}{ll}
p-1, & \text{ if } \zeta(\boldsymbol{\xi},\tau,\sigma)=0,\\
-1, & \text{ if } \zeta(\boldsymbol{\xi},\tau,\sigma)\neq 0.
\end{array} \right.\label{eq_frakS21}
\end{align}
Here, we used the facts that $\varepsilon_p^2 \left( \dfrac{-1}{p}\right)=1$, for every $p$, and $AB=\sigma^2-\xi_2^2$.
 If $A=0$, i.e., $\xi_2=\sigma$, then from \eqref{eq_A} we have that
\begin{align*}
\frak{S}_2(\boldsymbol{\xi},\tau,\sigma)&=\sum_{b\neq 0} \left( \sum_{x_1\in\mathbb F_p} e(-x_1\xi_1 )\right)
    \left( \sum_{t\in\mathbb F_p} e(-t\tau)\right) e(-\tfrac{bB}2)\\
    &=\left( \sum_{x_1\in\mathbb F_p} e(-x_1\xi_1)\right)
    \left( \sum_{t\in\mathbb F_p} e(-t\tau)\right) \left(\sum_{b\neq 0}  e(-\tfrac{bB}2)\right)\\
    &=p\delta_0(\xi_1)\times p\delta_0(\tau)\times(p\delta_0(B)-1),
\end{align*}
or equivalently
\begin{equation}\label{eq_frakS22}
    \frak{S}_2(\boldsymbol{\xi},\tau,\sigma) = p^2\times\left\{\begin{array}{ll}
p-1, & \text{ if } \xi_1=\tau=0 \text{ and } \xi_2=-\sigma,\\
-1, & \text{ if } \xi_1=\tau=0 \text{ and }\xi_2\neq -\sigma,\\
0,& \text{ otherwise }.
\end{array} \right.
\end{equation}
Identities \eqref{eq_frakS1} and \eqref{eq_frakS21}--\eqref{eq_frakS22} together imply that
\[\frak{S}(\boldsymbol{\xi},\tau,\sigma) =\frak{S}_1(\boldsymbol{\xi},\tau,\sigma)+\frak{S}_2(\boldsymbol{\xi},\tau,\sigma)=\left\{\begin{array}{ll}
p^3+p^2-p, & \text{ if } (\boldsymbol{\xi},\tau,\sigma)={\bf 0},\\
p^2-p, & \text{ if } (\boldsymbol{\xi},\tau,\sigma)\in\Upsilon^3,\\
-p,& \text{ if } (\boldsymbol{\xi},\tau,\sigma)\notin\Upsilon_0^3,
\end{array} \right.\]
from where \eqref{eq_convConeUps} follows.
This concludes the analysis of the two-fold convolution on ${\Upsilon}_0^3$.\\

\noindent{\bf The case of } $\Gamma_0^3$. If $p\equiv 1(\textup{mod }4)$, then $-1$ is a square in $\mathbb F_p$. In this case,  let $w^2=-1$. Then the change of variables $(\tau,\sigma)=(u+wv,u-wv)$ bijectively maps ${\Gamma}_0^3$ into ${\Upsilon}_0^3$, and this implies \eqref{eq_convConeGam}.

If $p\equiv 3(\textup{mod }4)$, then matters are different for $\Gamma_0^3$ but simpler than what we already did for $\Upsilon_0^3$, so we shall be brief. 
Continue to write $(\boldsymbol{x},t,s)\in \mathbb F_p^4$, where $\boldsymbol{x}=(x_1,x_2)\in\mathbb F_p^2$, for the  variables dual to $(\boldsymbol{\xi},\tau,\sigma)\in\mathbb F_p^{4}$.
Let $\eta(\boldsymbol{x},t,s):=\boldsymbol{x}^2-4ts$.
The first observation is that 
    \begin{equation}\label{eq_nuGammaHat}
        \nu_{\Gamma}^\vee (\boldsymbol{x},t,s)= \left\{ \begin{array}{ll}
1, & \text{ if } (\boldsymbol{x},t,s)={\bf 0},\\
\frac{1-p}{p^2-p+1}, & \text{ if } \eta(\boldsymbol{x},t,s)=0 \text{ but }(\boldsymbol{x},t,s)\neq {\bf 0},\\
\frac{1}{p^2-p+1}, & \text{ if } \eta(\boldsymbol{x},t,s)\neq 0;
\end{array} \right.
    \end{equation}
recall \eqref{eq_nuhatUps} and see \cite[Lemma 4.1]{KY17}. 
We then compute the convolution measure $\nu_{{\Gamma}}\ast\nu_{{\Gamma}}$ via Fourier inversion as in \eqref{eq_FirstFInv}:
\begin{multline}\label{eq_FirstMassTransport}
(\nu_{{\Gamma}}\ast\nu_{{\Gamma}})(\boldsymbol{\xi},\tau,\sigma)\\
    =1+\left(\tfrac{1}{p^2-p+1}\right)^2\sum_{(\boldsymbol{x},t,s)\neq {\bf 0}} e(-(\boldsymbol{x},t,s)\cdot(\boldsymbol{\xi},\tau,\sigma))+\tfrac{p^2-2p}{(p^2-p+1)^2}\left(-1+\sum_{\,\eta(\boldsymbol{x},t,s)= 0} e(-(\boldsymbol{x},t,s)\cdot(\boldsymbol{\xi},\tau,\sigma))\right).
\end{multline}
The first sum equals $p^4\delta_0(\boldsymbol{\xi},\tau,\sigma)-1$.
We decompose the second sum in two pieces, depending on whether
$t$ is zero or not: 
\begin{equation}
  \label{eq_GammaSum}
    \sum_{\,\eta(\boldsymbol{x},t,s)= 0} e(-(\boldsymbol{x},t,s)\cdot(\boldsymbol{\xi},\tau,\sigma))=
    \sum_{\boldsymbol{x}^2=0}\left(\sum_{s\in\mathbb F_p} e(-s\sigma)\right) e(-\boldsymbol{x}\cdot\boldsymbol{\xi})
    +\sum_{\boldsymbol{x}\in\mathbb F_p^2} \sum_{t\neq 0} e(-\boldsymbol{x}\cdot\boldsymbol{\xi}-t\tau)e\left(-\tfrac{\boldsymbol{x}^2}{4t}\sigma\right).
\end{equation}
The inner sum of the first summand on the right-hand side of \eqref{eq_GammaSum} is zero unless $\sigma=0$:
\[\sum_{s\in\mathbb F_p} e(-s\sigma) =\left\{
\begin{array}{ll}
p, & \text{ if } \sigma=0,\\
0, & \text{ if } \sigma\neq 0,
\end{array} \right.\]
and  the condition on the outer sum, $\boldsymbol{x}^2=0$, implies $\boldsymbol{x}=\boldsymbol{0}$ since $-1$ is  not a square in $\mathbb F_p$.
As for the second summand on the right-hand side of \eqref{eq_GammaSum}, if $\sigma=0$, then we have
\[\sum_{\boldsymbol{x}\in\mathbb F_p^2} \sum_{t\neq 0} e(-\boldsymbol{x}\cdot\boldsymbol{\xi}-t\tau)=\left\{
\begin{array}{ll}
p^2(p-1), & \text{ if } (\boldsymbol{\xi},\tau,\sigma)={\bf 0},\\
-p^2, & \text{ if } \boldsymbol{\xi}=\boldsymbol{0}\text{ and }\tau\neq 0,\\
0,&\text{ if }\boldsymbol{\xi}\neq\boldsymbol{0}.
\end{array} \right.\]
If $\sigma\neq 0$, then things are a bit more delicate. Completing squares, 
\begin{align}\label{eq_LongSum}
\begin{split}
    &\sum_{\boldsymbol{x}\in\mathbb F_p^2} \sum_{t\neq 0} e(-\boldsymbol{x}\cdot\boldsymbol{\xi}-t\tau)e\left(-\frac{\boldsymbol{x}^2}{4t}\sigma\right)\\
    &=\sum_{t\neq 0}\left\{\sum_{x_1\in\mathbb F_p} e\left(-\tfrac{\sigma}{4 t}\left(x_1+2\tfrac{\xi_1t}{\sigma}\right)^2\right)\right\}\left\{\sum_{x_2\in\mathbb F_p} e\left(-\tfrac{\sigma}{4 t}\left(x_2+2\tfrac{\xi_2t}{\sigma}\right)^2\right)\right\}  e\left(t\left(\tfrac{\boldsymbol{\xi}^2}{\sigma}-\tau\right)\right).
\end{split}
\end{align}
By translation symmetry, the inner Gauss sums are identical, giving rise to a contribution which equals
\[S\left(-\tfrac{\sigma}{4t}\right)^2=S(1)^2=-p,\]
and so \eqref{eq_LongSum} boils down to 
\[-p\sum_{t\neq 0}e\left(t\left(\tfrac{\boldsymbol{\xi}^2}{\sigma}-\tau\right)\right)=
\left\{
\begin{array}{ll}
p-p^2, & \text{ if } (\boldsymbol{\xi},\tau,\sigma)\in\Gamma_0^3 \text{ and }\sigma\neq 0,\\
p, & \text{ if } (\boldsymbol{\xi},\tau,\sigma)\notin \Gamma_0^3 \text{ and }\sigma\neq 0.
\end{array} \right.\]
Putting everything together, we have that
\begin{equation}\label{eq_defAp}
    (\nu_{{\Gamma}}\ast\nu_{{\Gamma}})(\boldsymbol{\xi},\tau,\sigma)=1+
    \frac{-1+p^4{\bf 1}_{(\boldsymbol{\xi},\tau,\sigma)={\bf 0}}+(p^2-2p)A_p(\boldsymbol{\xi},\tau,\sigma)}{(p^2-p+1)^2},
\end{equation}
where the function $A_p$ is given by
\begin{multline*}
    A_p(\boldsymbol{\xi},\tau,\sigma):=-1+p{\bf 1}(\sigma=0)+p^2(p-1){\bf 1}((\boldsymbol{\xi},\tau,\sigma)={\bf 0})\\
    -p^2{\bf 1}(\boldsymbol{\xi}=\boldsymbol{0},  \tau\neq 0=\sigma)+p{\bf 1}(\sigma\neq 0)-p^2{\bf 1}((\boldsymbol{\xi},\tau,\sigma)\in {\Gamma}_0^3,\,\sigma\neq 0).
\end{multline*}
The final observation is that 
\[\{(\boldsymbol{\xi},\tau,\sigma)\in\mathbb F_p^{4}: \boldsymbol{\xi}=\boldsymbol{0},  \tau\neq 0=\sigma\}\cup\{(\boldsymbol{\xi},\tau,\sigma)\in\mathbb F_p^{4}: (\boldsymbol{\xi},\tau,\sigma)\in {\Gamma}_0^3\text{ and }\sigma\neq 0\} = \Gamma^3,\]
and so \eqref{eq_defAp} simplifies to 
\[(\nu_{{\Gamma}}\ast\nu_{{\Gamma}})(\boldsymbol{\xi},\tau,\sigma)=1+\frac{p(p-1)(p-2)-1+p^4(p-1){\bf 1}_{(\boldsymbol{\xi},\tau,\sigma)={\bf 0}}-p^3(p-2){\bf 1}_{(\boldsymbol{\xi},\tau,\sigma)\in \Gamma_0^3}}{(p^2-p+1)^2},\]
from where the case $p\equiv 3(\textup{mod }4)$ of \eqref{eq_convConeGam} follows at once.
\end{proof}

\subsection{Cones in vector spaces over $\mathbb F_q$ via counting}\label{sec_ConesCount}
Letting $\nu=\nu_{\Gamma}$  denote the normalized surface measure on the cone $\Gamma^3_0\subset\mathbb F_q^{4}$, from Proposition \ref{prop_counting} it follows that
\begin{equation}\label{eq_countingNu}
    (\nu_\Gamma\ast\nu_\Gamma)(\boldsymbol{\xi},\tau,\sigma)= \frac{q^{4}}{|\Gamma_0^{3}|^2} | \Sigma_{\Gamma^3_0}^2(\boldsymbol{\xi},\tau,\sigma)|,
\end{equation}
where, similarly to \eqref{eq_defSigmaSet}, we define the set
\[\Sigma_{\Gamma_0^3}^2(\boldsymbol{\xi},\tau,\sigma) := \left\{ \left((\boldsymbol{\xi}_i,\tau_i,\sigma_i)\right)_{i=1}^2\in (\Gamma_0^3)^2: \sum_{i=1}^2(\boldsymbol{\xi}_i,\tau_i,\sigma_i)=(\boldsymbol{\xi},\tau,\sigma) \right\}. \]
In this section, we compute the convolution measure \eqref{eq_countingNu} whenever $q\equiv 3(\textup{mod } 4)$, thereby generalizing this instance of \eqref{eq_convConeGam} to vector spaces over $\mathbb F_q$ (and not just over $\mathbb F_p$).
We start by computing the size of the cone $\Gamma^3_0\subset\mathbb F_q^{4}$.

\begin{lemma}\label{numberofpointsonthecone}
 Let $q\equiv3(\textup{mod }\,4)$. Then 
 $\left| \Gamma_0^3\right|=q(q^2-q+1)$. 
\end{lemma}
\begin{proof}
We  count the number of solutions $(\boldsymbol{\xi},\tau,\sigma)\in \mathbb{F}_q^{4}$ to the equation $\boldsymbol{\xi}^2=\tau\sigma=:\rho$.
If $\rho=0$, then Remark \ref{rem_degenerate} leads to  $2q-1$ solutions.
If $\rho\neq 0$, then identity \eqref{eq_hyperbola} leads to $(q-1)^2 (q+1)$ solutions;  this uses the fact that $-1$ is not a square in $\mathbb F_q$.
To conclude, note that $(2q-1)+(q-1)^2(q+1)=q(q^2-q+1)$.
\end{proof}

 \begin{proposition}\label{propconiconvolutionfinitefields}
 Let $q\equiv 3(\textup{mod }4)$. 
 Then the two-fold convolution on ${\Gamma}_0^3\subset \mathbb{F}_q^{4}$ is given by
    \begin{equation}\label{eq}
    (\nu_{{\Gamma}}\ast\nu_{{\Gamma}})(\boldsymbol{\xi},\tau,\sigma)=\frac{q^3}{(q^2-q+1)^2}\times\left\{
\begin{array}{ll}
q^2-q+1, & \text{ if } (\boldsymbol{\xi},\tau,\sigma)={\bf 0},\\
1, & \text{ if } (\boldsymbol{\xi},\tau,\sigma)\in \Gamma^3,\\
q-1, & \text{ if } (\boldsymbol{\xi},\tau,\sigma)\notin {\Gamma}_0^3.
\end{array} \right.
    \end{equation}

 \end{proposition}
 \begin{proof}
From \eqref{eq_countingNu} and Lemma \ref{numberofpointsonthecone}, it suffices to verify that
    \begin{equation}\label{eq_sizeSigmaGamma}
    |\Sigma_{\Gamma_0^3}^2(\boldsymbol{\xi},\tau,\sigma)|=\left\{
\begin{array}{ll}
q(q^2-q+1), & \text{ if } (\boldsymbol{\xi},\tau,\sigma)={\bf 0},\\
q, & \text{ if } (\boldsymbol{\xi},\tau,\sigma)\in \Gamma^3,\\
q(q-1), & \text{ if } (\boldsymbol{\xi},\tau,\sigma)\notin {\Gamma}_0^3.
\end{array} \right.
    \end{equation}
Start by noting that   $\Sigma_{\Gamma_0^3}^2(\boldsymbol{\xi},\tau,\sigma)$ has the same number of elements as the set  $$S(\boldsymbol{\xi},\tau,\sigma):=\{(\boldsymbol{\xi_1},\tau_1,\sigma_1)\in {\Gamma}_0^3:\, (\boldsymbol{\xi}-\boldsymbol{\xi}_1, \tau-\tau_1,\sigma-\sigma_1)\in {\Gamma}_0^3\}.$$
Given $(\boldsymbol{\xi}_1,\tau_1,\sigma_1)\in S(\boldsymbol{\xi},\tau,\sigma)$, we have
$(\boldsymbol{\xi}-\boldsymbol{\xi}_1)^2=(\tau-\tau_1)(\sigma-\sigma_1).$
Since $\boldsymbol{\xi}_1^2=\tau_1\sigma_1$, this can be rewritten as 
\begin{align}\label{beforesplittingconicconvolution}
\tau\sigma-\boldsymbol{\xi}^2-\tau_1\sigma=\tau \sigma_1-2\boldsymbol{\xi}\cdot\boldsymbol{\xi}_1.
\end{align}
We split the analysis of \eqref{beforesplittingconicconvolution} into two cases, depending on whether or not $\tau$ is nonzero.\\

{\bf Case 1: $\tau\neq 0$.}
In this case, the number of solutions $(\boldsymbol{\xi}_1,\tau_1,\sigma_1)\in S(\boldsymbol{\xi},\tau,\sigma)$ with $\tau_1\neq 0$ of  
\begin{align*}
\left\{
\begin{array}{ll}
\boldsymbol{\xi}_1^2=\tau_1\sigma_1 \\
  \tau\sigma-\boldsymbol{\xi}^2-\tau_1\sigma=\tau \sigma_1-2\boldsymbol{\xi}\cdot\boldsymbol{\xi}_1
\end{array} \right. 
\end{align*}
equals the number of solutions with $\tau_1\neq 0$ of the equation
\begin{equation*}
\tau\sigma-\boldsymbol{\xi}^2-\tau_1\sigma
= \frac{\tau}{\tau_1}\left(\boldsymbol{\xi}_1-\frac{\tau_1}{\tau}\boldsymbol{\xi}\right)^2- \frac{\tau_1}{\tau}\boldsymbol{\xi}^2,  
\end{equation*}
or equivalently of
 \begin{align}\label{equationconicfq}
 (\boldsymbol{\xi}^2-\tau\sigma)\frac{\tau_1-\tau}{\tau}= \frac{\tau}{\tau_1}\left(\boldsymbol{\xi}_1-\frac{\tau_1}{\tau}\boldsymbol{\xi}\right)^2.   
 \end{align}
 We split the analysis into two further subcases.\\
 
 {\it Case 1.1. $\boldsymbol{\xi}^2 \neq \tau \sigma$.}
 By Lemma \ref{lem_pointsofellipse}, any nonzero $\tau_1\neq \tau$ defines $q+1$ points $\boldsymbol{\xi}_1$ that solve \eqref{equationconicfq}. If $\tau_1=\tau$, then necessarily  $\boldsymbol{\xi}_1=\boldsymbol{\xi}$. In this case, we then have $|S(\boldsymbol{\xi},\tau,\sigma)\cap \{\tau_1\neq 0\}|=(q-2)(q+1)+1=q^2-q-1$. On the other hand, if $\tau_1=0$, then $\boldsymbol{\xi}_1^2=0$ and therefore $\boldsymbol{\xi}_1=\boldsymbol{0}$; in particular, $\sigma_1= \sigma-\boldsymbol{\xi}^2/\tau$ yields the unique solution. All in all, we have $|S(\boldsymbol{\xi}, \tau,\sigma)|=q(q-1).$\\

 {\it Case 1.2. $\boldsymbol{\xi}^2=\tau \sigma$.}
For each $\tau_1\neq 0$,  \eqref{equationconicfq} has a unique solution, whence $|S(\boldsymbol{\xi},\tau,\sigma)\cap \{\tau_1\neq 0\}|=q-1$.
On the other hand, if $\tau_1=0$, then $\boldsymbol{\xi}_1=\boldsymbol{0}$, which by \eqref{beforesplittingconicconvolution} forces $\sigma_1=0$. Therefore $|S(\boldsymbol{\xi}, \tau,\sigma)|=q.$\\

{\bf Case 2: $\tau = 0$.}
In this case, equation \eqref{beforesplittingconicconvolution} boils down to
\begin{align}\label{tau=0beforesplittingconvolution}
\boldsymbol{\xi}^2+\tau_1\sigma=2\boldsymbol{\xi}\cdot\boldsymbol{\xi}_1,
\end{align}
which can be analyzed by splitting into three further subcases.\\

{\it Case 2.1. $\boldsymbol{\xi}=\boldsymbol{0}$.}
In this case, $\tau_1\sigma=0$. If $\sigma=0$, then Lemma \ref{numberofpointsonthecone} implies $|S(\boldsymbol{0},0,0)|=\left|{\Gamma}_0^3\right|=q(q^2-q+1)$. If $\sigma\neq 0$, then $(\boldsymbol{\xi}_1,\tau_1)=(\boldsymbol{0},0)$ while $\sigma_1$ is free, and so $|S(\boldsymbol{0},0,\sigma)|=q$. \\ 

In order to handle the two remaining subcases, we observe that the number of solutions of \eqref{tau=0beforesplittingconvolution} (alongside with $\boldsymbol{\xi}_1^2=\tau_1\sigma_1$) when $\sigma_1\neq 0$ equals the number of solutions of
\begin{align}\label{tau=0sigmawedontknow}
\boldsymbol{\xi}^2+\boldsymbol{\xi}_1^2\frac{\sigma}{\sigma_1}=2\boldsymbol{\xi}\cdot\boldsymbol{\xi}_1.   
\end{align}

{\it Case 2.2. $\boldsymbol{\xi}^2\neq 0$ and $\sigma\neq 0$.}
In this case, \eqref{tau=0sigmawedontknow} can be rewritten as 
\begin{equation}\label{eq_case22}
\left(\boldsymbol{\xi}_1-\frac{\sigma_1}{\sigma}\boldsymbol{\xi}\right)^2 \frac{\sigma}{\sigma_1}=\frac{\sigma_1-\sigma}{\sigma}\boldsymbol{\xi}^2.      
\end{equation}
For each nonzero $\sigma_1\neq \sigma$, Lemma \ref{lem_pointsofellipse} implies the existence of $q+1$ distinct points $\boldsymbol{\xi}_1$ that solve \eqref{eq_case22}. 
If $\sigma_1=\sigma,$ then $\boldsymbol{\xi}_1=\boldsymbol{\xi}$ is the unique solution. Therefore $|S(\boldsymbol{\xi},\tau,\sigma)\cap \{\sigma_1\neq0\}|=(q-2)(q+1)+1=q^2-q-1.$ 
On the other hand, if $\sigma_1=0$, then $\boldsymbol{\xi}_1=\boldsymbol{0}$ and \eqref{tau=0beforesplittingconvolution} becomes $\boldsymbol{\xi}^2+\tau_1\sigma=0$, and therefore $(0,0,-{\boldsymbol{\xi}^2}/{\sigma},0)$ is the unique solution.  Thus $|S(\boldsymbol{\xi},\tau,\sigma)|=q(q-1).$\\

{\it Case 2.3. $\boldsymbol{\xi}^2\neq 0=\sigma$.}
In this case, \eqref{tau=0sigmawedontknow} boils down to 
\begin{equation}\label{eq_case23}
\boldsymbol{\xi}^2=2\boldsymbol{\xi}\cdot\boldsymbol{\xi}_1,
\end{equation}
which has exactly $q$ solutions. From \eqref{eq_case23} it follows that $\boldsymbol{\xi}_1^2\neq 0$, and for each nonzero  $\boldsymbol{\xi}_1$ there exist $q-1$ pairs $(\tau_1,\sigma_1)$, such that $\boldsymbol{\xi}_1^2=\tau_1\sigma_1$. Thus $|S(\boldsymbol{\xi}, \tau,\sigma)|=q(q-1)$. \\

To conclude the proof of \eqref{eq_sizeSigmaGamma}, note that $(\boldsymbol{\xi},\tau,\sigma)\notin\Gamma_0^3$ in Cases 1.1, 2.2 and 2.3, that $(\boldsymbol{\xi},\tau,\sigma)\in \Gamma^3$ in Case 1.2, and that $(\boldsymbol{\xi},\tau,\sigma)\in\Gamma^3_0$ in Case 2.1.
\end{proof}

\section{Proof of Theorem \ref{thm1} }\label{sec_thm1}
In this section, we prove  Theorem \ref{thm1}.
In view of Proposition \ref{prop_counting} and the beginning of the proof of Proposition \ref{propositionparaboloidfinitefields}, we aim to verify the sharp inequality
\begin{equation}\label{eq_countingL2L4Parab}
\sum_{(\boldsymbol{\xi},\tau)\in \mathbb{F}_q^3}\left|\sum_{\boldsymbol{\xi}_1\in \mathcal{S}\left(\frac{\boldsymbol{\xi}}{2},\frac{2\tau-\boldsymbol{\xi}^2}{4}\right)}{f}(\boldsymbol{\xi}_1){f}(\boldsymbol{\xi}-\boldsymbol{\xi}_1)\right|^2\le \left(q+1-\frac{1}{q}\right)\left(\sum_{\mathbb{P}^2}|f|^2\right)^2,    
\end{equation}
for every function $f:\mathbb{P}^2\to \mathbb{C}$, with equality if $f$ is constant. Here, $\sum_{\mathbb{P}^2}|f|^2:=\sum_{\boldsymbol{\xi}\in\mathbb F_q^2}|f(\boldsymbol{\xi},\boldsymbol{\xi}^2)|^2$.
\begin{remark}\label{remarkexamplelines}
In the spirit of Foschi \cite[Eq.\,(13)]{Fo07} and Mockenhaupt--Tao \cite[Lemma 5.1]{MT04}, it may seem natural to use the inequality 
\[\left|\mathcal{S}(\tfrac{\boldsymbol{\xi}}{2},\tfrac{2\tau-\boldsymbol{\xi}^{2}}{4})\right|\le \sup \left|\mathcal{S}\right|,\] where the supremum is taken over all spheres $\mathcal{S}\subset \mathbb F_q^2$.
By Cauchy--Schwarz, this would lead to 
\begin{align}\label{eq_FailedCS}
\begin{split}
\sum_{(\boldsymbol{\xi},\tau)\in \mathbb{F}_q^3}\left|\sum_{\boldsymbol{\xi}_1\in \mathcal{S}\left(\frac{\boldsymbol{\xi}}{2},\frac{2\tau-\boldsymbol{\xi}^2}{4}\right)}{f}(\boldsymbol{\xi}_1){f}(\boldsymbol{\xi}-\boldsymbol{\xi}_1)\right|^2
&\le \sup \left|\mathcal{S}\right|\sum_{(\boldsymbol{\xi},\tau)\in \mathbb{F}_q^3}\sum_{\boldsymbol{\xi}_1\in \mathcal{S}\left(\frac{\boldsymbol{\xi}}{2},\frac{2\tau-\boldsymbol{\xi}^2}{4}\right)}|f(\boldsymbol{\xi}_1){f}(\boldsymbol{\xi}-\boldsymbol{\xi}_1)|^2\\
&=\sup \left|\mathcal{S}\right|\left(\sum_{\mathbb P^2} |f|^2\right)^2,
\end{split}
\end{align}
which  is never sharp. Indeed, from Lemma \ref{lem_pointsofellipse} and Remark \ref{rem_degenerate}, it follows that  $\sup \left|\mathcal{S}\right|$ equals $2q-1$ if $q\equiv1(\textup{mod }4)$ and $q+1$ if $q\equiv3(\textup{mod }4)$. This implies inequality \eqref{eq_countingL2L4Parab} with constants $2q-1$ and $q+1$, respectively, instead of the optimal $q+1-1/q$. Thus a more refined analysis is needed. 
\end{remark}
\noindent The analysis  splits into two cases, depending on the congruence class of $q$ modulo 4.
\subsection{The case $q\equiv3(\textup{mod }4)$}\label{sec_41}
In this case, $-1$ is not a square in $\mathbb F_q$ (Lemma \ref{lem_squaresmodQ}) and  spheres  of radius zero in $\mathbb F_q^2$ are singletons (Remark \ref{rem_degenerate}).
This simplifies the analysis considerably.
Decompose the ambient space $\mathbb F_q^3$ into the critical surface $\mathcal C_2:=\{(\boldsymbol{\xi},\tau)\in\mathbb F_q^3:\, 2\tau=\boldsymbol{\xi}^2\}$ and its complement,  $\mathbb F_q^3\setminus \mathcal C_2.$
On the latter, an application of Cauchy--Schwarz similar to \eqref{eq_FailedCS} yields
\[\sum_{(\boldsymbol{\xi},\tau)\in  \mathbb F_q^3\setminus \mathcal C_2}\left|\sum_{\boldsymbol{\xi}_1\in \mathcal{S}\left(\frac{\boldsymbol{\xi}}{2},\frac{2\tau-\boldsymbol{\xi}^2}{4} \right)}{f}(\boldsymbol{\xi}_1){f}(\boldsymbol{\xi}-\boldsymbol{\xi}_1)\right|^2\le
(q+1)\sum_{(\boldsymbol{\xi},\tau)\in\mathbb F_q^3\setminus\mathcal C_2}\sum_{\boldsymbol{\xi}_1\in \mathcal{S}\left(\frac{\boldsymbol{\xi}}{2},\frac{2\tau-\boldsymbol{\xi}^2}{4}\right)}|{f}(\boldsymbol{\xi}_1){f}(\boldsymbol{\xi}-\boldsymbol{\xi}_1)|^2,\]
with equality if $f$ is a constant function.
We now add and subtract the contribution of the critical surface -- a key step of {\it mass transport} flavor which has already appeared in \eqref{eq_FirstFInv} and \eqref{eq_FirstMassTransport} --   yielding 
\begin{align}\label{q=3maininequality2to4}
\begin{split}
 \sum_{(\boldsymbol{\xi},\tau)\in \mathbb{F}_q^3}&\left|\sum_{\boldsymbol{\xi}_1\in \mathcal{S}\left(\frac{\boldsymbol{\xi}}{2},\frac{2\tau-\boldsymbol{\xi}^2}{4}\right)}{f}(\boldsymbol{\xi}_1){f}(\boldsymbol{\xi}-\boldsymbol{\xi}_1)\right|^2
 \leq(q+1)\sum_{(\boldsymbol{\xi},\tau)\in\mathbb{F}_q^3}\sum_{\boldsymbol{\xi}_1\in \mathcal{S}\left(\frac{\boldsymbol{\xi}}{2},\frac{2\tau-\boldsymbol{\xi}^2}{4}\right)}|{f}(\boldsymbol{\xi}_1){f}(\boldsymbol{\xi}-\boldsymbol{\xi}_1)|^2\\
&+\sum_{(\boldsymbol{\xi},\tau)\in\mathcal{C}_2}\left(\left|\sum_{\boldsymbol{\xi}_1\in \mathcal{S}\left(\frac{\boldsymbol{\xi}}{2},0\right)}{f}(\boldsymbol{\xi}_1){f}(\boldsymbol{\xi}-\boldsymbol{\xi}_1)\right|^2
-(q+1)\sum_{\boldsymbol{\xi}_1\in \mathcal{S}\left(\frac{\boldsymbol{\xi}}{2},0\right)}|{f}(\boldsymbol{\xi}_1){f}(\boldsymbol{\xi}-\boldsymbol{\xi}_1)|^2\right).
\end{split}
\end{align}
The first sum on the right-hand side of \eqref{q=3maininequality2to4} can be computed as follows:
\begin{align}
 \sum_{(\boldsymbol{\xi},\tau)\in\mathbb{F}_q^3}&\sum_{\boldsymbol{\xi}_1\in \mathcal{S}\left(\frac{\boldsymbol{\xi}}{2},\frac{2\tau-\boldsymbol{\xi}^2}{4}\right)}|{f}(\boldsymbol{\xi}_1){f}(\boldsymbol{\xi}-\boldsymbol{\xi}_1)|^2\notag\\
 &=\sum_{\boldsymbol{\xi}_1,\boldsymbol{\xi}_2\in \mathbb{F}_q^2}|{f}(\boldsymbol{\xi}_1){f}(\boldsymbol{\xi}_2)|^2\sum_{\tau\in \mathbb{F}_q}{\bf 1}\left(\boldsymbol{\xi}_1, \boldsymbol{\xi}_2\in \mathcal{S}\left(\tfrac{\boldsymbol{\xi}_1+\boldsymbol{\xi}_2}2,\tfrac{2\tau-\left(\boldsymbol{\xi}_1+\boldsymbol{\xi}_2\right)^2}{4}\right)\right)\label{eq_constant1}\\
 &=\sum_{\boldsymbol{\xi}_1,\boldsymbol{\xi}_2\in \mathbb{F}_q^2}|{f}(\boldsymbol{\xi}_1){f}(\boldsymbol{\xi}_2)|^2=\left(\sum_{\mathbb P^2} |f|^2\right)^2. \notag   
\end{align}
Indeed, for each given pair $(\boldsymbol{\xi}_1,\boldsymbol{\xi}_2)\in (\mathbb F_q^2)^2$, there exists a unique $\tau\in\mathbb F_q$ such that $(\boldsymbol{\xi}_1-\boldsymbol{\xi}_2)^2=2\tau-\left({\boldsymbol{\xi}_1+\boldsymbol{\xi}_2}\right)^2$, and so the inner sum in \eqref{eq_constant1} is equal to 1. 
On the other hand, since $\mathcal{S}(\tfrac{\boldsymbol{\xi}}{2},0)=\{\tfrac{\boldsymbol{\xi}}{2}\}$, the second sum on the right-hand side of \eqref{q=3maininequality2to4} boils down to 
\begin{equation}\label{eq_preCS2}
    \sum_{(\boldsymbol{\xi},\tau)\in\mathcal{C}_2}\left(\left|\sum_{\boldsymbol{\xi}_1\in \mathcal{S}\left(\frac{\boldsymbol{\xi}}{2},0\right)}{f}(\boldsymbol{\xi}_1){f}(\boldsymbol{\xi}-\boldsymbol{\xi}_1)\right|^2
-(q+1)\sum_{\boldsymbol{\xi}_1\in \mathcal{S}\left(\frac{\boldsymbol{\xi}}{2},0\right)}|{f}(\boldsymbol{\xi}_1){f}(\boldsymbol{\xi}-\boldsymbol{\xi}_1)|^2\right) = -q\sum_{\mathbb P^2}|f|^4.
\end{equation}
A second application of the Cauchy--Schwarz inequality yields 
\begin{equation}\label{cauchyschwarzcgr}
    q \sum_{\mathbb P^2}|f|^4\geq \frac1q\left(\sum_{\mathbb P^2} |f|^2\right)^2,
\end{equation} 
with equality if and only if $|f|$ is constant.
The desired inequality \eqref{eq_countingL2L4Parab} follows from \eqref{q=3maininequality2to4}--\eqref{cauchyschwarzcgr}, and 
 is sharp since constant functions turn each step of the proof into an equality. 
 Finally, the cases of equality in \eqref{cauchyschwarzcgr} imply that any maximizer must necessarily have constant modulus. This concludes the proof of Theorem \ref{thm1} when $q\equiv3(\textup{mod }4)$.

\subsection{The case $q\equiv 1(\textup{mod }4)$ }\label{subsectionq=1mod4}
In this case, $-1$ is  a square in $\mathbb F_q$ (Lemma \ref{lem_squaresmodQ}) and  spheres of radius zero  in $\mathbb F_q^2$ have $2q-1$ elements (Remark \ref{rem_degenerate}).
This complicates the analysis, which nonetheless starts in a similar way to that of \S\ref{sec_41}. Via Cauchy--Schwarz and mass transport, we have 
\begin{align}
\begin{split}\label{q=1casefirstinequality}
 &\sum_{(\boldsymbol{\xi},\tau)\in \mathbb{F}_q^3}\left|\sum_{\boldsymbol{\xi}_1\in \mathcal{S}\left(\frac{\boldsymbol{\xi}}{2},\frac{2\tau-\boldsymbol{\xi}^2}{4}\right)}{f}(\boldsymbol{\xi}_1){f}(\boldsymbol{\xi}-\boldsymbol{\xi}_1)\right|^2
\leq (q-1)\left(\sum_{\mathbb P^2}|f|^2\right)^2\\
 &+\sum_{(\boldsymbol{\xi},\tau)\in\mathcal{C}_2}\left(\left|\sum_{\boldsymbol{\xi}_1\in \mathcal{S}\left(\frac{\boldsymbol{\xi}}{2},0\right)}{f}(\boldsymbol{\xi}_1){f}(\boldsymbol{\xi}-\boldsymbol{\xi}_1)\right|^2-(q-1)\sum_{\boldsymbol{\xi}_1\in \mathcal{S}\left(\frac{\boldsymbol{\xi}}{2},0\right)}|{f}(\boldsymbol{\xi}_1){f}(\boldsymbol{\xi}-\boldsymbol{\xi}_1)|^2\right).
 \end{split}
\end{align}
Equality in \eqref{q=1casefirstinequality} is achieved if and only if 
\begin{equation}\label{eq_CasesEq1111}
    {f}(\boldsymbol{\xi}_1){f}(\boldsymbol{\xi}-\boldsymbol{\xi}_1)=C(\boldsymbol{\xi},\tau), \text{ for every } \boldsymbol{\xi}_1\in \mathcal{S}\left(\frac{\boldsymbol{\xi}}{2},\frac{2\tau-\boldsymbol{\xi}^2}{4}\right) \text{ such that } 2\tau\neq\boldsymbol{\xi}^2.
\end{equation}
\begin{remark} In the spirit of  \S\ref{sec_41}, one may try to estimate the left-hand side of \eqref{q=1casefirstinequality}  by Cauchy--Schwarz only, via the upper bound
\begin{align}\label{linesaretheenemy2to4}
\sum_{(\boldsymbol{\xi},\tau)\in{\mathbb{F}_q^3}}\left|\mathcal{S}\left(\frac{\boldsymbol{\xi}}{2},\frac{2\tau-\boldsymbol{\xi}^2}{4}\right)\right|\sum_{\boldsymbol{\xi}_1\in \mathcal{S}\left(\frac{\boldsymbol{\xi}}{2},\frac{2\tau-\boldsymbol{\xi}^2}{4}\right)}|{f}(\boldsymbol{\xi}_1){f}(\boldsymbol{\xi}-\boldsymbol{\xi}_1)|^2.
\end{align}
However, for fixed $\sum_{\mathbb P^2}|f|^2$, this expression is not maximized by constants. 
Indeed, consider the indicator function of the line (contained in $\mathbb P^2$) spanned by the vector $(1,w,0)\in\mathbb F_q^3$, where $w^2=-1$, i.e., let $f_0:={\bf 1}\left(t(1,w,0):\,t\in \mathbb{F}_q\right)$.
If $(\boldsymbol{\xi},\tau)=(t(1,w),0)$, then Remark \ref{rem_degenerate} yields
$$\left|\mathcal{S}\left(\frac{\boldsymbol{\xi}}{2}, \frac{2\tau-\boldsymbol{\xi}^2}{4}\right)\right|=\left|\mathcal{S}\left(\frac{t}{2}(1,w),0\right)\right|=2q-1.$$ 
For any  $t\in \mathbb{F}_q$, we further have that
$$\sum_{t_1(1,w)\in \mathcal{S}\left(\frac{t}{2}(1,w),0\right)}|{f_0}(t_1(1,w))|^2|{f_0}((t-t_1)(1,w))|^2=q,$$ 
and so \eqref{linesaretheenemy2to4} equals  $q^2(2q-1)$ when $f=f_0$. Since $\sum_{\mathbb P^2}|f_0|^2=q$, it then follows that
\begin{align*}
{\left(\sum_{\mathbb P^2} |f_0|^2\right)^{-2}}\sum_{(\boldsymbol{\xi},\tau)\in{\mathbb{F}_q^3}}\left|\mathcal{S}\left(\frac{\boldsymbol{\xi}}{2},\frac{2\tau-\boldsymbol{\xi}^2}{4}\right)\right|\sum_{\boldsymbol{\xi}_1\in \mathcal{S}\left(\frac{\boldsymbol{\xi}}{2},\frac{2\tau-\boldsymbol{\xi}^2}{4}\right)}|{f_0}(\boldsymbol{\xi}_1){f_0}(\boldsymbol{\xi}-\boldsymbol{\xi}_1)|^2=2q-1.
\end{align*}
This implies inequality \eqref{eq_countingL2L4Parab} with constant $2q-1$ instead of the optimal $q+1-1/q$. Thus a more refined analysis is needed. 
\end{remark}

We proceed to analyze the second summand on the right-hand side of \eqref{q=1casefirstinequality}, which  is nothing but \eqref{criticalcurveintro} when $d=k=2$.
We will  prove that it is maximized by constants for fixed $\sum_{\mathbb P^2}|f|^2$, via  the following four steps.\\

{\it Step 1:  Line decomposition.}
Let $w\in \mathbb{F}_q$ be such that $w^2=-1$ in $\mathbb F_q$. 
Given $(\boldsymbol{\xi},\tau)\in\mathcal C_2$,  the sphere $\mathcal{S}(\frac{\boldsymbol{\xi}}{2},0)$  is the union of the two lines 
\[\mathcal L_\pm(\boldsymbol{\xi}):=\left\{\boldsymbol{\xi}_1\in \mathbb{F}_q^2:\, \boldsymbol{\xi}_1=\frac{\boldsymbol{\xi}}{2}+t(1,\pm w),\, t\in \mathbb{F}_q\right\},\]
which intersect exactly at $\frac{\boldsymbol{\xi}}{2}$.
Defining the punctured lines $\mathcal L^\circ_\pm(\boldsymbol{\xi}) := \mathcal L_\pm(\boldsymbol{\xi}) \setminus\{\tfrac{\boldsymbol{\xi}}{2}\}$, we then have 
$\mathcal L^\circ_-(\boldsymbol{\xi}) \cup \mathcal L^\circ_+(\boldsymbol{\xi})=\mathcal{S}(\tfrac{\boldsymbol{\xi}}{2},0)\setminus\{\tfrac{\boldsymbol{\xi}}2\}$.
Going back to \eqref{q=1casefirstinequality}, it follows that
\begin{align}\label{eq_decomposition2to4}
\begin{split}
\sum_{(\boldsymbol{\xi},\tau)\in\mathcal{C}_2}\left|\sum_{\boldsymbol{\xi}_1\in \mathcal{S}\left(\frac{\boldsymbol{\xi}}{2},0\right)}{f}(\boldsymbol{\xi}_1){f}(\boldsymbol{\xi}-\boldsymbol{\xi}_1)\right|^2=
\sum_{(\boldsymbol{\xi},\tau)\in\mathcal{C}_2}\left|\sum_{\boldsymbol{\xi}_1\in\mathcal{L}_-^\circ(\boldsymbol{\xi})\cup\mathcal{L}_+^\circ(\boldsymbol{\xi})}{f}(\boldsymbol{\xi}_1){f}(\boldsymbol{\xi}-\boldsymbol{\xi}_1)\right|^2\\
\\+\sum_{(\boldsymbol{\xi},\tau)\in\mathcal{C}_2}\left(\left|{f}\left(\tfrac{\boldsymbol{\xi}}{2}\right)\right|^4+2\Re\left(\overline{{f}\left(\tfrac{\boldsymbol{\xi}}{2}\right)}^2\sum_{\boldsymbol{\xi}_1\in \mathcal{L}_-^\circ(\boldsymbol{\xi})\cup \mathcal{L}_+^\circ(\boldsymbol{\xi})}{f}(\boldsymbol{\xi}_1){f}(\boldsymbol{\xi}-\boldsymbol{\xi}_1)\right)\right).
\end{split}
\end{align}
We proceed to estimate the two summands on the right-hand side of \eqref{eq_decomposition2to4}.\\

{\it Step 2: Estimating the first summand.}
Let us start with
\begin{multline*}
\left|\sum_{\boldsymbol{\xi}_1\in\mathcal{L}^\circ_-(\boldsymbol{\xi})\cup\mathcal{L}^\circ_+(\boldsymbol{\xi})}f(\boldsymbol{\xi}_1)f(\boldsymbol{\xi}-\boldsymbol{\xi}_1)\right|^2
=\left|\sum_{\boldsymbol{\xi}_1\in\mathcal{L}_-^\circ(\boldsymbol{\xi})}f(\boldsymbol{\xi}_1)f(\boldsymbol{\xi}-\boldsymbol{\xi}_1)\right|^2+\left|\sum_{\boldsymbol{\xi}_1\in\mathcal{L}_+^\circ(\boldsymbol{\xi})}f(\boldsymbol{\xi}_1)f(\boldsymbol{\xi}-\boldsymbol{\xi}_1)\right|^2\\
+2\Re\left(\sum_{\boldsymbol{\xi}_1\in\mathcal{L}_-^\circ(\boldsymbol{\xi})}f(\boldsymbol{\xi}_1)f(\boldsymbol{\xi}-\boldsymbol{\xi}_1)\sum_{\boldsymbol{\xi}_1\in\mathcal{L}_+^\circ(\boldsymbol{\xi})}\overline{f(\boldsymbol{\xi}_1)f(\boldsymbol{\xi}-\boldsymbol{\xi}_1)}\right).
\end{multline*}
By Cauchy--Schwarz, it follows that
\begin{equation}
    \label{cuachyschwarzsecondpiece2to4}
\left|\sum_{\boldsymbol{\xi}_1\in\mathcal{L}_\pm^\circ(\boldsymbol{\xi})}f(\boldsymbol{\xi}_1)f(\boldsymbol{\xi}-\boldsymbol{\xi}_1)\right|^2\le (q-1)\sum_{\boldsymbol{\xi}_1\in\mathcal{L}_\pm^\circ(\boldsymbol{\xi})}|f(\boldsymbol{\xi}_1)f(\boldsymbol{\xi}-\boldsymbol{\xi}_1)|^2,
\end{equation}
where equality holds if and only if $f(\boldsymbol{\xi}_1)f(\boldsymbol{\xi}-\boldsymbol{\xi}_1)=C_\pm(\boldsymbol{\xi})$, for every $\boldsymbol{\xi}_1\in\mathcal{L}_\pm^\circ(\boldsymbol{\xi})$.
 Moreover, a repeated  application of the elementary inequality $2xy\leq x^2+y^2$ yields
\begin{equation}\label{mamgsecondpiece2to4}
\Re\left(\sum_{\boldsymbol{\xi}_1\in\mathcal{L}_-^\circ(\boldsymbol{\xi})}f(\boldsymbol{\xi}_1)f(\boldsymbol{\xi}-\boldsymbol{\xi}_1)\sum_{\boldsymbol{\xi}_1\in\mathcal{L}_+^\circ(\boldsymbol{\xi})}\overline{f(\boldsymbol{\xi}_1)f(\boldsymbol{\xi}-\boldsymbol{\xi}_1)}\right)\le  
\sum_{\boldsymbol{\xi}_1\in\mathcal{L}_-^\circ(\boldsymbol{\xi})}|f(\boldsymbol{\xi}_1)|^2\sum_{\boldsymbol{\xi}_2\in \mathcal{L}_+^\circ(\boldsymbol{\xi})}|f(\boldsymbol{\xi}_2)|^2.
\end{equation}
Inequalities \eqref{cuachyschwarzsecondpiece2to4} and \eqref{mamgsecondpiece2to4} together imply
\begin{align*}
&\left|\sum_{\boldsymbol{\xi}_1\in\mathcal{L}_-^\circ(\boldsymbol{\xi})\cup\mathcal{L}_+^\circ(\boldsymbol{\xi})}f(\boldsymbol{\xi}_1)f(\boldsymbol{\xi}-\boldsymbol{\xi}_1)\right|^2\\
&\quad\le (q-1)\sum_{\boldsymbol{\xi}_1\in \mathcal{L}_-^\circ(\boldsymbol{\xi})\cup\mathcal{L}_+^\circ(\boldsymbol{\xi})}|f(\boldsymbol{\xi}_1)f(\boldsymbol{\xi}-\boldsymbol{\xi}_1)|^2
+2\sum_{\boldsymbol{\xi}_1\in \mathcal{L}_-^\circ(\boldsymbol{\xi})}|f(\boldsymbol{\xi}_1)|^2\sum_{\boldsymbol{\xi}_2\in \mathcal{L}_+^\circ(\boldsymbol{\xi})}|f(\boldsymbol{\xi}_2)|^2,
\end{align*}
and therefore
\begin{align}\begin{split}\label{secondpieceupdated2to4}
&\left|\sum_{\boldsymbol{\xi}_1\in\mathcal{L}_-^\circ(\boldsymbol{\xi})\cup\mathcal{L}_+^\circ(\boldsymbol{\xi})}f(\boldsymbol{\xi}_1)f(\boldsymbol{\xi}-\boldsymbol{\xi}_1)\right|^2-(q-1)\sum_{\boldsymbol{\xi}_1\in \mathcal S(\frac{\boldsymbol{\xi}}2,0)}|f(\boldsymbol{\xi}_1)f(\boldsymbol{\xi}-\boldsymbol{\xi}_1)|^2\\
&\quad\le 2\sum_{\boldsymbol{\xi}_1\in \mathcal{L}_-^\circ(\boldsymbol{\xi})}|f(\boldsymbol{\xi}_1)|^2\sum_{\boldsymbol{\xi}_2\in \mathcal{L}_+^\circ(\boldsymbol{\xi})}|f(\boldsymbol{\xi}_2)|^2-
(q-1)\left|f\left(\tfrac{\boldsymbol{\xi}}{2}\right)\right|^4.
\end{split}
\end{align}

{\it Step 3: Estimating the second summand.}
We need the following estimate whose proof is omitted.

\begin{lemma}\label{lem_Cbasic}
    Let $a,b,c\in \mathbb{C}\setminus\{0\}$. 
The following inequality holds 
\begin{align*}
 \Re(a^2\overline{bc})\le \frac{|a|^2}2(|b|^2+|c|^2),
\end{align*}
and equality holds if and only if $\textup{Arg}(a)=\frac12(\textup{Arg}(b)+\textup{Arg}(c))$ and $|b|=|c|$.
\end{lemma} 
Lemma \ref{lem_Cbasic} and symmetry considerations together yield
\begin{align}
 \Re\left(\overline{f\left(\tfrac{\boldsymbol{\xi}}{2}\right)}^2\sum_{\boldsymbol{\xi}_1\in \mathcal{L}_-^\circ(\boldsymbol{\xi})\cup \mathcal{L}_+^\circ(\boldsymbol{\xi})}f(\boldsymbol{\xi}_1)f(\boldsymbol{\xi}-\boldsymbol{\xi}_1)\right)
& \le \frac12\left|f\left(\tfrac{\boldsymbol{\xi}}{2}\right)\right|^2\sum_{\boldsymbol{\xi}_1\in \mathcal{L}_-^\circ(\boldsymbol{\xi})\cup \mathcal{L}_+^\circ(\boldsymbol{\xi})}\left(|f(\boldsymbol{\xi}_1)|^2+|f(\boldsymbol{\xi}-\boldsymbol{\xi}_1)|^2\right)\notag\\
 &=\left|f\left(\tfrac{\boldsymbol{\xi}}{2}\right)\right|^2\sum_{\boldsymbol{\xi}_1\in \mathcal{L}_-^\circ(\boldsymbol{\xi})\cup \mathcal{L}_+^\circ(\boldsymbol{\xi})}|f(\boldsymbol{\xi}_1)|^2.\label{firstinequality2to4}
\end{align}
Assuming that $f$ never vanishes,  equality holds in \eqref{firstinequality2to4} if and only if 
\begin{equation}\label{eq_CasesEqLemmaArg}
    |f(\boldsymbol{\xi}_1)|=|f(\boldsymbol{\xi}-\boldsymbol{\xi}_1)| \text{ and } 2\textup{Arg}\left(f\left(\tfrac{\boldsymbol{\xi}}{2}\right)\right)=\textup{Arg}(f(\boldsymbol{\xi}_1))+\textup{Arg}(f(\boldsymbol{\xi}-\boldsymbol{\xi}_1)),
\end{equation}
for every $\boldsymbol{\xi}\in \mathbb{F}_q^2$ and $\boldsymbol{\xi}_1\in \mathcal{L}_-^\circ(\boldsymbol{\xi})\cup \mathcal{L}_+^\circ(\boldsymbol{\xi})$.\\

{\it Step 4: End of proof.}
From \eqref{secondpieceupdated2to4} and \eqref{firstinequality2to4}, we see that the second summand on the right-hand side of \eqref{q=1casefirstinequality} is bounded by
\begin{align}\label{ineqalitysimplified2to4}
\sum_{(\boldsymbol{\xi},\tau)\in\mathcal{C}_2}\left(2\sum_{\boldsymbol{\xi}_1\in \mathcal{L}_-^\circ(\boldsymbol{\xi})}|f(\boldsymbol{\xi}_1)|^2\sum_{\boldsymbol{\xi}_2\in \mathcal{L}_+^\circ(\boldsymbol{\xi})}|f(\boldsymbol{\xi}_2)|^2+
(2-q)\left|f\left(\tfrac{\boldsymbol{\xi}}{2}\right)\right|^4+2\left|f\left(\tfrac{\boldsymbol{\xi}}{2}\right)\right|^2\sum_{\boldsymbol{\xi}_1\in \mathcal{L}_-^\circ(\boldsymbol{\xi})\cup \mathcal{L}_+^\circ(\boldsymbol{\xi})}|f(\boldsymbol{\xi}_1)|^2\right).  
\end{align}
This is equal to
\begin{align}\begin{split}\label{secondtolast2to4}
2\sum_{\boldsymbol{\xi}_1,\boldsymbol{\xi}_2\in \mathbb{F}_q^2}|f(\boldsymbol{\xi}_1)|^2|f(\boldsymbol{\xi}_2)|^2-q\sum_{(\boldsymbol{\xi},\tau)\in\mathcal{C}_2}\left|f\left(\tfrac{\boldsymbol{\xi}}{2}\right)\right|^4
=2\left(\sum_{\mathbb P^2} |f|^2\right)^2-q\sum_{\mathbb P^2}|f|^4,
\end{split}
\end{align}
from where the desired sharp inequality \eqref{eq_countingL2L4Parab} follows via \eqref{cauchyschwarzcgr}.
To verify that \eqref{ineqalitysimplified2to4} and \eqref{secondtolast2to4} indeed coincide,
 note that
\begin{align*}\begin{split}\label{expandingrhs2to4}
\sum_{\boldsymbol{\xi}_1\in \mathcal{L}_-^\circ(\boldsymbol{\xi})}|f(\boldsymbol{\xi}_1)|^2\sum_{\boldsymbol{\xi}_2\in \mathcal{L}_+^\circ(\boldsymbol{\xi})}|f(\boldsymbol{\xi}_2)|^2&
+\left|f\left(\tfrac{\boldsymbol{\xi}}{2}\right)\right|^4
+\left|f\left(\tfrac{\boldsymbol{\xi}}{2}\right)\right|^2\sum_{\boldsymbol{\xi}_1\in \mathcal{L}_-^\circ(\boldsymbol{\xi})\cup \mathcal{L}_+^\circ(\boldsymbol{\xi})}|f(\boldsymbol{\xi}_1)|^2
\\
&=\sum_{\boldsymbol{\xi}_1\in \mathcal{L}_-(\boldsymbol{\xi})}|f(\boldsymbol{\xi}_1)|^2\sum_{\boldsymbol{\xi}_2\in \mathcal{L}_+(\boldsymbol{\xi})}|f(\boldsymbol{\xi}_2)|^2.
\end{split}
\end{align*}
Interchanging the order of summation, we further have that
\begin{align*}
 \sum_{(\boldsymbol{\xi},\tau)\in\mathcal{C}_2}\sum_{\boldsymbol{\xi}_1\in \mathcal{L}_-(\boldsymbol{\xi})}|f(\boldsymbol{\xi}_1)|^2\sum_{\boldsymbol{\xi}_2\in \mathcal{L}_+(\boldsymbol{\xi})}|f(\boldsymbol{\xi}_2)|^2
 =\sum_{\boldsymbol{\xi}_1,\boldsymbol{\xi}_2\in \mathbb{F}_q^2}|f(\boldsymbol{\xi}_1)|^2|f(\boldsymbol{\xi}_2)|^2
\end{align*}
since
\begin{equation}\label{oneintersection2to4}
\sum_{(\boldsymbol{\xi},\tau)\in\mathcal{C}_2}{\bf 1}\left(\boldsymbol{\xi}_1\in \mathcal{L}_-(\boldsymbol{\xi}),\boldsymbol{\xi}_2\in \mathcal{L}_+(\boldsymbol{\xi})\right)=1.
\end{equation}
To verify \eqref{oneintersection2to4}, note
that the $4\times 4$ matrix associated to the system of equations
\begin{align*}
\left\{
\begin{array}{ll}
\boldsymbol{\xi}_1=\frac{\boldsymbol{\xi}}{2}+t_1(1,w) \\
 \boldsymbol{\xi}_2=\frac{\boldsymbol{\xi}}{2}+t_2(1,-w)
\end{array} \right. 
\end{align*}
 has nonzero determinant (equal to $\pm w/2$) and therefore, given any pair $(\boldsymbol{\xi}_1,\boldsymbol{\xi}_2)\in(\mathbb F_q^2)^2$, there exists a unique $\boldsymbol{\xi}\in\mathbb F_q^2$, such that $\boldsymbol{\xi}_1\in \mathcal{L}_-(\boldsymbol{\xi})$ and $\boldsymbol{\xi}_2\in \mathcal{L}_+(\boldsymbol{\xi})$. 

 Constant functions turn every single step of the preceding proof into an equality and, as in \S\ref{sec_41}, the cases of equality in \eqref{cauchyschwarzcgr} imply that any maximizer must necessarily have constant modulus.
 This concludes the proof of Theorem \ref{thm1}. 

\section{Proof of Theorem \ref{thm2}}\label{sec_thm2}
In this section, we prove  Theorem \ref{thm2}.
It suffices to show that all maximizers of inequality \eqref{eq_countingL2L4Parab} are of the form \eqref{eq_MaxShape} when $q\equiv 1(\textup{mod}\, 4)$.
This will follow from studying the functional equations satisfied by  functions which saturate the intermediate inequalities from \S\ref{subsectionq=1mod4}.

Let $f_\star:\mathbb P^2\to\Co$ be a maximizer of \eqref{eq_countingL2L4Parab}, which as usual is identified with its projection $f_\star:\mathbb F_q^2\to\Co$.
We have already observed that  $|f_\star|$ is constant (for otherwise \eqref{cauchyschwarzcgr} would be strict). Hence  ${f_\star}=\lambda\rho_\star,$ where $\lambda\in \mathbb{C}\setminus\{0\}$ and $\rho_\star:\mathbb{F}_q^2\to \mathbb{S}^1$ satisfies $\rho_\star(\boldsymbol{0})=1$. 
From the second condition in \eqref{eq_CasesEqLemmaArg}, we have 
\begin{equation}\label{characterlikeclass}
\rho_\star^2\left(\tfrac{\boldsymbol{\xi}}{2}\right)=\rho_\star(\boldsymbol{\xi}_1)\rho_\star(\boldsymbol{\xi}-\boldsymbol{\xi}_1), \text{ for every }  \boldsymbol{\xi}_1\in \mathcal L_+(\boldsymbol{\xi})\cup \mathcal L_-(\boldsymbol{\xi}). 
\end{equation}
The next result is key towards the solution of the functional equation \eqref{characterlikeclass}.
\begin{lemma}\label{middlepointlemma2to4}
Let $q=p^n$ be the power of an odd prime.
Let  $\rho:\mathbb{F}_q\to \mathbb{S}^1$ be such that $\rho(0)=1$ and
\begin{align}\label{mainconditionlemmacauchy}
\rho(x)\rho(y)=\rho\left(\tfrac{x+y}{2}\right)^2,
\text{ for every } x,y \in \mathbb{F}_q.
\end{align}
 Then there exists a unique $a\in \mathbb{F}_q$, such that 
 $$\rho\left(x\right)=\exp{\frac{2\pi i\textup{Tr}_n(ax)}{p}},\text{ for every }x\in\mathbb F_q.$$
\end{lemma}
\begin{proof}
In light of our discussion at the beginning of \S\ref{FAFF_sec}, it suffices to verify that 
\begin{equation}\label{eq_goalLemma51}
    \rho(x+y)=\rho(x)\rho(y), \text{ for every } x,y\in\mathbb F_q.
\end{equation}
From \eqref{mainconditionlemmacauchy}, it follows that 
\begin{equation}\label{eq_multConseq}
    \rho(tx)=\rho(x)^t, \text{ for any } (t,x)\in \mathbb{F}_p\times\mathbb{F}_q. 
\end{equation}
This is a direct consequence of the following chain of identities:
$$\rho((t+1)x)=\rho((t+1)x)\rho(0)=\rho\left(\tfrac{t+1}{2}x\right)^2=\rho(x)\rho(tx).$$
Since $px=0$ for every $x\in\mathbb F_q$, from \eqref{eq_multConseq} it then follows that
\begin{align}\label{rootsofunitylemmacauchy}
1=\rho(0)=\rho(px)=\rho(x)^p.    
\end{align}
On the other hand, \eqref{eq_multConseq} implies  $\rho(2x)=\rho(x)^2$ and so $\rho(x)^2\rho(y)^2=\rho(2x)\rho(2y)=\rho(x+y)^2$. Therefore $\rho(x)\rho(y)=\pm \rho(x+y).$ 
But if $\rho(x)\rho(y)=-\rho(x+y)$, then $(\rho(x)\rho(y))^{p+1}=\rho(x+y)^{p+1}$, which by \eqref{rootsofunitylemmacauchy} would imply $\rho(x)\rho(y)=\rho(x+y)$. This leads to  a contradiction since $p>2$ and $\rho$ is nonzero. Thus \eqref{eq_goalLemma51} holds, as desired.\qedhere
 \end{proof}

By \eqref{characterlikeclass} and Lemma \ref{middlepointlemma2to4},  there exist unique $a,b\in \mathbb{F}_q$ such that, for every $\eta,\zeta\in\mathbb F_q$, 
\begin{align}
&\rho_\star|_{\mathcal{L}_+({\bf 0})}(\eta(1,w))=\exp \frac{2\pi i\textup{Tr}_n(a\eta)}{p},\label{classificationa_1}\\
&\rho_\star|_{\mathcal{L}_-({\bf 0})}(\zeta(1,-w))=\exp \frac{2\pi i\textup{Tr}_n(b\zeta)}{p}.\label{classificationa_2}
\end{align}
More generally, let $\{e_1,\dots,e_n\}$ be a basis of the vector space $\mathbb{F}_q$ over $\mathbb{F}_p$. By Lemma \ref{middlepointlemma2to4}, there exists a unique $n$-tuple $(v_1,\dots, v_n)\in \mathbb{F}_q^{n}$ such that, for every  $i\in\{1,\dots,n\}$ and $\eta\in \mathbb{F}_q$,
\begin{align}\label{classificationa_3}\frac{\rho_\star}{\rho_\star(e_i(1,-w))}\Big|_{\mathcal{L}_+(2e_i(1,-w))}(\eta(1,w)+e_i(1,-w))=\exp \frac{2\pi i\textup{Tr}_n(v_i\eta)}{p}.
\end{align}
Similarly, given any $\eta(1,w)\in \mathcal{L}_+(\boldsymbol{0})$, there exist a unique $v_\eta\in \mathbb{F}_q$ such that
\begin{align}\label{classification2to4fourequation}
\frac{\rho_\star}{\rho_\star(\eta(1,w))}\Big|_{\mathcal{L}_-(2\eta(1,w))}(\eta(1,w)+\zeta(1,-w))=\exp \frac{2\pi i \textup{Tr}_n(v_\eta \zeta)}{p},
\end{align}
  for every $\zeta\in\mathbb F_q$.
  Identities \eqref{classificationa_1}--\eqref{classification2to4fourequation} together imply
\begin{align*}
\exp \frac{2\pi i\textup{Tr}_n(a\eta)}{p}&\exp \frac{2\pi i \textup{Tr}_n(v_\eta e_i)}{p}=\rho_\star(\eta(1,w))\exp \frac{2\pi i \textup{Tr}_n(v_\eta e_i)}{p}=\rho_\star(\eta(1,w)+e_i(1,-w))\\
&=\exp \frac{2\pi i\textup{Tr}_n(v_i\eta)}{p}\rho_\star(e_i(1,-w))=\exp \frac{2\pi i\textup{Tr}_n(v_i\eta)}{p}\exp \frac{2\pi i\textup{Tr}_n(be_i)}{p}. 
\end{align*}
It follows that $$\textup{Tr}_n(a\eta)+\textup{Tr}_n(v_\eta e_i)=\textup{Tr}_n(v_i\eta)+\textup{Tr}_n(be_i),$$
for every $i\in\{1,\ldots,n\}.$
The next result allows us to gain further control over the element $v_\eta$.

\begin{lemma}\label{lemmaxistenceclasstr}
 Let $\{e_1,\dots, e_n\}$ be a basis of the vector space $\mathbb{F}_q$ over $\mathbb{F}_p$. Let $t_1,\dots t_n\in \mathbb{F}_p$.  Then there exists a unique $a\in \mathbb{F}_q$ such that $\textup{Tr}_n(ae_i)=t_i$, for every $i\in\{1,\ldots,n\}$.
\end{lemma}
\begin{proof}
Each $a\in \mathbb{F}_q$ gives rise to a unique $n$-tuple $(\textup{Tr}_n(ae_i))_{i=1}^{n}.$ Indeed, $\textup{Tr}_n(ae_i)=\textup{Tr}_n(be_i)$ for all $i\in\{1,\dots,n\}$  implies  $\textup{Tr}_n(a\cdot)=\textup{Tr}_n(b\cdot)$, and therefore $a=b$; see \cite[Theorem 2.24]{LN97}. The map $a\mapsto (\textup{Tr}_n(ae_i))_{i=1}^n$ is thus injective from $\mathbb{F}_q$ to $\mathbb{F}_p^n$. To conclude, note that  $|\mathbb{F}_q|=|\mathbb{F}_p^n|$.  
\end{proof}

By Lemma \ref{lemmaxistenceclasstr}, $v_\eta$ is the unique element in $\mathbb{F}_q$ such that $\textup{Tr}_n(v_\eta e_i)=\textup{Tr}_n((v_i-a)\eta+b e_i)$, for every $i\in\{1,\ldots,n\}$. Consequently, if $\zeta=\sum_{i=1}^{n}\lambda_ie_i,$  for some $\lambda_i\in \mathbb{F}_p,$ then 
\begin{equation}\label{eq_TrMatch}
    \textup{Tr}_n(v_\eta \zeta)=\textup{Tr}_n\left(\sum_{i=1}^n\lambda_i(v_i-a)\eta+b\zeta\right).
\end{equation}
From \eqref{classificationa_1} and \eqref{classification2to4fourequation}--\eqref{eq_TrMatch}, it follows that
\begin{align}\begin{split}\label{almostcharacterizationclassrho}
\rho_\star(\eta(1,w)&+\zeta(1,-w))
=\rho_\star(\eta(1,w))\frac{\rho_\star(\eta(1,w)+\zeta(1,-w))}{\rho_\star(\eta(1,w))}\\
&=\rho_\star(\eta(1,w))\exp \frac{2\pi i \textup{Tr}_n(v_{\eta}\zeta)}{p}\\
&=\exp \frac{2\pi i \textup{Tr}_n(a \eta)}{p}\exp \frac{2\pi i \textup{Tr}_n(\sum_{i=1}^n\lambda_i(v_i-a)\eta+b\zeta)}{p}\\
&=\exp \frac{2\pi i \textup{Tr}_n(a \eta+ b \zeta +L(\zeta)\eta)}{p},
\end{split}
\end{align}
where  
$L:\mathbb{F}_q\to \mathbb{F}_q$ is the $\mathbb{F}_p$-linear map whose matrix representation  with respect to the basis $\{e_1,\dots, e_n\}$ has columns  $(v_i-a)_{i=1}^n$.
We proceed to investigate the map $L$, with the goal of showing that $L(\zeta)=L(1)\zeta$, for every $\zeta\in\mathbb F_q$.  

From \eqref{eq_CasesEq1111}, 
given $(\boldsymbol{\xi},s)\in\mathbb F_q^2\times\mathbb F_q^\times$, we know that
\begin{equation}\label{eq_EqCond112}
    {f_\star}(\boldsymbol{\xi}_1){f_\star}(\boldsymbol{\xi}-\boldsymbol{\xi}_1)=C(\boldsymbol{\xi},s), \text{ for every } \boldsymbol{\xi}_1\in \mathcal{S}\left(\tfrac{\boldsymbol{\xi}}{2},s\right).
\end{equation}
 Writing $\boldsymbol{\xi}=\eta(1,w)+\zeta(1,-w)$ and $\boldsymbol{\xi}_1=\eta_1(1,w)+\zeta_1(1,-w)$,  condition \eqref{eq_EqCond112} can be rewritten in terms of the function $\rho_\star$ as follows:
 \begin{align}\label{equationcircleclass2to4}
 \rho_\star((\eta-\eta_1)(1,w)+(\zeta-\zeta_1)(1,-w))
 \rho_\star(\eta_1(1,w)+\zeta_1(1,-w))
 =\exp \frac{2\pi i C(\eta,\zeta,s)}p,
 \end{align} 
 whenever  $s=((\eta_1-\frac{\eta}2)(1,w)+(\zeta_1-\frac{\zeta}2)(1,-w))^2=(2\eta_1-\eta)(2\zeta_1-\zeta)$ is nonzero. 
 From \eqref{almostcharacterizationclassrho} and \eqref{equationcircleclass2to4}, we then have 
\begin{align*}
C(\eta,\zeta,s)=&\textup{Tr}_n(a \eta_1+b \zeta_1+L(\zeta_1)\eta_1)
+\textup{Tr}_n(a(\eta-\eta_1)+b(\zeta-\zeta_1)+L(\zeta-\zeta_1)(\eta-\eta_1))\\
=&\textup{Tr}_n( a \eta+b \zeta+L(\zeta_1)\eta_1+L(\zeta-\zeta_1)(\eta-\eta_1)),
\end{align*}
whenever $(2\eta_1-\eta)(2\zeta_1-\zeta)=s$  is nonzero. 
Noting that 
$$L(\zeta_1)\eta_1+L(\zeta-\zeta_1)(\eta-\eta_1)-\frac12 L(\zeta)\eta
=\frac12 L(\zeta-2\zeta_1)(\eta-2\eta_1),$$ 
we then have 
\begin{equation}\label{eq_alnmostThere}
    \textup{Tr}_n(L(\zeta-2\zeta_1)(\eta-2\eta_1))=C(\eta,\zeta,s),
\end{equation}
whenever $(2\eta_1-\eta)(2\zeta_1-\zeta)=s$  is nonzero. 
Writing $u=\zeta-2\zeta_1$, the latter constraint becomes $\eta-2\eta_1=s/u$, and so \eqref{eq_alnmostThere} boils down to
\begin{equation}\label{eq_TrEqCus}
    \textup{Tr}_n\left(s{L(u)}/{u}\right)=C(\eta,\zeta,s), 
\end{equation}
where $u\neq 0$ is now a free variable. In particular, the left-hand side of \eqref{eq_TrEqCus} and therefore its right-hand side are actually independent  of $\eta,\zeta.$
Therefore the choice $u=1$ leads to
\begin{align}\label{characterizationrhofinal}
\textup{Tr}_n\left(s\left({L(u)}/{u}-L(1)\right)\right)=0, \text{ for every } s,u\in \mathbb{F}_q^\times.
\end{align}
From Lemma \ref{lemmaxistenceclasstr}, it then follows that $L(u)=L(1)u$, for every $u\in\mathbb F_q$. 
This concludes the proof of Theorem \ref{thm2}.

\section{Proof of Theorem \ref{thm3}}\label{sec_thm3}

In this section, we prove  Theorem \ref{thm3}.
In view of Proposition \ref{prop_counting}, we aim to verify that
\begin{equation}\label{eq_countingL2L6Parab}
\sum_{(\xi,\tau)\in \mathbb{F}_q^2}\left|\sum_{\substack{(\xi_1,\xi_2,\xi_3)\in\Sigma_{\mathbb P^1}^3(\xi,\tau)}}f(\xi_1)f(\xi_2)f(\xi_3)\right|^2\le \left(q+1-\frac{1}{q}\right)\left(\sum_{\mathbb{P}^1}|f|^2\right)^3, 
\end{equation}
for every function $f:\mathbb{P}^1\to \mathbb{C}$. Here, $\sum_{\mathbb{P}^1}|f|^2:=\sum_{\xi\in\mathbb F_q}|f(\xi,\xi^2)|^2$.
As in \S\ref{sec_thm1}, we  split the sum on the left-hand side of \eqref{eq_countingL2L6Parab} into the contribution from the critical curve, 
\[\mathcal{C}_1:=\left\{\left(\xi,\tfrac{\xi^2}{3}\right):\,\xi\in \mathbb{F}_q\right\},\]
and from its complement, $\mathbb{F}_q^2\setminus \mathcal{C}_1$. 
By Proposition \ref{propositionparaboloidfinitefields}, the cardinality of $\Sigma_{\mathbb P^1}^3(\xi,\tau)$ is constant in each of these sets. The sum over $\mathbb{F}_q^2\setminus \mathcal{C}_1$ can  be controlled by a direct application of  Cauchy--Schwarz:
\begin{align}\label{firstcauchyschwarz2to6}
\begin{split}
\sum_{(\xi,\tau)\in \mathbb{F}_q^2}\left|\sum_{(\xi_1,\xi_2,\xi_3)\in \Sigma_{\mathbb P^1}^3(\xi,\tau)}f(\xi_1)f(\xi_2)f(\xi_3)\right|^2
\le \sum_{(\xi,\tau)\in \mathcal{C}_1}\left|\sum_{(\xi_1,\xi_2,\xi_3)\in \Sigma_{\mathbb P^1}^3(\xi,\tau)}f(\xi_1)f(\xi_2)f(\xi_3)\right|^2\\
+\sum_{(\xi,\tau)\in \mathbb F_q^2\setminus\mathcal{C}_1}\left|\Sigma_{\mathbb P^1}^3(\xi,\tau)\right|\sum_{(\xi_1,\xi_2,\xi_3)\in \Sigma_{\mathbb P^1}^3(\xi,\tau)}|f(\xi_1)f(\xi_2)f(\xi_3)|^2.
\end{split}
\end{align}
The critical curve $\mathcal{C}_1$ requires a  more delicate analysis depending on the geometry of the sets  $\Sigma_{\mathbb P^1}^3(\xi,\tau)$, which we proceed to explore.  \\

Let $(\xi,\tau)\in \mathcal{C}_1$ and $(\xi_1,\xi_2,\xi_3)\in\Sigma_{\mathbb P^1}^3(\xi,\tau)$. 
From the proof of Proposition \ref{propositionparaboloidfinitefields}, we can write $$(\xi_1,\xi_2,\xi_3)=\frac13({\xi},{\xi},{\xi})+(\eta_1,\eta_2,-\eta_1-\eta_2), \textup{ for some } \eta_1,\eta_2\in\mathbb F_q,$$
 in which case
 identity \eqref{eq_ToCount} boils down to
\begin{equation}\label{eq_quadraticEta}
    \eta_1^2+\eta_1\eta_2+\eta_2^2=0.
\end{equation}
If $q\equiv 2(\textup{mod }3)$, then $-3$ is not a square in $\mathbb F_q$ (recall Lemma \ref{lem_squaresmodQ}) and equation \eqref{eq_quadraticEta} has no nonzero solutions.
If $q\equiv 1(\textup{mod }3)$, then $-3$ is a square in $\mathbb F_q$, and the solutions of \eqref{eq_quadraticEta} can be parametrized by
\[(\eta_1,\eta_2)=\ell(j,1),\text{ where }\ell\in\mathbb F_q^\times\text{ and }j^2+j+1=0.\]
The analysis thus splits into two cases, depending on the congruence class of $q$ modulo 3.

\subsection{The case $q\equiv 2(\textup{mod }3)$}
In this case, $\Sigma_{\mathbb P^1}^3(\xi,\tau)=\left\{\frac13\left({\xi},{\xi},{\xi}\right)\right\}$ whenever $(\xi,\tau)\in\mathcal C_1$. This simple structure simplifies the analysis significantly.
Invoking \eqref{eq_SigmaSetCount}, the right-hand side of\eqref{firstcauchyschwarz2to6} can then be bounded by
\begin{align}\label{q=2mod32to6}
\begin{split}
\sum_{(\xi,\tau)\in \mathcal{C}_1}\left|f\left(\tfrac{\xi}{3}\right)\right|^6&+(q+1)\sum_{(\xi,\tau)\in \mathbb F_q^2\setminus\mathcal{C}_1}\sum_{(\xi_1,\xi_2,\xi_3)\in \Sigma_{\mathbb P^1}^3(\xi,\tau)}|f(\xi_1)f(\xi_2)f(\xi_3)|^2\\
&=(q+1)\sum_{(\xi,\tau)\in \mathbb{F}_q^2}\sum_{(\xi_1,\xi_2,\xi_3)\in\Sigma_{\mathbb Pª^1}^3(\xi,\tau)}|f(\xi_1)f(\xi_2)f(\xi_3)|^2-q\sum_{(\xi,\tau)\in \mathcal{C}_1}\left|f\left(\tfrac{\xi}{3}\right)\right|^6.
\end{split}
\end{align}
Interchanging the order of summation as in \eqref{eq_constant1}, we have that
\begin{equation}\label{eq_InterchageOrder}
 \sum_{(\xi,\tau)\in \mathbb{F}_q^2}\sum_{(\xi_1,\xi_2,\xi_3)\in \Sigma_{\mathbb P^1}^3(\xi,\tau)}|f(\xi_1)f(\xi_2)f(\xi_3)|^2=\left(\sum_{\mathbb P^1}|f|^2\right)^3,      
\end{equation}
On the other hand, Hölder's inequality yields
\begin{equation}\label{eq_CSno1inthisproof}
\sum_{(\xi,\tau)\in \mathcal{C}_1}\left|f\left(\tfrac{\xi}{3}\right)\right|^6=\sum_{\mathbb P^1}|f|^6
\ge \frac{1}{q^2}\left(\sum_{\mathbb P^1}|f|^2\right)^3.     
\end{equation}
Combining \eqref{eq_InterchageOrder}--\eqref{eq_CSno1inthisproof} with \eqref{firstcauchyschwarz2to6} and \ref{q=2mod32to6}, we obtain the sharp inequality \eqref{eq_countingL2L6Parab}.
As before, maximizers necessarily have constant modulus.
This concludes the proof of Theorem \ref{thm3} when $q\equiv2(\textup{mod }3)$. 

\subsection{The case $q\equiv 1(\textup{mod }3)$}
In this case, the right-hand side of\eqref{firstcauchyschwarz2to6} can  be bounded by
\begin{align}
\sum_{(\xi,\tau)\in \mathcal{C}_1}&\left|\sum_{(\xi_1,\xi_2,\xi_3)\in \Sigma_{\mathbb P^1}^3(\xi,\tau)}f(\xi_1)f(\xi_2)f(\xi_3)\right|^2+(q-1)\sum_{(\xi,\tau)\in \mathbb F_q^2\setminus\mathcal{C}_1}\sum_{(\xi_1,\xi_2,\xi_3)\in \Sigma_{\mathbb P^1}^3(\xi,\tau)}|f(\xi_1)f(\xi_2)f(\xi_3)|^2\notag\\
&=\sum_{(\xi,\tau)\in \mathcal{C}_1}\left(\left|\sum_{(\xi_1,\xi_2,\xi_3)\in \Sigma_{\mathbb P^1}^3(\xi,\tau)}f(\xi_1)f(\xi_2)f(\xi_3)\right|^2
-(q-1)\sum_{(\xi_1,\xi_2,\xi_3)\in \Sigma_{\mathbb P^1}^3(\xi,\tau)}\left|f(\xi_1)f(\xi_2)f(\xi_3)\right|^2\right)\label{firstcauchyschwarzq=12to6}\\
&+(q-1)\left(\sum_{\mathbb P^1}|f|^2\right)^3.\notag
\end{align}

\begin{remark}\label{remarkexamplediracdelta}
In order to handle \eqref{firstcauchyschwarzq=12to6}, it is not enough to use Cauchy--Schwarz directly. In light of \eqref{eq_SigmaSetCount0},  $|\Sigma_{\mathbb P^1}^3(\xi,\tau)|=2q-1$ for all $(\xi,\tau)\in\mathcal{C}_1$, and so  the resulting upper bound would be
 $$q\sum_{(\xi,\tau)\in \mathcal{C}_1}\sum_{(\xi_1,\xi_2,\xi_3)\in \Sigma_{\mathbb P^1}^3(\xi,\tau)}|f(\xi_1)f(\xi_2)f(\xi_3)|^2,$$
 which is not bounded by the value attained by constant functions, $\left(2-{1}/{q}\right)\left(\sum_{\mathbb P^1}|f|^2\right)^3$.
To see this, it suffices to consider the case when $f=\delta_0$ is  a Dirac delta at the origin. 
\end{remark} 
We proceed to analyze \eqref{firstcauchyschwarzq=12to6}, which coincides with  \eqref{criticalcurveintro} when $(d,k)=(1,3)$.
We will  prove that it is maximized by constants for fixed $\sum_{\mathbb P^1}|f|^2$, via  the following six steps.\\

{\it Step 1: Line decomposition.} 
Let $j_\pm$ denote the two distinct roots of the polynomial $j^2+j+1$ in $\mathbb{F}_q$. 
If $(\xi,\tau)\in\mathcal C_1$, then
$\Sigma_{\mathbb P^1}^3(\xi,\tau)$
is the union of the two lines
\[\mathcal L_\pm(\xi):=\{\tfrac13(\xi,\xi,\xi)+\ell(j_\pm,1,-1-j_\pm):\, \ell\in\mathbb F_q\},\]
which intersect exactly at $\frac13(\xi,\xi,\xi).$
Given  $g:\mathbb{F}_q\to \mathbb{C}$, a key observation which is particular to this $L^2\to L^6$ setting is that
\begin{align}\label{permutation2to6}
\sum_{(\xi_1,\xi_2,\xi_3)\in \mathcal{L}_+(\xi)}g(\xi_1)g(\xi_2)g(\xi_3)=\sum_{(\xi_1,\xi_2,\xi_3)\in \mathcal{L}_-(\xi)}g(\xi_1)g(\xi_2)g(\xi_3).    
\end{align}
Indeed, every element of $\mathcal{L}_-(\xi)$ is a permutation of an element of $\mathcal{L}_+(\xi)$ since $j_-j_+=1$ implies $j_-(j_+,1,-1-j_+)=(1,j_-,-1-j_-)$.
Writing $\mathcal{L}_\pm^\circ(\xi):=\mathcal{L}_\pm(\xi)\setminus \{\frac13({\xi},{\xi},{\xi})\}$, the term inside the outer sum in \eqref{firstcauchyschwarzq=12to6} then equals
\begin{equation}
    \label{reductionbysymmetry2to6}
\left|2\sum_{(\xi_1,\xi_2,\xi_3)\in \mathcal{L}_+^\circ(\xi)}f(\xi_1)f(\xi_2)f(\xi_3)+f\left(\tfrac{\xi}{3}\right)^3\right|^2 
-(q-1)\left(2\sum_{(\xi_1,\xi_2,\xi_3)\in \mathcal{L}_+^\circ(\xi)}\left|f(\xi_1)f(\xi_2)f(\xi_3)\right|^2+\left|f\left(\tfrac{\xi}{3}\right)\right|^6\right),
\end{equation}
which is the same as 
\begin{align}\label{reductionbysymmetry2to62}
\begin{split}
&4\left|\sum_{(\xi_1,\xi_2,\xi_3)\in \mathcal{L}_+^\circ(\xi)}f(\xi_1)f(\xi_2)f(\xi_3)\right|^2-(q-2)\left|f\left(\tfrac{\xi}{3}\right)\right|^6 \\
&\quad+4\Re\left(\overline{f\left(\tfrac{\xi}{3}\right)}^3\sum_{(\xi_1,\xi_2,\xi_3)\in \mathcal{L}_+^\circ(\xi)}f(\xi_1)f(\xi_2)f(\xi_3)\right)
-2(q-1)\sum_{(\xi_1,\xi_2,\xi_3)\in \mathcal{L}_+^\circ(\xi)}\left|f(\xi_1)f(\xi_2)f(\xi_3)\right|^2.
\end{split}
\end{align}
{\it Step  2: Intermediate inequalities.}
The first summand in \eqref{reductionbysymmetry2to62} can be estimated by Cauchy--Schwarz:
\begin{align*}
 \left|\sum_{(\xi_1,\xi_2,\xi_3)\in \mathcal{L}_+^\circ(\xi)}f(\xi_1)f(\xi_2)f(\xi_3)\right|^2\le (q-1)\sum_{(\xi_1,\xi_2,\xi_3)\in \mathcal{L}_+^\circ(\xi)}\left|f(\xi_1)f(\xi_2)f(\xi_3)\right|^2.
\end{align*}
The third summand in \eqref{reductionbysymmetry2to62} can be estimated via the triangle inequality:
\begin{align*}
 \Re\left(\overline{f\left(\tfrac{\xi}{3}\right)}^3\sum_{(\xi_1,\xi_2,\xi_3)\in \mathcal{L}_+^\circ(\xi)}f(\xi_1)f(\xi_2)f(\xi_3)\right)
 \le \left|f\left(\tfrac{\xi}{3}\right)\right|^3\sum_{(\xi_1,\xi_2,\xi_3)\in \mathcal{L}_+^\circ(\xi)}\left|f(\xi_1)f(\xi_2)f(\xi_3)\right|.
\end{align*}
It follows that \eqref{firstcauchyschwarzq=12to6} is bounded by 
\begin{align}\label{symmetrycauchyschwarz2to6}
\begin{split}
\sum_{(\xi,\tau)\in \mathcal{C}_1}\Bigg(2\left|\sum_{(\xi_1,\xi_2,\xi_3)\in \mathcal{L}_+^\circ(\xi)}f(\xi_1)f(\xi_2)f(\xi_3)\right|^2-(q-2)\left|f\left(\tfrac{\xi}{3}\right)\right|^6 \\
+4\left|f\left(\tfrac{\xi}{3}\right)\right|^3\sum_{(\xi_1,\xi_2,\xi_3)\in \mathcal{L}_+^\circ(\xi)}\left|f(\xi_1)f(\xi_2)f(\xi_3)\right|\Bigg),
\end{split}
\end{align}
with equality if $f$ is constant.\\

{\it Step 3: Analyzing the first term in \eqref{symmetrycauchyschwarz2to6}.}
Interchanging the order of summation, 
\begin{align}\label{fubiniselforthogonal}
\begin{split}
&\sum_{(\xi,\tau)\in \mathcal{C}_1}\left|\sum_{(\xi_1,\xi_2,\xi_3)\in \mathcal{L}_+^\circ(\xi)}f(\xi_1)f(\xi_2)f(\xi_3)\right|^2\\
&\quad=\sum_{\substack{ (\xi_1,\xi_2,\xi_3)\in \mathbb{F}_q^3\\ (\eta_1,\eta_2,\eta_3)\in \mathbb{F}_q^3}}f(\xi_1)f(\xi_2)f(\xi_3)\overline{f(\eta_1)f(\eta_2)f(\eta_3)}m(\xi_1,\xi_2,\xi_3,\eta_1,\eta_2,\eta_3),
\end{split}
\end{align}
where
$$m(\xi_1,\xi_2,\xi_3,\eta_1,\eta_2,\eta_3):=\sum_{(\xi,\tau)\in \mathcal{C}_1}{\bf 1}((\xi_1,\xi_2,\xi_3),(\eta_1,\eta_2,\eta_3)\in \mathcal{L}_+^\circ(\xi)).$$
The function $m$ takes values in $\{0,1\}$, and it equals 1 if and only if there exist $u,v\in \mathbb{F}_q^\times$, such that 
\begin{align}
\left\{
\begin{array}{ll}
\xi_1+\xi_2+\xi_3=\eta_1+\eta_2+\eta_3=:3\zeta\\
(\xi_1,\xi_2,\xi_3)=(\zeta+uj_+,\zeta+u,\zeta-u-uj_+)\\
(\eta_1,\eta_2,\eta_3)=(\zeta+vj_+,\zeta+v,\zeta-v-vj_+)
\end{array} \right. 
\end{align}
We proceed to analyze the set $\mathcal{A}:=m^{-1}(1)$, and claim the existence of (explicit) functions
$\omega_1:\mathbb{F}_q^2\setminus \{(\ell,\ell):\,\ell\in \mathbb{F}_q\}\to \mathbb{F}_q$ and 
$\omega_2,\omega_3:\mathcal{B}\to \mathbb{F}_q$, where $\mathcal B\subset \mathbb F_q^3$ is defined as
\begin{equation}\label{eq_defB}
    \mathcal B:= \left\{(\ell_1,\ell_2,\ell_3)\in\mathbb F_q^3:\, \ell_1\neq \ell_2\, \text{and}\, \ell_3\neq \frac{\ell_2j_+-\ell_1}{j_+-1}\right\},
\end{equation}
such that 
\begin{equation}\label{eq_defAA}
    \mathcal{A}=\{(\xi_1,\xi_2,\omega_1(\xi_1,\xi_2),\eta_1,\omega_2(\xi_1,\xi_2,\eta_1),\omega_3(\xi_1,\xi_2,\eta_1)):\,(\xi_1,\xi_2,\eta_1)\in \mathcal{B}\}.
\end{equation}
Indeed, any non-diagonal pair $(\xi_1,\xi_2)\in\mathbb F_q^2$  
defines a unique {\it center} $\zeta=\zeta(\xi_1,\xi_2)\in \mathbb{F}_q$ and a unique nonzero {\it height} $u=u(\xi_1,\xi_2)\in \mathbb{F}_q^\times$ such that $(\xi_1,\xi_2)=(\zeta+uj_+,\zeta+u)$. 
If $(\xi_1,\xi_2,\xi_3,\eta_1,\eta_2,\eta_3)\in \mathcal{A}$,  then  $\xi_3=\zeta-u-uj_+=:\omega_1(\xi_1,\xi_2).$ 
In particular, $\zeta=\frac{\xi_2j_+-\xi_1}{j_+-1}$. 
Moreover, any $\eta_1\neq \zeta$ defines a unique nonzero height $v=v(\xi_1,\xi_2,\eta_1)\in\mathbb F_q^\times$, such that $\eta_1=\zeta+vj_+$ and $\eta_2=\zeta+v=:\omega_2(\xi_1,\xi_2, \eta_1)$ and $\eta_3=\zeta-v-vj_+=:\omega_3(\xi_1,\xi_2, \eta_1).$
The claim follows, as does the fact that $(\omega_1,\omega_2,\omega_3):\mathcal{B}\to \mathcal{B}'$
defines a bijection between the set $\mathcal  B$ from \eqref{eq_defB} and 
\[\mathcal B':=\left\{(\ell_1,\ell_2,\ell_3)\in\mathbb F_q^3:\,  \ell_2\neq \ell_3\, \text{and}\, \ell_1\neq \frac{\ell_3+\ell_2+\ell_2j_+}{2+j_+}\right\}.\]
Write $(\omega_4,\omega_5,\omega_6)=(\omega_1,\omega_2,\omega_3)^{-1}:\mathcal{B}'\to \mathcal{B}$, and observe that the function $\omega_6$  depends only on the last two coordinates of $\mathcal{B}'$. In fact, given $\eta_2\neq \eta_3$, if $(\xi_1,\xi_2,\eta_1)$ is such that $\omega_2(\xi_1,\xi_2,\eta_1)=\eta_2$ and $\omega_3(\xi_1,\xi_2,\eta_1)=\eta_3$, then $(\eta_1,\eta_2,\eta_3)=(\zeta+vj_+,\zeta+v,\zeta-v-vj_+)$,
and consequently
\begin{align}\label{eq_omega61}
\begin{split}
\omega_6(\xi_3,\eta_2,\eta_3)&=\eta_1=\zeta-(1+j_+)\frac{-vj_+}{1+j_+}=\omega_1\left(\zeta+j_+\frac{-vj_+}{1+j_+}, \zeta+\frac{-vj_+}{1+j_+}\right)\\
&=\omega_1(\zeta+v,\zeta-j_+-vj_+)=\omega_1(\eta_2,\eta_3).
\end{split}
\end{align}
\vspace{.1cm}

{\it Step 4: Bounding \eqref{fubiniselforthogonal}.}
Let $\omega_1=\omega_1(\xi_1,\xi_2)$, $\omega_2=\omega_2(\xi_1,\xi_2,\eta_1)$, $\omega_3=\omega_3(\xi_1,\xi_2,\eta_1)$ be  as in the previous step.
By \eqref{eq_defAA}, the right-hand side of \eqref{fubiniselforthogonal} equals
\begin{align}\begin{split}\label{inequalityCS2to6}
 \sum_{(\xi_1,\xi_2,\eta_1)\in \mathcal{B}}&f(\xi_1)f(\xi_2)f(\omega_1)\overline{f(\eta_1)f(\omega_2)f(\omega_3)}\\
 &\le \frac{1}{2}  \left(\sum_{(\xi_1,\xi_2,\eta_1)\in \mathcal{B}}|f(\xi_1)f(\xi_2)f(\eta_1)|^2+|f(\omega_1)f(\omega_2)f(\omega_3)|^2\right)\\
 &=\frac{1}{2}\left(\sum_{(\xi_1,\xi_2,\eta_1)\in \mathcal{B}}|f(\xi_1)f(\xi_2)f(\eta_1)|^2+\sum_{(\xi_3,\eta_2,\eta_3)\in \mathcal{B}'}|f(\xi_3)f(\eta_2)f(\eta_3)|^2\right).
 \end{split}
\end{align}
Since $\zeta(\xi_1,\xi_2)=\frac13(\xi_1+\xi_2+\omega_1(\xi_1,\xi_2))$, we  have that
\begin{align}\label{inequalityB2to6}
\begin{split}
    & \sum_{(\xi_1,\xi_2,\eta_1)\in \mathcal{B}}|f(\xi_1)f(\xi_2)f(\eta_1)|^2=\sum_{(\xi_1,\xi_2,\eta_1)\in \mathbb{F}_q^3}|f(\xi_1)f(\xi_2)f(\eta_1)|^2\\
 &\quad-\sum_{(\xi_1,\eta_1)\in \mathbb{F}_q^2}|f(\xi_1)|^4|f(\eta_1)|^2-\sum_{\xi_1\neq\xi_2}\left|f(\xi_1)f(\xi_2)f\left(\tfrac{\xi_1+\xi_2+\omega_1(\xi_1,\xi_2)}{3}\right)\right|^2.   
\end{split}
\end{align}
Invoking the fact that $\omega_6=\omega_1$, recall \eqref{eq_omega61}, we further have that
\begin{align}\label{inequalityBprime2to6}
\begin{split}
& \sum_{(\xi_3,\eta_2,\eta_3)\in \mathcal{B}'}|f(\xi_3)f(\eta_2)f(\eta_3)|^2=\sum_{(\xi_3,\eta_2,\eta_3)\in \mathbb{F}_q^3}|f(\xi_3)f(\eta_2)f(\eta_3)|^2\\
&\quad-\sum_{(\xi_3,\eta_2)\in \mathbb{F}_q^2}|f(\eta_2)|^4|f(\xi_3)|^2
-\sum_{\eta_2\neq \eta_3}\left|f(\eta_2)f(\eta_3)f\left(\tfrac{\eta_2+\eta_3+\omega_1(\eta_2,\eta_3)}{3}\right)\right|^2.
\end{split}
\end{align}
Estimates \eqref{inequalityCS2to6}--\eqref{inequalityBprime2to6} together imply that \eqref{fubiniselforthogonal} is bounded by
\begin{align}\begin{split}\label{finalmathcalA2to6}
\left(\sum_{\mathbb P^1}|f|^2\right)^3-\left(\sum_{\mathbb P^1}|f|^2\right)\left(\sum_{\mathbb P^1}|f|^4\right)
-\sum_{\xi_1\neq\xi_2}\left|f(\xi_1)f(\xi_2)f\left(\tfrac{\xi_1+\xi_2+\omega_1(\xi_1,\xi_2)}{3}\right)\right|^2,
\end{split}
\end{align}
with equality if $f$ is constant.\\

{\it Step 5: Bounding the last term in \eqref{symmetrycauchyschwarz2to6}.}
Interchanging the order of summation, 
\begin{align}\begin{split}\label{step72to6}
\sum_{(\xi,\tau)\in \mathcal{C}_1}&\left|f\left(\tfrac{\xi}{3}\right)\right|^3\sum_{(\xi_1,\xi_2,\xi_3)\in \mathcal{L}_+^\circ(\xi)}\left|f(\xi_1)f(\xi_2)f(\xi_3)\right|\\
&=\sum_{(\xi_1,\xi_2,\xi_3)\in \mathbb{F}_q^{3}}|f(\xi_1)f(\xi_2)f(\xi_3)|\sum_{(\xi,\tau)\in \mathcal{C}_1}\left|f\left(\tfrac{\xi}{3}\right)\right|^3{\bf 1}((\xi_1,\xi_2,\xi_3)\in \mathcal{L}_+^\circ(\xi)).
\end{split}
\end{align}
By the previous steps, $(\xi_1,\xi_2,\xi_3)\in \mathcal{L}_{+}^\circ(\xi)$ if and only if $\xi=\xi_1+\xi_2+\omega_1(\xi_1,\xi_2)$ and $\xi_1\neq \xi_2$ and $\xi_3=\omega_1(\xi_1,\xi_2)$. Therefore, \eqref{step72to6} equals
\begin{align}\label{step72to6final}
\sum_{\xi_1\neq \xi_2}|f(\xi_1)f(\xi_2)f(\omega_1(\xi_1,\xi_2))|
\left|f\left(\tfrac{\xi_1+\xi_2+\omega_1(\xi_1,\xi_2)}{3}\right)\right|^3.
\end{align}
\vspace{.1cm}

{\it Step 6: End of proof.}
The bounds \eqref{finalmathcalA2to6} and \eqref{step72to6final} combine in \eqref{symmetrycauchyschwarz2to6} to yield the following upper bound for \eqref{firstcauchyschwarzq=12to6}:
\begin{align}\begin{split}\label{assemble2to6}
& 2\left(\sum_{\mathbb P^1}|f|^2\right)^3
-2\left(\sum_{\mathbb P^1}|f|^2\right)\left(\sum_{\mathbb P^1}|f|^4\right)-2\sum_{\xi_1\neq\xi_2}\left|f(\xi_1)f(\xi_2)f\left(\tfrac{\xi_1+\xi_2+\omega_1(\xi_1,\xi_2)}{3}\right)\right|^2\\
 &-(q-2)\sum_{\mathbb P^1}|f|^6+4\sum_{\xi_1\neq \xi_2}
|f(\xi_1)f(\xi_2)f(\omega_1(\xi_1,\xi_2))|\left|f\left(\tfrac{\xi_1+\xi_2+\omega_1(\xi_1,\xi_2)}{3}\right)\right|^3,
\end{split}
\end{align}
with equality when $f$ is constant. 
We further have that 
\begin{align}\label{mamgstep82to6}
\begin{split}
&2\sum_{\xi_1\neq \xi_2}|f(\xi_1)f(\xi_2)f(\omega_1(\xi_1,\xi_2))|\left|f\left(\tfrac{\xi_1+\xi_2+\omega_1(\xi_1,\xi_2)}{3}\right)\right|^3\\
&\quad\le \sum_{\xi_1\neq \xi_2}\left|f(\xi_1)f(\xi_2)f\left(\tfrac{\xi_1+\xi_2+\omega_1(\xi_1,\xi_2)}{3}\right)\right|^2
+\sum_{\xi_1\neq \xi_2}|f(\omega_1(\xi_1,\xi_2))|^2\left|f\left(\tfrac{\xi_1+\xi_2+\omega_1(\xi_1,\xi_2)}{3}\right)\right|^4.
\end{split}
\end{align}
The map $(\xi_1,\xi_2)\mapsto \left(\omega_1(\xi_1,\xi_2),\tfrac13(\xi_1+\xi_2+\omega_1(\xi_1,\xi_2))\right)$
is a bijection from the set $\mathbb{F}_q^2\setminus \{(\ell,\ell):\, \ell\in \mathbb{F}_q\}$ onto itself. 
Indeed,  $\zeta=\frac{1}{3}(\xi_1+\xi_2+\omega_1(\xi_1,\xi_2))$ and  $(\xi_1,\xi_2,\omega_1(\xi_1,\xi_2))=(\zeta+uj_+,\zeta+u,\zeta-(1+j_+)u)$. Knowing $\omega_1(\xi_1,\xi_2)$ and $\frac{1}{3}(\xi_1+\xi_2+\omega_1(\xi_1,\xi_2))$, we thus recover  $u, \zeta$ and  $\xi_1, \xi_2$. It follows that  the map in question is injective, and therefore a bijection. Hence
$$\sum_{\xi_1\neq \xi_2}|f(\omega_1(\xi_1,\xi_2))|^2\left|f\left(\tfrac{\xi_1+\xi_2+\omega_1(\xi_1,\xi_2)}{3}\right)\right|^4=\sum_{\xi_1\neq \xi_2}|f(\xi_1)|^2|f(\xi_2)|^4.$$
This identity together with \eqref{mamgstep82to6} implies that  \eqref{assemble2to6} is bounded by 
\begin{align*}
2\left(\sum_{\mathbb P^1}|f|^2\right)^3-2\left(\sum_{\mathbb P^1}|f|^2\right)\left(\sum_{\mathbb P^1}|f|^4\right)&-(q-2)\sum_{\mathbb P^1}|f|^6+2\sum_{\xi_1\neq \xi_2}|f(\xi_1)|^2|f(\xi_2)|^4\\
&=2\left(\sum_{\mathbb P^1}|f|^2\right)^3-q\sum_{\mathbb P^1}|f|^6.
\end{align*}
A final application of  Hölder's inequality as in \eqref{eq_CSno1inthisproof} yields the sharp inequality 
\eqref{eq_countingL2L6Parab}, and maximizers again  have constant modulus.
This concludes the proof of Theorem \ref{thm3} when $q\equiv1(\textup{mod }3)$.

\section{Proof of Theorem \ref{thm4}}\label{sec_Thm4}

The proof of Theorem \ref{thm4} parallels that of Theorem \ref{thm1} when $q\equiv 1(\textup{mod }4)$ 
and that of Theorem  \ref{thm2}, and so
we only highlight the necessary changes.

Let $q=p^n$ be an arbitrary power of an odd prime. Firstly, the two-fold convolution of normalized counting measure $\sigma=\sigma_{\mathbb H^2}$ is given by
\begin{equation}\label{eq_2foldconvHyp}
    (\sigma\ast\sigma)(\boldsymbol{\xi},\tau)=\frac1q\times\left\{ \begin{array}{ll}
2q-1, & \text{ if } \tau=\frac{\boldsymbol{\xi}\odot\boldsymbol{\xi}}{2},\\
q-1, & \text{ otherwise, }
\end{array} \right.
\end{equation}
where $\boldsymbol{\xi}\odot\boldsymbol{\xi}:=\xi_1^2-\xi_2^2$ if $\boldsymbol{\xi}=(\xi_1,\xi_2)\in\mathbb F_q^2.$
Secondly, inequality \eqref{eq_sharpineq4} is equivalent to 
\begin{equation}\label{eq_countingL2L4Hyp}
\sum_{(\boldsymbol{\xi},\tau)\in \mathbb{F}_q^3}\left|\sum_{\boldsymbol{\xi}_1\in \mathcal{H}\left(\frac{\boldsymbol{\xi}}{2},\frac{2\tau-\boldsymbol{\xi}\odot\boldsymbol{\xi}}{4}\right)}{f}(\boldsymbol{\xi}_1){f}(\boldsymbol{\xi}-\boldsymbol{\xi}_1)\right|^2\le \left(q+1-\frac{1}{q}\right)\left(\sum_{\mathbb{H}^2}|f|^2\right)^2,    
\end{equation}
where $\sum_{\mathbb{H}^2}|f|^2:=\sum_{\boldsymbol{\xi}\in\mathbb F_q^2}|f(\boldsymbol{\xi},\boldsymbol{\xi}\odot\boldsymbol{\xi})|^2$ and, 
given $(\boldsymbol{\gamma},s)\in\mathbb F_q^d\times\mathbb F_q$, we define the {\it saddle}
\begin{equation}\label{eq_BallDef}
    \mathcal H(\boldsymbol{\gamma},s):=\{\boldsymbol{\eta}\in\mathbb F_q^d:(\boldsymbol{\gamma}-\boldsymbol{\eta})\odot(\boldsymbol{\gamma}-\boldsymbol{\eta})=s\}.
\end{equation}
Thirdly, the critical surface is now  $\widetilde{\mathcal C}_2:=\{(\boldsymbol{\xi},\tau)\in\mathbb F_q^3:\, 2\tau=\boldsymbol{\xi}\odot\boldsymbol{\xi}\}$ and, given $(\boldsymbol{\xi},\tau)\in\widetilde{\mathcal C}_2$,  the saddle $\mathcal{H}(\frac{\boldsymbol{\xi}}{2},0)$  is the union of the two lines 
\begin{equation}\label{eq_HypLines}
    \widetilde{\mathcal L}_\pm(\boldsymbol{\xi}):=\left\{\boldsymbol{\xi}_1\in \mathbb{F}_q^2:\, \boldsymbol{\xi}_1=\frac{\boldsymbol{\xi}}{2}+t(1,\pm 1),\, t\in \mathbb{F}_q\right\},
\end{equation}
which intersect exactly at ${\boldsymbol{\xi}}/{2}$.
The rest of the argument goes through as in \S\ref{sec_41} without further changes, leading to the sharp inequality \eqref{eq_sharpineq4}.

The characterization of maximizers follows the same steps as the ones in \S\ref{sec_thm2}. From the  proof outlined in the previous paragraph, any maximizer $f_{\star}$ of \eqref{eq_countingL2L4Hyp} has constant modulus, whence $f_{\star}=\lambda \rho_{\star}$ with $\rho_{\star}:\mathbb{F}_q^2\to \mathbb{S}^1$ and $\lambda\in \mathbb{C}\setminus{0}$.  
From Lemma \ref{middlepointlemma2to4} and the functional equation derived from the cases of equality in \eqref{eq_countingL2L4Hyp}, it follows that $\rho_{\star}$ is a character over any line $\widetilde{\mathcal L}_\pm(\boldsymbol{\xi})$.
We are then able to conclude  that there exist unique $a,b\in\mathbb F_q$, such that
\begin{align}\label{Lhiperboloid}
\rho_\star(\eta(1,1)+\zeta(1,-1))=\exp \frac{2\pi i \textup{Tr}_n(a \eta+ b \zeta +\widetilde L(\zeta)\eta)}{p}, \textup{ for every } \eta,\zeta\in \mathbb{F}_q,
\end{align}
for a certain $\mathbb{F}_p$-linear map $\widetilde L:\mathbb{F}_q\to \mathbb{F}_q$. We want to verify that $\widetilde L(\zeta)=\widetilde L(1)\zeta$, for all $\zeta\in\mathbb F_q$. From the equality cases of  the intermediate inequalities required for \eqref{eq_countingL2L4Hyp}, we obtain  
 \begin{align}\label{equationcircleclass2to4hiperboloid}
 \rho_\star((\eta-\eta_1)(1,1)+(\zeta-\zeta_1)(1,-1))
 \rho_\star(\eta_1(1,1)+\zeta_1(1,-1))
 =\exp \frac{2\pi i C(\eta,\zeta,s)}p
 \end{align} 
 whenever  $s=\boldsymbol{\xi} \odot\boldsymbol{\xi}=(2\eta_1-\eta)(2\zeta_1-\zeta)$ is nonzero, where $\boldsymbol{\xi}=(\eta_1-\frac{\eta}2)(1,1)+(\zeta_1-\frac{\zeta}2)(1,-1)$. From \eqref{Lhiperboloid} and \eqref{equationcircleclass2to4hiperboloid}, it follows that
 \begin{align*}
C(\eta,\zeta,s)=&\textup{Tr}_n(a \eta_1+b \zeta_1+\widetilde L(\zeta_1)\eta_1)
+\textup{Tr}_n(a(\eta-\eta_1)+b(\zeta-\zeta_1)+\widetilde L(\zeta-\zeta_1)(\eta-\eta_1))\\
=&\textup{Tr}_n( a \eta+b \zeta+\widetilde L(\zeta_1)\eta_1+\widetilde L(\zeta-\zeta_1)(\eta-\eta_1))
\end{align*}
whenever $s=(2\eta_1-\eta)(2\zeta_1-\zeta)$ is nonzero.
From this point onwards, the proof follows that of Theorem \ref{thm2} line by line.
This concludes the proof of Theorem \ref{thm4}.

\section{Proof of Theorem \ref{thm5}}\label{sec_Thm5}
In this section, we prove Theorem \ref{thm5}.
Let $q\equiv 3(\text{mod}\,4)$.
From Lemma \ref{numberofpointsonthecone}, it follows that $|\Gamma^3|=(q-1)(q^2+1)$.
In view of Proposition \ref{prop_counting}, we aim to establish the sharp inequality
\begin{align}\label{sharpinequalityconic}
\sum_{\boldsymbol{\eta}\in \mathbb{F}_q^4}\left| \sum_{\substack{\boldsymbol{\eta}_1,\boldsymbol{\eta}_2\in \Gamma^3\\ \boldsymbol{\eta}_1+\boldsymbol{\eta}_2=\boldsymbol{\eta}}}f(\boldsymbol{\eta}_1)f(\boldsymbol{\eta}_2)\right|^2\le \frac{q^5-2q^4+2 q^3-3q+3}{(q-1)(q^2+1)}\left(\sum_{ \Gamma^3}|f|^2\right)^2,    
\end{align}
for every function $f:\Gamma^3\to \mathbb{C}$. Here, $\sum_{\Gamma^3}|f|^2:=\sum_{\tau\sigma=\boldsymbol{\xi}^2}|f(\boldsymbol{\xi},\tau,\sigma)|^2$.\\

 Our approach can be summarized as follows.
We  decompose $\mathbb{F}^4_q$ into the three disjoint subsets where 
the two-fold convolution  is constant:
 $\Gamma^3$, $\{{\bf 0}\}$, and $\mathbb F_q^4\setminus\Gamma_0^3$; recall Proposition \ref{propconiconvolutionfinitefields}. 
A direct application of the Cauchy--Schwarz inequality suffices to handle the  complement of ${\Gamma}_0^3$. Points in the cone $\Gamma^3$ require knowledge  of the preimages of the corresponding two-fold convolution, combined with Cauchy--Schwarz. Crucially, these preimages correspond to disjoint punctured lines that folliate the cone. The contribution from the origin is dealt with in a similar way, taking into account the higher number of antipodal pairs.
We proceed to establish  \eqref{sharpinequalityconic} in the course of the following four steps.\\

{\it Step 1: Slicing the cone.} 
Define the sets
\begin{align*}
S_1&:=\{(\boldsymbol{\xi},\tau, \sigma)\in \mathbb F_q^4:\, \boldsymbol{\xi}^2=\tau\sigma=1\},\\
S_2&:=\{(\boldsymbol{\xi},\tau,\sigma)\in\mathbb F_q^4:\, \boldsymbol{\xi}^2=\tau\sigma=-1\}.
\end{align*}
Given $i\in\{1,2\}$, let  $S^*_i\subset S_i$ be such that, for each pair $\{\boldsymbol{\eta},-\boldsymbol{\eta}\}\subset S_i$, one and only one element of $\{\boldsymbol{\eta},-\boldsymbol{\eta}\}$ belongs to $S^*_i$. 
Further define $S^*_3:=\{(\boldsymbol{0},1,0),(\boldsymbol{0},0,1)\}$. 
Each point $\boldsymbol{\eta}\in S^*_i,$ $i\in\{1,2,3\}$,  defines a punctured line $\mathcal{L}_{\boldsymbol{\eta}}:=\{\alpha\boldsymbol{\eta}:\,\alpha\in \mathbb{F}_q^\times\}$, and lines corresponding to distinct points do not intersect. Indeed,  if $(\boldsymbol{\xi}_1,\tau_1,\sigma_1)=\boldsymbol{\eta}_1\in S^*_1$ and $(\boldsymbol{\xi}_2,\tau_2,\sigma_2)=\boldsymbol{\eta}_2\in S^{*}_2$, then $(\alpha\boldsymbol{\xi}_1)^2$ is a square in $\mathbb{F}_q$, whereas $\boldsymbol{\xi}_2^2$ is not. 
Moreover, if $\boldsymbol{\eta}_1\neq \boldsymbol{\eta}_2$ are such that $\boldsymbol{\eta}_1,\boldsymbol{\eta}_2 \in S_1$ and $\alpha\boldsymbol{\xi}_1=\boldsymbol{\xi}_2$, then $\alpha^2=\alpha^2\boldsymbol{\xi}_1^2=\boldsymbol{\xi}_2^2=1$. Thus $\alpha=-1$ and  $\{\boldsymbol{\eta}_1,-\boldsymbol{\eta}_1\}\subset S^*_1$, which is absurd. The case of $\boldsymbol{\eta}_1, \boldsymbol{\eta}_2\in S_2$ is analogous. The disjointness of the lines generated by the elements of $S^*_3$ is immediate. 
Letting $S^*=S^*_1\cup S^*_2\cup S^*_3$, we then have that $\Gamma^3$ equals the disjoint union of all  punctured lines indexed by elements of $S^*$,
\begin{align}\label{slicingthecone}
\Gamma^3=\bigcup_{\boldsymbol{\eta}\in S^*}\mathcal{L}_{\boldsymbol{\eta}}.    
\end{align}
Indeed, given $(\boldsymbol{\xi},\tau,\sigma)\in \Gamma^3$ such that $\boldsymbol{\xi}^2=t^2$ for some $t\in \mathbb{F}_q^\times$, then $t^{-1}\left(\boldsymbol{\xi},{\tau},{\sigma}\right)\in S_1$, and thus $t^{-1}\left(\boldsymbol{\xi},{\tau},{\sigma}\right)\in S^*_1$ or $-t^{-1}\left(\boldsymbol{\xi},{\tau},{\sigma}\right)\in S^*_1.$ 
On the other hand, if $0\neq \boldsymbol{\xi}^2\neq t^2$ for all $t\in \mathbb{F}_q^\times$, then there exists $t_0\in \mathbb{F}_q^\times$ such that $-t_0^2=\boldsymbol{\xi}^2$, since  $\{t^2:\, t\in \mathbb{F}_q^\times\}$ and $\{-t^2:\, t\in \mathbb{F}_q^\times\}$ are disjoint subsets of $\mathbb{F}_q^\times$ with $(q-1)/2$ elements each; in particular,  $t_0^{-1}\left( \boldsymbol{\xi},{\tau},{\sigma}\right)\in S_2$. Finally, if $\boldsymbol{\xi}^2=0$, then  $\tau^{-1}(\boldsymbol{\xi},\tau,\sigma)\in S^*_3$ or ${\sigma}^{-1}(\boldsymbol{\xi},\tau,\sigma)\in S^*_3$, and \eqref{slicingthecone} follows. 
As a consequence, given $\boldsymbol{\eta}\in \Gamma^3$ and $\boldsymbol{s}\in S^*$ such that $\boldsymbol{\eta}=\alpha \boldsymbol{s}$ for some $\alpha\neq 0$, we have that
\begin{align}\label{structureofpreimagescone}
\{(\boldsymbol{\eta}_1,\boldsymbol{\eta}_2)\in (\Gamma^3)^2:\, \boldsymbol{\eta}_1+\boldsymbol{\eta}_2=\boldsymbol{\eta}\}
=\{(\beta \boldsymbol{s},(\alpha-\beta)\boldsymbol{s}):\,\beta\in \mathbb{F}_q^\times\setminus\{\alpha\}\}.
\end{align}
Indeed, the right-hand side of \eqref{structureofpreimagescone} contains $q-2$ elements of the left-hand side. That these are all follows from \eqref{eq_sizeSigmaGamma}.\\

{\it Step 2:  Mass transport.}
The decomposition $\mathbb F_q^4=(\mathbb F_q^4\setminus\Gamma_0^3)\cup\Gamma^3\cup\{\bf 0\}$ and two applications of Cauchy--Schwarz together with Proposition \ref{propconiconvolutionfinitefields} lead to 
\begin{align}\label{CSgenericpointscone}
\begin{split}
    \sum_{\boldsymbol{\eta}\in \mathbb{F}_q^4}\left| \sum_{\substack{\boldsymbol{\eta}_1,\boldsymbol{\eta}_2\in \Gamma^3\\ \boldsymbol{\eta}_1+\boldsymbol{\eta}_2=\boldsymbol{\eta}}}f(\boldsymbol{\eta}_1)f(\boldsymbol{\eta}_2)\right|^2
\le q(q-1)\sum_{\boldsymbol{\eta}\in \mathbb F_q^4\setminus{\Gamma}_0^3}\sum_{\substack{\boldsymbol{\eta}_1,\boldsymbol{\eta}_2\in \Gamma^3\\ \boldsymbol{\eta}_1+\boldsymbol{\eta}_2=\boldsymbol{\eta}}}|f(\boldsymbol{\eta}_1)f(\boldsymbol{\eta}_2)|^2\\
+(q-2)\sum_{\boldsymbol{\eta}\in \Gamma^3}\sum_{\substack{\boldsymbol{\eta}_1,\boldsymbol{\eta}_2\in \Gamma^3\\ \boldsymbol{\eta}_1+\boldsymbol{\eta}_2=\boldsymbol{\eta}}}|f(\boldsymbol{\eta}_1)f(\boldsymbol{\eta}_2)|^2+\left|\sum_{\boldsymbol{\eta}\in {\Gamma}^3}f(\boldsymbol{\eta})f(-\boldsymbol{\eta})\right|^2,
\end{split}
\end{align}
with equality if $f$ is constant.
Interchanging the order of summation as in \eqref{eq_constant1}, we have $$\sum_{\boldsymbol{\eta}\in \mathbb{F}_q^4}\sum_{\substack{\boldsymbol{\eta}_1,\boldsymbol{\eta}_2\in \Gamma^3\\ \boldsymbol{\eta}_1+\boldsymbol{\eta}_2=\boldsymbol{\eta}}}|f(\boldsymbol{\eta}_1)f(\boldsymbol{\eta}_2)|^2=\left(\sum_{\Gamma^3}|f|^2\right)^2,$$
and therefore the right-hand side of \eqref{CSgenericpointscone} equals
\begin{align}\label{conefirstgrouping}
\begin{split}
q(q-1)\left(\sum_{\Gamma^3}|f|^2\right)^2-((q-1)q-(q-2))\sum_{\boldsymbol{\eta}\in \Gamma^3}\sum_{\substack{\boldsymbol{\eta}_1,\boldsymbol{\eta}_2\in \Gamma^3\\ \boldsymbol{\eta}_1+\boldsymbol{\eta}_2=\boldsymbol{\eta}}}|f(\boldsymbol{\eta}_1)f(\boldsymbol{\eta}_2)|^2\\
+\left|\sum_{\boldsymbol{\eta}\in {\Gamma}^3}f(\boldsymbol{\eta})f(-\boldsymbol{\eta})\right|^2-q(q-1)\sum_{\boldsymbol{\eta}\in {\Gamma}^3}|f(\boldsymbol{\eta})f(-\boldsymbol{\eta})|^2.
\end{split}
\end{align}
We proceed to analyze the cone slices coming from the second summand in \eqref{conefirstgrouping}, and the antipodal pairs from the third and fourth summands in \eqref{conefirstgrouping}.\\

{\it Step 3: Cone slices.}
Interchanging the order of summation, we have
\begin{equation}\label{conefubinicone}
\sum_{\boldsymbol{\eta}\in \Gamma^3}\sum_{\substack{\boldsymbol{\eta}_1,\boldsymbol{\eta}_2\in \Gamma^3\\ \boldsymbol{\eta}_1+\boldsymbol{\eta}_2=\boldsymbol{\eta}}}|f(\boldsymbol{\eta}_1)f(\boldsymbol{\eta}_2)|^2=\sum_{\boldsymbol{\eta}_1,\boldsymbol{\eta}_2\in \Gamma^3}|f(\boldsymbol{\eta}_1)f(\boldsymbol{\eta}_2)|^2{\bf 1}(\boldsymbol{\eta}_1+\boldsymbol{\eta}_2\in \Gamma^3).  
\end{equation}
In light of \eqref{structureofpreimagescone}, it holds that $\boldsymbol{\eta}_1+\boldsymbol{\eta}_2\in \Gamma^3$ if and only if there exist $\boldsymbol{s}\in S^*$ and $\beta_1,\beta_2\in \mathbb{F}_q^\times$, such that $\beta_1\boldsymbol{s}=\boldsymbol{\eta}_1$ and $\beta_2 \boldsymbol{s}=\boldsymbol{\eta}_2$ and $\beta_1\neq -\beta_2$. Therefore \eqref{conefubinicone} boils down to
\begin{align}\begin{split}\label{conerearrangeincone}
\sum_{\boldsymbol{s}\in S^*}\sum_{\substack{\beta_1,\beta_2\in \mathbb{F}_q^\times\\ \beta_1\neq -\beta_2}}|f(\beta_1\boldsymbol{s})f(\beta_2\boldsymbol{s})|^2&=\sum_{\boldsymbol{s}\in S^*}\sum_{\beta_1,\beta_2\in \mathbb{F}_q^\times}|f(\beta_1\boldsymbol{s})f(\beta_2\boldsymbol{s})|^2-\sum_{\boldsymbol{\eta}\in \Gamma^3}|f(\boldsymbol{\eta})f(-\boldsymbol{\eta})|^2\\
&=\sum_{\boldsymbol{s}\in S^*}\left(\sum_{\beta\in \mathbb{F}_q^\times}|f(\beta\boldsymbol{s})|^2\right)^2-\sum_{\boldsymbol{\eta}\in \Gamma^3}|f(\boldsymbol{\eta})f(-\boldsymbol{\eta})|^2.
\end{split}
\end{align}
Since $|S^*|=(q-1)^{-1}|\Gamma^3|=q^2+1$,  a further application of Cauchy--Schwarz yields 
\begin{equation}\label{q^2+1cone}
\sum_{\boldsymbol{s}\in S^*}\left(\sum_{\beta\in \mathbb{F}_q^\times}|f(\beta \boldsymbol{s})|^2\right)^2\ge \frac{1}{q^2+1}\left(\sum_{\boldsymbol{s}\in S^*}\sum_{\beta\in \mathbb{F}_q^\times}|f(\beta \boldsymbol{s})|^2\right)^2=\frac{1}{q^2+1}\left(\sum_{\Gamma^3}|f|^2\right)^2,
\end{equation}
where we used \eqref{slicingthecone} in the last identity. 
Equality holds  in \eqref{q^2+1cone} if $f$ is constant. \\

{\it Step 4: Antipodal pairs.}
It remains to analyze the last two summands in \eqref{conefirstgrouping} along with the additional term coming from the antipodal pairs in \eqref{conerearrangeincone}.
In light of Lemma \ref{numberofpointsonthecone}, these can be bounded by Cauchy--Schwarz as follows:
\begin{align}\begin{split}\label{oppsoitepoints}
\left|\sum_{\boldsymbol{\eta}\in {\Gamma}^3}f(\boldsymbol{\eta})f(-\boldsymbol{\eta})\right|^2-\frac{(q-2)}{q(q^2-q+1)-1}\left |\sum_{\boldsymbol{\eta}\in \Gamma^3}f(\boldsymbol{\eta})f(-\boldsymbol{\eta})\right|^2
\le \frac{q^3-q^2+1}{q(q^2-q+1)-1}\left (\sum_{ \Gamma^3}|f|^2\right)^2.
\end{split}
\end{align}
Combining \eqref{CSgenericpointscone}--\eqref{oppsoitepoints},   we obtain the desired 
\eqref{sharpinequalityconic}, with equality if $f$ is constant. 
We proceed to prove that all maximizers of \eqref{sharpinequalityconic} have constant modulus.

\subsection{Maximizers of \eqref{sharpinequalityconic} have constant modulus}
  Let $f_\star:\Gamma^3\to \Co$ be a maximizer of \eqref{sharpinequalityconic}.
 We note that  $g:=|f_\star|$ is also a maximizer of \eqref{sharpinequalityconic}, and aim to show that $g$ is constant.
  In order for equality to hold in \eqref{oppsoitepoints}, the value of $g(\boldsymbol{\eta})g(-\boldsymbol{\eta})$ must not depend on $\boldsymbol{\eta}\in\Gamma^3$. 
Moreover, in order for equality to hold in the second application of Cauchy--Schwarz in  \eqref{CSgenericpointscone}, we must have
\begin{align}\label{sumuniquenesscone}
g(\boldsymbol{\eta}_1)g(\boldsymbol{\eta}_2)=C(\boldsymbol{\eta}_1+\boldsymbol{\eta}_2), \text{ for every }    \boldsymbol{\eta}_1,\boldsymbol{\eta}_2\in \mathcal{L}_{\boldsymbol{s}} \text{ with } \boldsymbol{\eta}_1\neq \boldsymbol{\eta}_2,
\end{align}
and, in light of \eqref{structureofpreimagescone}, that $g(\alpha \boldsymbol{s})g(\beta \boldsymbol{s})=g(\tfrac{\alpha+\beta}2\boldsymbol{s})^2$
whenever $\alpha\neq -\beta$.
Interestingly, the analysis splits into two cases, depending on whether $q$ equals 3 or not. \\

{\bf Case 1: $q> 3$.}
 Given $\boldsymbol{s}\in S^*$, assume that the function $\alpha\mapsto g(\alpha \boldsymbol{s})$ is maximized for $\alpha=\alpha_0\neq 0$. 
 Given any nonzero $\beta \neq 2\alpha_0,$ we then have 
 $$g(\beta \boldsymbol{s})g(\alpha_0 \boldsymbol{s})\ge g(\beta \boldsymbol{s})g((2\alpha_0-\beta)\boldsymbol{s})=g(\alpha_0 \boldsymbol{s})^2,$$ and therefore $g(\beta \boldsymbol{s})=g(\alpha_0 \boldsymbol{s})$. 
Similarly, we   conclude  that $g(2\alpha_0 \boldsymbol{s})=g(\alpha_0 \boldsymbol{s})$. Indeed, let $\beta\in\mathbb F_q^\times$ be such that $\beta\neq \alpha_0$ and $\beta\neq 2\alpha_0$ (this requires $q>3$). We have already seen that $g(\beta \boldsymbol{s})=g(\alpha_0 \boldsymbol{s})$ is maximal, and so $$g(2\alpha_0 \boldsymbol{s})g(\beta \boldsymbol{s})\ge g(2\alpha_0 \boldsymbol{s})g((2\beta -2\alpha_0)\boldsymbol{s})=g(\beta \boldsymbol{s})^2.$$
It follows that $g(\alpha \boldsymbol{s})=C(\boldsymbol{s})$, for every $\alpha\in \mathbb{F}_q^\times$.
This implies that $g$ is constant, since  equality in \eqref{q^2+1cone} forces  $\sum_{\alpha\in \mathbb{F}_q^\times}g(\alpha \boldsymbol{s})^2$
to be constant.\\

The case $q=3$ is more involved, and combinatorially more interesting.\\

{\bf Case 2: $q=3$.}  
By  Lemma \ref{numberofpointsonthecone} and \eqref{slicingthecone}, the cone $\Gamma^3\subset\mathbb F_3^4$ has twenty points and  equals the disjoint union of ten lines, each with two antipodal points. As before, there exists $c\geq 0$ such that \begin{equation}\label{eq_1srconstraint}
    g(\boldsymbol{\eta})g(-\boldsymbol{\eta})=c, \text{ for every } \boldsymbol{\eta}\in \Gamma^3.
\end{equation}
On each of the ten  lines $\{\mathcal{L}_{\boldsymbol{s}}:\, \boldsymbol{s}\in S^*\}$ that make up $\Gamma^3$, take  $\boldsymbol{s}'\in\mathcal L_{\boldsymbol{s}}$ such that $g(\boldsymbol{s}')\ge \sqrt{c}$, and denote the set of such $\boldsymbol{s}'$ by $S'$. 
In order for equality to hold in \eqref{q^2+1cone}, we  need $g(\boldsymbol{s}')^2+g(-\boldsymbol{s}')^2=C$ to be constant; since $g$ is nonzero, it follows that $g(\boldsymbol{s}')>0$, for all $\boldsymbol{s}'\in S'$. Identity \eqref{eq_1srconstraint} then implies 
$$g(\boldsymbol{s}')^2+\frac{c^2}{g(\boldsymbol{s}')^2}=C,\text{ for all }\boldsymbol{s}'\in S'.$$ 
The function $x\mapsto x^2+c^2/x^2$ is strictly increasing if $x\ge \sqrt{c}$, and so $g$ is constant on $S'$. Writing $g\mid_{S'}=:\rho\ge \sqrt{c}$,
it  suffices to show that $\rho=\sqrt{c}$. 
We will suppose $\rho>\sqrt{c}$, and establish sufficiently many structural constraints on the set $S'$ to reach a contradiction.

To implement this strategy, let $\boldsymbol{s}'_0:=(\boldsymbol{0},a,0)\in S'$, where ${a}\in\{1,2\}$, and  $\pi:\Gamma^3\to \mathbb{F}_3$ denote the projection onto the last coordinate, $(\boldsymbol{\xi},\tau,\sigma)\mapsto \sigma$.
Given $i\in\{1,2\}$, write $S'_i:=\{\boldsymbol{s}'\in S':\,\pi(\boldsymbol{s}')=i\}.$  
In order to get equality in \eqref{CSgenericpointscone}, we need that  $g(\boldsymbol{s}'_1)g(\boldsymbol{s}'_2)=C(\boldsymbol{s}'_1+\boldsymbol{s}'_2)$ for all $\boldsymbol{s}_1',\boldsymbol{s}_2'\in S'$. Moreover, given $\boldsymbol{\eta} \in \mathbb F_3^4\setminus\Gamma_0^3$, the set of unordered pairs $A(\boldsymbol{\eta}):=\{\{\boldsymbol{s}_1,\boldsymbol{s}_2\}:\, \boldsymbol{s}_1,\boldsymbol{s}_2\in \Gamma^3,\, \boldsymbol{s}_1+\boldsymbol{s}_2=\boldsymbol{\eta}\}$ has exactly three elements by \eqref{eq_sizeSigmaGamma}. Assume that $S'_1$ is nonempty (the case of  nonempty $S_2'$ is dealt with in a similar way). Given $\boldsymbol{s}_1'\in S'_1$, we thus have  $|A(\boldsymbol{s}'_0+\boldsymbol{s}'_1)|=3$. Consider the other two pairs $\{\boldsymbol{s}'_2,\boldsymbol{s}'_3\},\{\boldsymbol{s}'_4,\boldsymbol{s}'_5\}\in A(\boldsymbol{s}'_0+\boldsymbol{s}'_1)$.
Since $$\rho^2=g(\boldsymbol{s}'_0)g(\boldsymbol{s}'_1)=g(\boldsymbol{s}'_2)g(\boldsymbol{s}'_3)=g(\boldsymbol{s}'_4)g(\boldsymbol{s}'_5),$$
it follows from $\rho>\sqrt c$ that $\rho=g(\boldsymbol{s}'_2)=g(\boldsymbol{s}'_3)=g(\boldsymbol{s}'_4)=g(\boldsymbol{s}'_5)$, and thus $\boldsymbol{s}'_2,\boldsymbol{s}'_3, \boldsymbol{s}'_4,\boldsymbol{s}'_5\in S'.$
Crucially,  we observe that $\boldsymbol{s}'_2,\boldsymbol{s}'_3, \boldsymbol{s}'_4,\boldsymbol{s}'_5\in S_2'$ since $\pi(\boldsymbol{s}'_2+\boldsymbol{s}'_3)=\pi(\boldsymbol{s}'_4+\boldsymbol{s}'_5)=1$.
On the other hand,  if $A(\boldsymbol{s}'_0+\boldsymbol{s}_2')=\{\{\boldsymbol{s}'_0,\boldsymbol{s}'_2\},\{\boldsymbol{s}'_6,\boldsymbol{s}'_7\},\{\boldsymbol{s}'_8,\boldsymbol{s}'_9\}\}$, then $\pi(\boldsymbol{s}'_0+\boldsymbol{s}'_2)=2,$ and so we conclude in a similar way that $\boldsymbol{s}'_6,\boldsymbol{s}'_7,\boldsymbol{s}'_8,\boldsymbol{s}'_9\in S'_1$. 
Further note that $\boldsymbol{s}'_1\notin \{\boldsymbol{s}'_6,\boldsymbol{s}'_7,\boldsymbol{s}'_8,\boldsymbol{s}'_9\},$ for otherwise $\boldsymbol{s}'_1+\boldsymbol{s}'_i=\boldsymbol{s}'_0+\boldsymbol{s}'_2$ for some $i\in\{6,7,8,9\}$; from  $\boldsymbol{s}'_0+\boldsymbol{s}'_1=\boldsymbol{s}'_2+\boldsymbol{s}'_3$, we would then obtain $2\boldsymbol{s}'_0=\boldsymbol{s}'_i+\boldsymbol{s}'_3$, which is absurd  since
$2\boldsymbol{s}'_0, \boldsymbol{s}'_i, \boldsymbol{s}'_3$ belong to distinct lines (recall \eqref{structureofpreimagescone}). Thus $S'=\{\boldsymbol{s}_0'\}\cup S_1'\cup S_2'$, where $S_1'=\{\boldsymbol{s}'_1,\boldsymbol{s}'_6,\boldsymbol{s}'_7,\boldsymbol{s}'_8,\boldsymbol{s}'_9\}$ and $S_2':=\{\boldsymbol{s}'_2,\boldsymbol{s}'_3,\boldsymbol{s}'_4,\boldsymbol{s}'_5\}$ are disjoint, and disjoint from $\{\boldsymbol{s}_0'\}$. It follows that the set 
$$\bigcup_{\boldsymbol{s}'_i\in S'_1}(A(\boldsymbol{s}'_0+\boldsymbol{s}'_i)\setminus \{\boldsymbol{s}'_0+\boldsymbol{s}'_i\})$$ contains ten distinct  pairs, and thus cannot be a subset of the six-element set $\{\{\boldsymbol{u},\boldsymbol{v}\}:\, \boldsymbol{u},\boldsymbol{v}\in S_2', \boldsymbol{u}\neq \boldsymbol{v}\}$. This contradiction results from assuming $\rho>\sqrt c$; thus  $\rho=\sqrt{c}$, and $g=|f_\star|$ is constant. This concludes the proof of Theorem \ref{thm5}.

\section{Proof of Theorem \ref{thm6}}\label{sec_thm6}
In this section, we prove Theorem \ref{thm6}.
Starting with the case of the cone $\Upsilon_0^3$ equipped with normalized counting measure $\nu=\nu_\Upsilon$,
we   test the functional 
\begin{equation}\label{eq_Phi}
    \Phi_p(\varepsilon):=\frac{\sum_{\boldsymbol{x}\in \mathbb F_p^4} |(f_\varepsilon\nu)^\vee(\boldsymbol{x})|^4}{\left(\frac1{|\Upsilon_0^3|}\sum_{\boldsymbol{\xi}\in \Upsilon_0^3}|f_\varepsilon(\boldsymbol{\xi})|^2\right)^2}
\end{equation}
against the function $f_\varepsilon:={\mathbf 1}_{\Upsilon_0^3}+\varepsilon\delta_0$,
for small values of $\varepsilon>0$.
The denominator in \eqref{eq_Phi} is straightforward to compute:
\[\frac1{|\Upsilon_0^3|}\sum_{\boldsymbol{\xi}\in \Upsilon_0^3}|f_\varepsilon(\boldsymbol{\xi})|^2 =\frac{(|\Upsilon_0^3|-1)\times 1^2+1\times(1+\varepsilon)^2}{|\Upsilon_0^3|}=1-\frac1{|\Upsilon_0^3|}+\frac{(1+\varepsilon)^2}{|\Upsilon_0^3|}.\]
As for the numerator in \eqref{eq_Phi}, note that \eqref{eq_FEOdef} implies
\[(\delta_0 \nu)^\vee(\boldsymbol{x})=\frac{1}{|\Upsilon_0^3|}\sum_{\boldsymbol{\xi}=0} e(\boldsymbol{x}\cdot\boldsymbol{\xi})=\frac1{|\Upsilon_0^3|}, \text{ for every }\boldsymbol{x}\in\mathbb F_p^4,\] 
whereas
$({\bf 1}_{\Upsilon_0^3}\nu)^\vee=\nu_\Upsilon^\vee$
has been computed in \eqref{eq_nuhatUps}.
Together with $|\Upsilon_0^3|=p^3+p^2-p$ (Proposition \ref{prop_ConicConv}), this leads to
\[\Phi_p(\varepsilon)=\frac{1\times\left(\frac{\varepsilon+p^3+p(p-1)}{|\Upsilon_0^3|}\right)^4+(|\Upsilon_0^3|-1)\times\left(\frac{\varepsilon+p(p-1)}{|\Upsilon_0^3|}\right)^4+(p^4-|\Upsilon_0^3|)\times\left(\frac{\varepsilon-p}{|\Upsilon_0^3|}\right)^4}{\left(1-\frac1{|\Upsilon_0^3|}+\frac{(1+\varepsilon)^2}{|\Upsilon_0^3|}\right)^2},\]
which can be simplified to 
$\Phi_p(\varepsilon)={A_p(\varepsilon)}/{B_p(\varepsilon)}$,
where
\[ A_p(\varepsilon):=2 p^5 + p^6 - 7 p^7 - p^8 + 5 p^9 + p^{10} +(- 8 p^5  + 4 p^6  + 
 8 p^7 )\varepsilon +(- 6 p^3 + 6 p^4 + 6 p^5 ) \varepsilon^2 + 4 p^2 \varepsilon^3 + p^2 \varepsilon^4;\]
\[     B_p(\varepsilon):=(p^2+p-1)^2 (p^3 +p^2-p + \varepsilon (2 + \varepsilon))^2.\]
Consequently,
 \[\Phi_p'(0)=\frac{4 p^2(p-2)  (p^2-1)^2}{(p^2+p-1)^5},\]
 which is a strictly positive quantity for every prime $p>2$.
 
To handle the cone $\Gamma_0^3$ equipped with normalized counting measure $\nu=\nu_\Gamma$, consider the functional
\begin{equation}\label{eq_Psi}
    \Psi_p(\varepsilon):=\frac{\sum_{\boldsymbol{x}\in \mathbb F_p^4} |(f_\varepsilon\nu)^\vee(\boldsymbol{x})|^4}{\left(\frac1{|\Gamma_0^3|}\sum_{\boldsymbol{\xi}\in \Gamma_0^3}|f_\varepsilon(\boldsymbol{\xi})|^2\right)^2}.
\end{equation}
If $p\equiv 1(\text{mod } 4)$, then the proof is the same as the one for $\Upsilon_0^3$ above; recall our discussion in the course of the proof of Proposition \ref{prop_ConicConv}.
 If $p\equiv 3(\text{mod } 4)$, then $({\bf 1}_{\Gamma_0^3}\nu)^\vee=\nu_\Gamma^\vee$
is given by \eqref{eq_nuGammaHat}, which together with  $|\Gamma_0^3|=p^3-p^2+p$ (Lemma \ref{numberofpointsonthecone}) leads to
\[\Psi_p(\varepsilon)=\frac{1\times\left(\frac{\varepsilon+p^3-p(p-1)}{|\Gamma_0^3|}\right)^4+(|\Gamma_0^3|-1)\times\left(\frac{\varepsilon-p(p-1)}{|\Gamma_0^3|}\right)^4+(p^4-|\Gamma_0^3|)\times\left(\frac{\varepsilon+p}{|\Gamma_0^3|}\right)^4}{\left(1-\frac1{|\Gamma_0^3|}+\frac{(1+\varepsilon)^2}{|\Gamma_0^3|}\right)^2}.\]
This can be simplified to $\Psi_p(\varepsilon)={C_p(\varepsilon)}/{D_p(\varepsilon)}$,
where
 \[C_p(\varepsilon):=-2 p^5 + 5 p^6 - 5 p^7 + 5 p^8 - 3 p^9 + p^{10} + 4 p^6 \varepsilon + (6 p^3  - 
 6 p^4  + 6 p^5 ) \varepsilon^2 + 4 p^2 \varepsilon^3 + p^2 \varepsilon^4;\]
 \[D_p(\varepsilon):=(p^2-p+1)^2 (p^3 -p^2 +p + \varepsilon (2 + \varepsilon))^2.\]
 It follows that
\[\Psi_p'(0)=-\frac{4  p^2 (p-2)( p-1)^2(p^2+1)}{(p^2-p+1)^5},\]
 which is a  strictly negative quantity for every prime $p\equiv 3(\textup{mod } 4)$.

As a consequence, for $S\in\{\Gamma_0^3,\Upsilon_0^3\}$ and any prime $p$,  the  function ${\bf 1}_S$ is  not a critical point of the functionals $\Psi_p, \Phi_p$, respectively, and therefore not a local or global maximizer for the $L^2(S,\textup d \nu)\to L^4 (\mathbb F_p^4,\textup d \boldsymbol{x})$ extension inequality from $S\subset \mathbb F_p^4$. This concludes the proof of Theorem \ref{thm6}.

\section*{Acknowledgments}
The authors are partially supported by FCT/Portugal through CAMGSD, IST-ID, projects UIDB/ 04459/2020 and UIDP/04459/2020 and
3 by IST Santander Start Up Funds. They are grateful to Asem Abdelraouf and Emanuel Carneiro for inspiring discussions regarding the present work.


\begin{thebibliography}{03}

\bibitem{Ar24} B. Arsovski,
\newblock {\it The $p$-adic Kakeya conjecture.}
\newblock J. Amer. Math. Soc. {\bf 37} (2024), no.~1, 69--80.

\bibitem{BBCH09} J. Bennett, N. Bez, A. Carbery, D. Hundertmark,
\newblock {\it Heat-flow monotonicity of Strichartz norms.}
\newblock Anal. PDE {\bf 2} (2009), no.~2, 147--158.

\bibitem{BCT06} J. Bennett, A. Carbery, T. Tao,
\newblock {\it On the multilinear restriction and Kakeya conjectures.}
\newblock Acta Math. {\bf 196} (2006), no.~2, 261--302.

\bibitem{BG11} J. Bourgain, L. Guth, 
\newblock {\it Bounds on oscillatory integral operators based on multilinear estimates.}
\newblock Geom. Funct. Anal. {\bf 21} (2011), no.~6, 1239--1295.

\bibitem{Ca06} A. Carbery, 
\newblock Harmonic analysis on vector spaces over finite fields.  Lecture notes.
\url{https://www.maths.ed.ac.uk/~carbery/analysis/notes/fflpublic.pdf}, 2006.

\bibitem{Ca09} E. Carneiro, 
\newblock {\it A sharp inequality for the Strichartz norm.} 
\newblock Int. Math. Res. Not. IMRN 2009, no.~16, 3127--3145.

\bibitem{COS22} E. Carneiro, L. Oliveira, M. Sousa,  
\newblock {\it  Gaussians never extremize Strichartz inequalities for hyperbolic paraboloids.}
\newblock Proc. Amer. Math. Soc. {\bf 150} (2022), no.~8, 3395--3403.

\bibitem{Dh24} M. Dhar,
\newblock {\it The Kakeya set conjecture over $\Z/N\Z$ for general $N$.}
\newblock Adv. Comb. (2024), Paper No.~2, 26 pp.

\bibitem{DD21} M. Dhar, Z. Dvir, 
\newblock {\it Proof of the Kakeya set conjecture over rings of integers modulo square-free $N$.}
\newblock Comb. Theory {\bf 1} (2021), Paper No.~4, 21 pp.

\bibitem{DMPS18} B. Dodson, J. Marzuola, B. Pausader, D. Spirn, 
\newblock {\it The profile decomposition for the hyperbolic Schrödinger equation.}
\newblock Illinois J. Math. {\bf 62} (2018), no.~1--4, 293--320.
\newblock Erratum: Illinois J. Math. {\bf 65} (2021), no.~1, 259--260.

\bibitem{Dv09} Z. Dvir, 
\newblock {\it On the size of Kakeya sets in finite fields.}
\newblock J. Amer. Math. Soc. {\bf 22} (2009), no.~4, 1093--1097.

\bibitem{EOT10} J. Ellenberg, R. Oberlin, T. Tao, 
\newblock {\it The Kakeya set and maximal conjectures for algebraic varieties over finite fields.}
\newblock Mathematika {\bf 56} (2010), no.~1, 1--25.

\bibitem{Fo07} D. Foschi,
\newblock {\it Maximizers for the Strichartz inequality.}
\newblock J. Eur. Math. Soc. (JEMS) {\bf 9} (2007), no.~4, 739--
774. 


\bibitem{Go19} F. Gon\c{c}alves,
\newblock {\it Orthogonal polynomials and sharp estimates for the Schrödinger equation.} 
\newblock Int. Math. Res. Not. IMRN 2019, no.~8, 2356--2383.

\bibitem{Gr13} B. Green,
\newblock Restriction and Kakeya phenomena. Cambridge Part III course notes. \url{http://people.maths.ox.ac.uk/greenbj/papers/rkp.pdf}, 2013

\bibitem{Gu16} L. Guth,
\newblock {\it A restriction estimate using polynomial partitioning.}
\newblock J. Amer. Math. Soc. {\bf 29} (2016), no.~2, 371--413.

\bibitem{Gu18} L. Guth,
\newblock {\it Restriction estimates using polynomial partitioning II.}
\newblock Acta Math. {\bf 221} (2018), no.~1, 81--142.

\bibitem{HW18} J. Hickman, J. Wright,
\newblock {\it The Fourier restriction and Kakeya problems over rings of integers modulo $N$.}
\newblock Discrete Anal. (2018), Paper No.~11, 54 pp. 

\bibitem{IK10} A. Iosevich, D. Koh,
\newblock {\it Extension theorems for spheres in the finite field setting.}
\newblock Forum Math. {\bf 22} (2010), no.~3, 457--483.

\bibitem{KLP21} D. Koh, S. Lee, T. Pham,
\newblock {\it On the finite field cone restriction conjecture in four dimensions and applications in incidence geometry.}
\newblock Int. Math. Res. Not. IMRN (2022), no.~21, 17079--17111.

\bibitem{KY17} D. Koh, S. Yeom,
\newblock {\it Restriction of averaging operators to algebraic varieties over finite fields.}
\newblock Taiwanese J. Math. {\bf 21} (2017), no.~1, 211--229. 


\bibitem{HZ06} D. Hundertmark, V. Zharnitsky, 
\newblock {\it On sharp Strichartz inequalities in low dimensions.} 
\newblock Int. Math. Res. Not. IMRN (2006), Art. ID 34080, 1--18.

\bibitem{Le19} M. Lewko, 
\newblock {\it Finite field restriction estimates based on Kakeya maximal operator estimates.}
\newblock J. Eur. Math. Soc. (JEMS) {\bf 21} (2019), no.~12, 3649--3707.

\bibitem{LN97} R. Lidl, H. Niederreiter,
\newblock Finite fields.
With a foreword by P. M. Cohn. 
\newblock Second edition.
Encyclopedia Math. Appl., {\bf 20}.
Cambridge University Press, Cambridge, 1997. 

\bibitem{MT04} G. Mockenhaupt, T. Tao, 
\newblock {\it Restriction and Kakeya phenomena for finite fields.} 
\newblock Duke Math. J. {\bf 121} (2004), no.~1, 35--74. 

\bibitem{NOSST23}
G. Negro, D. Oliveira e Silva, B. Stovall, J. Tautges,
\newblock{\it Exponentials rarely maximize Fourier extension inequalities for cones.}
\newblock arXiv:2302.00356.

\bibitem{NOST23}
G. Negro, D. Oliveira e Silva, C. Thiele,
\newblock{\it  When does $e^{-|\tau|}$ maximize Fourier extension for a conic section?}
\newblock Harmonic analysis and convexity, 391--426.
Adv.\@ Anal.\@ Geom.,  9,
De Gruyter, Berlin, 2023.

\bibitem{Sa23} A. Salvatore, 
\newblock{\it The Kakeya conjecture on local fields of positive characteristic.}
\newblock Mathematika {\bf 69} (2023), no.~1, 1--16.

\bibitem{St93} E. M. Stein,
\newblock Harmonic Analysis: Real-Variable Methods, Orthogonality, and Oscillatory Integrals.
\newblock Princeton Univ. Press, Princeton, NJ, 1993.

\bibitem{Str77} R. Strichartz, 
\newblock {\it Restrictions of Fourier transforms to quadratic surfaces and decay of solutions of wave equations.}
\newblock Duke Math. J. {\bf 44} (1977), no.~3, 705--714.

\bibitem{Wo99} T. Wolff, 
\newblock {\it Recent work connected with the Kakeya problem.}
\newblock Prospects in mathematics (Princeton, NJ, 1996), 129--162.
American Mathematical Society, Providence, RI, 1999

\end{thebibliography}
\end{document}